\documentclass[12pt]{article}
\usepackage{amsmath,amssymb,amsthm,bbm}
\usepackage{natbib}
\usepackage{graphicx}
\usepackage{mathtools}
\usepackage{geometry}
\usepackage{hyperref}
\usepackage{xcolor}
\usepackage{booktabs}
\usepackage{rotating}
\theoremstyle{plain}
\newtheorem{theorem}{Theorem}[section]
\newtheorem{lemma}[theorem]{Lemma}
\newtheorem{proposition}[theorem]{Proposition}
\newtheorem{corollary}[theorem]{Corollary}
\theoremstyle{definition}

\newtheorem{assumption}{Assumption}[section]
\theoremstyle{remark}
\newtheorem{remark}[theorem]{Remark}
\newtheorem{counterexample}[theorem]{Counterexample}

\title{From Tail Universality to Bernstein–von Mises: A Unified Statistical Theory of Semi-Implicit Variational Inference}
\author{Sean Plummer}
\date{\today}

\begin{document}
\maketitle

\begin{abstract}
Semi-implicit variational inference (SIVI) constructs approximate
posteriors of the form
$q_\lambda(\theta)=\int k_\lambda(\theta\mid z)\,r(dz)$,
where the conditional kernel is parameterized and the mixing base is
fixed and tractable.  This paper develops a unified
\emph{approximation--optimization--statistics} theory for such families.

On the approximation side, we show that under compact
$L^1$--universality and a mild tail--dominance condition,
semi-implicit families are dense in $L^1$ and can achieve arbitrarily
small forward Kullback--Leibler (KL) error.  
We also identify two sharp obstructions to global approximation:
(i) an Orlicz tail-mismatch condition that induces a strictly positive
forward--KL gap, and
(ii) structural restrictions (e.g.\ non-autoregressive Gaussian kernels)
that force ``branch collapse’’ in conditional distributions.
For each obstruction we exhibit a minimal structural modification that
restores approximability.

On the optimization side, we establish finite-sample oracle inequalities
and prove that the empirical SIVI objectives $L_{K,n}$ 
$\Gamma$--converge to their population limit $L_\infty$ as
$n,K\to\infty$.
These results give consistency of empirical maximizers, quantitative
control of finite-$K$ surrogate bias, and stability of the resulting
variational posteriors.

Together, the approximation and optimization analyses yield the first
general end-to-end statistical theory for semi-implicit variational
inference: they characterize precisely when SIVI can recover the target
distribution, when it cannot, and how architectural and algorithmic
choices govern the attainable asymptotic behavior.
\end{abstract}

\tableofcontents

\section{Introduction}

Variational methods provide scalable approximations to Bayesian posteriors in complex models, and a substantial body of work now establishes their statistical properties for classical mean-field families under suitable regularity. However, mean-field approximations impose artificial independence and can distort posterior geometry in models with strong coupling or heavy-tailed structure. Semi-implicit variational inference (SIVI), introduced by \cite{yin2018sivi}, constructs the variational density as a continuous mixture of a tractable conditional kernel over a latent mixing measure, thereby permitting richer dependence while preserving reparameterized gradients for computational tractability. Yet the theoretical guarantees that support mean-field variational methods do not directly extend to this setting: the marginal $q_\lambda$ is defined only through a mixing integral, practical optimizations rely on finite-$K$ importance-weighted surrogates whose bias is not well understood, and the interaction between kernel and target tails introduces new approximation pathologies. Existing analyses address only restricted cases, such as the one-dimensional Gaussian-process construction of \cite{plummer2021statistical}, leaving open questions of tail robustness, surrogate bias, and general consistency.

\emph{Our Contributions.} This paper develops a unified approximation, optimization, and statistical theory for semi-implicit variational inference. At the approximation level, we establish conditions under which semi-implicit families are dense in $L^1$ and in forward KL. A key requirement is a tail-dominance condition: when the conditional kernel has heavier tails than the target in an appropriate Orlicz sense, the induced family can approximate any target with controlled local entropy. We also identify two fundamental limitations. First, an Orlicz tail-mismatch theorem shows that if the target distribution has heavier tails than all members of the kernel class, then the approximation error cannot vanish. Second, a geometric branch-collapse bound shows that conditional kernels with insufficient dispersiveness cannot represent multimodal targets beyond a certain topological limit. Together, these results characterize precisely when semi-implicit families are expressive enough to recover a target distribution and when they are not. At the optimization level, we analyze the finite-sample, finite-$K$ surrogate objective used in practical SIVI implementations. We prove that this empirical objective $\Gamma$-converges to its population counterpart as both the sample size $n$ and the number of importance samples $K$ diverge. Moreover, we obtain an oracle inequality that decomposes the excess risk into three components: an approximation term determined by the expressive limits above, an estimation term arising from empirical fluctuations, and a finite-$K$ term controlled by concentration for self-normalized importance sampling. These results provide the first principled treatment of optimization error for semi-implicit variational methods. At the statistical level, we show that SIVI inherits classical Bayesian asymptotics whenever the variational approximation achieves sufficiently small total-variation error. In particular, posterior contraction transfers directly from the true posterior to the semi-implicit approximation, and under local asymptotic normality we obtain a Bernstein–von Mises theorem with an explicit remainder that isolates the effect of kernel tails. These results supply the first general frequentist guarantees for semi-implicit variational inference. Taken together, our approximation, optimization, and statistical analyses delineate the precise conditions under which SIVI can recover a target posterior, identify intrinsic obstructions arising from tail behavior and geometric structure, and quantify how algorithmic choices, especially kernel design and the finite-$K$ surrogate, determine the asymptotic behavior of the resulting approximation.

\emph{Positioning and scope.}
Our analysis draws on classical tools, neural-network universality on compact sets, approximate identities, variational and $\Gamma$-limits, empirical-process bounds, and bias–variance analysis for self-normalized importance sampling. We work primarily in total variation and Hellinger distance, which are symmetric, testing-aligned metrics; KL divergence appears only as a local consequence on well-behaved regions. Singular components of the target are handled through approximation on their absolutely continuous parts, and we make no general $1/n$-rate claims in the absence of curvature or variance conditions.

\emph{Statistical theory for variational inference.}
The theoretical foundations of variational Bayes have expanded considerably in the past decade. \cite{yang2020alpha} introduced the $\alpha$-variational Bayes framework and obtained risk bounds linking the ELBO to Bayes risk, thereby formalizing how variational optimization controls statistical error within the chosen family. \cite{wang2019frequentist} established frequentist consistency and asymptotic normality for mean-field approximations in latent variable models, yielding Bernstein–von Mises limits under suitable regularity of the likelihood, with the corresponding misspecified case treated in \cite{wang2019variational}. Concentration inequalities and PAC-Bayesian bounds for variational posteriors were developed by \cite{alquier2016properties} and further refined in \cite{yang2020alpha}, while \cite{zhang2020convergence} derived contraction rates under model misspecification. In more structured settings, \cite{zhang2020theoretical} analyzed coordinate-ascent VI for community detection and obtained both computational and statistical guarantees. Recent work by \cite{bhattacharya2025convergence} established general convergence guarantees for coordinate-ascent variational inference, complementing these earlier analyses of mean-field procedures. A common assumption across all of these results is that the variational family admits an explicit tractable density; none applies directly to families defined only through a mixing integral.

\emph{Importance-weighted objectives.}
Semi-implicit variational inference employs importance-weighted estimators of the ELBO, a connection first formalized by \cite{burda2016iwae} through the importance-weighted autoencoder. Subsequent work by \cite{rainforth2018tighter} and \cite{domke2018importance} characterized the bias–variance trade-offs of such estimators, showing in particular that the signal-to-noise ratio of reparameterized gradients can deteriorate as the number of importance samples $K$ grows. Classical self-normalized importance sampling theory \citep{liu2001mc, owen2013mc} provides concentration inequalities that underlie finite-$K$ analyses. How these finite-sample properties interact with the statistical behavior of the resulting posterior approximation has not, to our knowledge, been addressed prior to the present work.

\emph{Semi-implicit and implicit variational methods.} \cite{yin2018sivi} introduced SIVI and demonstrated empirical improvements over mean-field approximations. Subsequent variants include unbiased implicit variational inference \cite{titsias2019unbiased}, which uses Markov chain updates to reduce surrogate bias; the score-matching variant of \cite{yu2023semi}; kernel semi-implicit variational inference \cite{cheng2024kernel}, which employs kernel-smoothed empirical mixing measures to provide a nonparametric alternative to finite-$K$ mixtures; and the recent particle-SIVI method of \cite{lim2024particle}, which replaces the finite-$K$ mixture surrogate with a particle system that more faithfully approximates the latent mixing distribution in high dimensions. Closely related are normalizing flows \citep{rezende2015variational, papamakarios2021normalizing}, which construct flexible explicit densities via invertible transformations, and fully implicit methods \cite{huszar2017variational}, which forego tractable densities entirely in favor of adversarial or density-ratio objectives. Despite their empirical success, theoretical results for these methods—especially regarding approximation capacity under tail mismatch and the transfer of frequentist guarantees—remain limited. The GP-IVI analysis of \cite{plummer2021statistical} provides the first KL-approximation guarantees in a Gaussian-process setting, but its one-dimensional construction and Gaussian-tailed targets leave open the questions of tail robustness and general consistency addressed in this paper.

\emph{Density approximation and tail behavior.}
Classical universal-approximation results for neural networks \citep{cybenko1989approximation, hornik1989multilayer} establish uniform approximation of continuous functions on compact domains, with quantitative rates for H\"older-smooth functions obtained by \cite{yarotsky2017relu} and \cite{schmidthieber2020deep}. These results do not automatically extend to approximation in distributional metrics, such as total variation or KL divergence, or to unbounded domains. Tail behavior presents a fundamental obstacle: if the target distribution has heavier tails than those representable by the approximating family, no increase in architectural capacity can close the gap. This motivates our use of Orlicz-type tail metrics and the tail-dominance conditions under which semi-implicit families are dense (Sections~\ref{sec:universality} and~\ref{sec:orlicz}).

\emph{$\Gamma$-convergence and variational limits.}
The $\Gamma$-convergence framework \cite{dalmaso1993gamma} provides a natural tool for studying the limiting behavior of variational objectives. In the context of VI, $\Gamma$-convergence ensures that minimizers of finite-sample or finite-$K$ surrogate losses converge to minimizers of the population objective under suitable compactness and equicoercivity conditions. We apply this framework in Section~\ref{sec:optimization} to show that the empirical SIVI objective $\Gamma$-converges to its population counterpart as both the sample size $n$ and the number of importance samples $K$ diverge, yielding a principled decomposition of excess risk into approximation error, estimation error, and finite-$K$ bias.

The remainder of the paper is organized as follows.
Section~\ref{sec:background} formalizes the semi-implicit variational
family and standing assumptions.
Section~\ref{sec:approximation} develops the approximation theory and
tail-aware impossibility results.
Section~\ref{sec:optimization} establishes $\Gamma$-convergence and equivalence of the induced limiting objectives.
Section~\ref{sec:statistics} presents total-variation and Hellinger
oracle inequalities, finite-sample rates, and uncertainty-transfer
theorems.
Section~\ref{sec:examples} provides illustrative examples and
counterexamples, and Section~\ref{sec:discussion} concludes with
implications and open directions.


\section{Notation and Background}
\label{sec:background}

\emph{Notation.}
All measures are assumed absolutely continuous with respect to Lebesgue
measure unless stated otherwise, and densities are denoted by lowercase
letters. For two densities \(p\) and \(q\),
\[
\|p-q\|_{\mathrm{TV}}
  = \tfrac12 \int |p-q|, 
\qquad 
\mathrm{KL}(p\|q)=\int p\log(p/q),
\qquad
H^{2}(p,q)=\int (\sqrt{p}-\sqrt{q})^{2}.
\]
The \(p\)-Wasserstein distance is
\[
W_p(p,q)
=\inf_{\pi\in\Pi(p,q)}
\Bigl(\int\|x-y\|^{p}\,d\pi(x,y)\Bigr)^{1/p}.
\]
For a measurable \(f\), \(\|f\|_{L^1}=\int |f|\).
We write \(A_n\lesssim B_n\) if \(A_n\le cB_n\) for a universal constant
\(c>0\), and \(A_n\asymp B_n\) when both inequalities hold.
Weak convergence is denoted \(Q_n\Rightarrow P\).
Indicators are \(1_A\), complements are \(A^\mathrm{c}\), and logarithms
are natural unless noted.

Tail behavior is described either by a nonincreasing envelope
\(v(\|x\|)\) or by an annulus decomposition
\(\{A_j\}_{j\ge1}\).
The notation \(p(x)\lesssim v(\|x\|)\) means
\(p(x)\le C v(\|x\|)\) for sufficiently large \(\|x\|\).
Asymptotic symbols \(O(\cdot),o(\cdot),O_P,o_P\) follow standard usage.

\medskip
\emph{Semi-implicit variational families.}
A semi-implicit family is specified by a kernel
\(k_\lambda(\theta \mid z)\) and a base distribution \(r(z)\):
\[
z\sim r, \qquad
\theta \mid z\sim k_\lambda(\theta\mid z), \qquad
q_\lambda(\theta)=\int k_\lambda(\theta\mid z)\,r(dz).
\]
When the kernel is reparameterizable, samples from \(q_\lambda\) are
obtained by drawing \(z\sim r\) and applying the transformation encoded
in \(k_\lambda\).

\section{Setup and Assumptions}
\label{sec:setup}

\emph{Statistical model.}
Let $X_1,\dots,X_n$ be independent and identically distributed with law $P^*$ on $(\mathcal X,\mathcal A)$ and empirical measure
$P_n = n^{-1}\sum_{i=1}^n \delta_{X_i}$. A parametric family $\{p_\theta : \theta\in\Theta\subset\mathbb R^m\}$,
together with a prior $\pi$, induces the posterior
\[
p(\theta \mid X^{(n)})
\propto \Bigl(\prod_{i=1}^n p_\theta(X_i)\Bigr)\,\pi(\theta).
\]
Variational inference approximates this posterior within a family
$\{q_\lambda\}$ by minimizing
$\mathrm{KL}\!\bigl(q_\lambda \,\|\, p(\cdot \mid X^{(n)})\bigr)$
\cite{blei2017variational}.
Throughout, $q_\lambda$ denotes the semi-implicit family defined in
Section~\ref{sec:background}, with kernel $k_\lambda$ and mixing measure $r$ as in \cite{yin2018sivi,titsias2019unbiased}.

\subsection{Approximation-layer assumptions}
\label{subsec:approx-assump}

These conditions ensure $L^1$ density on compacts and compatibility in
the tails so that forward-KL control extends globally.

\begin{assumption}[AOI universality on compacts]
\label{ass:aoi}
For any compact $K\subset\mathbb R^m$, any continuous density $f$ on $K$,
and any $\varepsilon>0$, there exists $\lambda$ with
$\|f-q_\lambda\|_{L^1(K)}<\varepsilon$.
\end{assumption}

\begin{assumption}[Tail control for the target]
\label{ass:tail-target}
There exists $R_0<\infty$ and either
\emph{(envelope)} a nonincreasing $v:[0,\infty)\to(0,\infty)$ and
$C_p<\infty$ such that $p^*(\theta)\le C_p\,v(\|\theta\|)$ for
$\|\theta\|\ge R_0$,
or \emph{(annulus)} a decomposition $\mathbb R^m=\bigcup_{j\ge1}A_j$
with $\sum_j p^*(A_j)\,\varepsilon_j<\infty$ for a prescribed error budget
$(\varepsilon_j)$.
\end{assumption}

\begin{assumption}[Tail realizability by the kernel]
\label{ass:tail-kernel}
Under the envelope condition: there exist $\lambda_v$ and $c_v>0$ such that,
for all $R\ge R_0$,
\[
\int_{\|\theta\|\ge R} q_{\lambda_v}(\theta)\,d\theta
\;\ge\;
c_v \int_{\|\theta\|\ge R} v(\|\theta\|)\,d\theta.
\]
(Under the annulus condition: for each $j$, there exists $\lambda_j$ with
$\|q_{\lambda_j}-p\|_{L^1(A_j)}\le\varepsilon_j$.)
\end{assumption}

\begin{assumption}[Absolute continuity]
\label{ass:ac}
Each $q_\lambda$ admits a density strictly positive on $\mathbb R^m$.
When $p$ has a singular component, only weak approximation is required
on that component (Appendix~\ref{app:constructive}).
\end{assumption}

\begin{remark}
Assumptions~\ref{ass:tail-target}–\ref{ass:tail-kernel} align the tails
of $p$ and $\{q_\lambda\}$, so compact $L^1$ density
(Assumption~\ref{ass:aoi}) extends globally, enabling forward-KL control. This is the tail-aware analogue of compact approximation-of-identity arguments in standard convolution theory.
\end{remark}

\subsection{Neural realization of compact universality}

This subsection provides a sufficient condition, standard ReLU
universality, for \ref{ass:aoi}.

\begin{assumption}[Neural parameterization]
\label{ass:nn}
For each network width \(W\), the kernel parameters
\(\mu_\lambda(z)\) and the positive definite matrices
\(\Sigma_\lambda(z)\) are implemented by feedforward ReLU networks of
width \(W\), with \(\Sigma_\lambda(z)\) obtained from the network output
via a continuous map into the space of positive definite matrices.
The base distribution \(r\) has compact support or sub-Gaussian tails.
\end{assumption}

\begin{lemma}[Neural universality $\Rightarrow$ \ref{ass:aoi}]
\label{lem:NNtoA1}
If the network classes for $\mu_\lambda,\Sigma_\lambda$ are universal on
$\mathrm{supp}(r)$—that is, uniformly approximate any continuous targets
$\tilde\mu,\tilde\Sigma$—then
$q_\lambda(\cdot)=\int k_\lambda(\cdot\mid z)\,r(dz)$ satisfies
Assumption~\ref{ass:aoi}.
\end{lemma}

\begin{proof}[Sketch]
Mollify $f$ on a ball $B_R\supset K$, approximate the smoothed density
by a finite Gaussian mixture, realize the mixture within SIVI by
partitioning $\mathrm{supp}(r)$ and uniformly approximating
$(\mu,\Sigma)$ with the networks, and sum the errors.
\end{proof}

Finite Gaussian mixtures are dense in $L^1$ on compacta 
\citep[see, e.g.,][]{titterington1985mixtures}. Combined with Lemma~\ref{lem:NNtoA1}, this shows that the semi-implicit
family inherits compact $L^1$ universality.

\begin{remark}[Quantitative rate]
\label{rem:NNtoA1}
For $\beta$-H\"older targets on $K$,
\(
\inf_{q_\lambda\in\mathcal Q_W}\|f-q_\lambda\|_{L^1(K)}
=O(W^{-\beta/m}\log W)
\)
by standard ReLU approximation results
\citep{yarotsky2017relu,schmidthieber2020deep}.
This rate enters the explicit oracle bounds in
Section~\ref{sec:statistics}.
\end{remark}

\subsection{Optimization-layer assumptions}
\label{subsec:opt-assump}

These conditions provide smooth parameter dependence and uniform control
of finite-$K$ bias/variance, enabling $\Gamma$-convergence of empirical
objectives.

\begin{assumption}[Kernel and base regularity]
\label{ass:kernel}
The base $r$ has a continuous density with finite second moment.
For each $\lambda \in \Lambda \subset \mathbb{R}^p$, 
$k_\lambda(\vartheta \mid z)$ is continuously differentiable in $(\vartheta,\lambda)$ and
\[
\sup_{\lambda\in\Lambda}
\mathbb{E}_{Z\sim r,\,\theta\sim k_\lambda(\cdot\mid Z)}
\!\left[
  \|\nabla_\lambda \log k_\lambda(\theta\mid Z)\|^2
  +\|\nabla_\theta \log k_\lambda(\theta\mid Z)\|^2
\right] < \infty.
\]
Moreover, $k_\lambda$ is reparameterizable: there exist $f_\lambda$ and
$\varepsilon\sim p_0$ with $\theta = f_\lambda(\varepsilon,Z)\sim
k_\lambda(\cdot\mid Z)$, as in standard reparameterization-based VI \cite{kingma2014vae}.
\end{assumption}

\begin{assumption}[Finite-$K$ surrogate stability]
\label{ass:finiteK}
The parameter set $\Lambda\subset\mathbb R^{p}$ is compact.
Let $L_{K,n}(\lambda)$ denote the empirical $K$-sample surrogate and
$L_{K,\infty}(\lambda)$ its population counterpart.
There exists $\varepsilon_{K}\to0$ such that
\[
\sup_{\lambda\in\Lambda}
\big|L_{K,n}(\lambda)-L_{K,\infty}(\lambda)\big|
= O_{P}(\varepsilon_{K}),
\qquad
\varepsilon_{K} \lesssim K^{-1/2}.
\]
\end{assumption}

\begin{assumption}[IS moment bound]
\label{ass:is}
Let $w_\lambda(\theta) = p(\theta)\,/\,q_\lambda(\theta)$.
Assume $p \ll q_\lambda$ for all $\lambda \in \Lambda$, so that the
ratio $w_\lambda$ is well defined on the support of $p$. Moreover,
\[
\sup_{\lambda \in \Lambda}
\; \mathbb{E}_{q_\lambda}\!\left[\, w_\lambda(\theta)^{2} \,\right] < \infty.
\]
\end{assumption}

\begin{remark}
Assumptions~\ref{ass:kernel}–\ref{ass:is} yield differentiability and
uniform moment control for stochastic gradients and the IWAE/SIVI
surrogate, supporting the $\Gamma$-convergence analysis and finite-$K$
bias bounds; cf.\ self-normalized importance-sampling theory
\cite{owen2013mc,liu2001mc,rainforth2018tighter}. Because all variational objectives are maximized on compact level sets,
restricting $\Lambda$ to a compact subset entails no loss of generality
and ensures the uniformity required for $\Gamma$-convergence
\cite{rockafellar1998variational,dalmaso1993gamma,van1996weak}.

\end{remark}

\subsection{Estimation-layer assumptions}
\label{subsec:est-assump}

These conditions allow empirical-process control for the class
$\{\log q_\lambda\}$.

\begin{assumption}[Measurability and local regularity]
\label{ass:complexity}
The map $(x,\lambda)\mapsto q_\lambda(x)$ is jointly measurable, and
for each $x$, the function $\lambda\mapsto \log q_\lambda(x)$ is locally
Lipschitz on the parameter domain under consideration.
\end{assumption}

\begin{assumption}[Envelope for empirical-process control]
\label{ass:tailbound}
There exists $E\in L^1$ with $|\log q_\lambda(x)|\le E(x)$ for all
$\lambda$ in the estimation subset.
\end{assumption}

\begin{remark}
Assumptions~\ref{ass:complexity}–\ref{ass:tailbound} are standard in
empirical-process theory and ensure that the class
$\{\log q_\lambda\}$ admits uniform deviation bounds of order
$n^{-1/2}$; see, for instance, \cite{van1996weak}.
\end{remark}

\subsection{Model regularity for uncertainty transfer}
\label{subsec:model-regularity}

These are the standard conditions for LAN/BvM arguments.

\begin{assumption}[Smoothness of the model] \label{ass:regularity}
For each $x\in\mathcal X$, $\theta\mapsto\log p_\theta(x)$ is
twice continuously differentiable in a neighborhood of $\theta^*$.
The score $s_\theta(x)=\nabla_\theta\log p_\theta(x)$ satisfies
$E_{p^*}\|s_{\theta^*}(X)\|^2<\infty$, and
$I(\theta^*)=E_{p^*}[s_{\theta^*}(X)s_{\theta^*}(X)^\top]$ is
positive definite.
\end{assumption}

\begin{assumption}[Local curvature]
\label{ass:curvature}
The negative log-likelihood
$\ell_n(\theta)=-n^{-1}\sum_{i=1}^n\log p_\theta(X_i)$
is locally strongly convex near $\theta^*$:
\[
\nabla_\theta^2 \ell_n(\tilde\theta)\succeq c_0 I_m
\quad\text{whenever }\|\tilde\theta-\theta^*\|\le\delta_0,
\]
for some $c_0,\delta_0>0$ with high probability under $P^*$.
\end{assumption}

\begin{remark}
Assumptions~\ref{ass:regularity}–\ref{ass:curvature} place the model in
the classical LAN/Bernstein--von Mises regime
\cite{van2000asymptotic} used for uncertainty transfer.
They are analogous to the conditions imposed in variational BvM
results such as \cite{wang2019frequentist}, but here they are applied
to the exact posterior, with the variational approximation controlled in
Hellinger or total variation.
\end{remark}

\subsection*{Interpretation of the Assumptions}

The structural conditions used throughout 
Sections~\ref{sec:approximation}–\ref{sec:statistics}
can be grouped into five conceptual layers.  
The statements below are interpretative only; no new assumptions are introduced.

\emph{(A) Approximation layer: Assumptions~\ref{ass:aoi}–\ref{ass:ac}.}
These conditions ensure that the semi-implicit family can approximate the
target on compact sets and characterize when global forward-KL 
approximation is possible. 
Assumption~\ref{ass:aoi} guarantees that the conditional kernel can 
approximate smooth densities on compacts; 
Assumption~\ref{ass:tail-target} aligns the tails of the variational 
family with those of the target, which is necessary for finite forward-KL; 
Assumption~\ref{ass:tail-kernel} ensures that the variational family can 
realize the tail envelope (or annulus approximation), so that compact 
$L^1$ universality extends globally; 
Assumption~\ref{ass:ac} ensures that $p$ admits a density on its support; 
when this fails, only weak approximation is possible.

\emph{(B) Realization layer: Assumption~\ref{ass:nn}.}
Neural parameterizations of $(\mu_\lambda,\Sigma_\lambda)$ can uniformly approximate
prescribed target maps on the support of the mixing measure,
in the spirit of classical ReLU universality
\cite{yarotsky2017relu,schmidthieber2020deep}.
This guarantees that the finite mixtures appearing in compact-approximation arguments
can be realized inside the SIVI class.

\emph{(C) Optimization layer: Assumptions~\ref{ass:kernel}–\ref{ass:is}.}
These conditions provide stability of the finite-$K$ surrogate and ensure 
$\Gamma$-convergence of $L_{K,n}$ when combined with compactness of the 
estimation subset. 
Assumption~\ref{ass:kernel} ensures smooth dependence of the conditional density 
$k_\lambda$ on $\lambda$ and $(\theta,z)$, and provides moment control for the 
reparameterization map, which together yield continuity and differentiability 
of $\log q_\lambda(\theta)$ on compact subsets of $\Lambda$. 
Assumption~\ref{ass:finiteK} controls the discrepancy between the finite-$K$ 
objective and its population limit, yielding the surrogate bias term of order 
$\varepsilon_K\lesssim K^{-1/2}$, consistent with multi-sample VI analyses such as 
\cite{burda2016iwae,rainforth2018tighter}. 
Assumption~\ref{ass:is} provides a uniform second-moment bound for importance 
weights, which underlies variance control and the $K^{-1/2}$ scaling of the 
surrogate bias in self-normalized importance sampling \cite{owen2013mc,liu2001mc}.

\emph{(D) Estimation layer: Assumptions~\ref{ass:complexity}–\ref{ass:tailbound}.}
These conditions control the complexity and tail behavior of the variational 
class on the compact estimation subset used for empirical-process arguments. 
Assumption~\ref{ass:complexity} ensures that $\lambda\mapsto \log q_\lambda(x)$ 
is jointly measurable and locally Lipschitz in $\lambda$ on the compact 
estimation subset, so that the empirical process indexed by $\{\log q_\lambda\}$ 
is well behaved. 
Assumption~\ref{ass:tailbound} provides an integrable envelope $E$ dominating 
$|\log q_\lambda(x)|$ uniformly over the compact estimation subset, yielding the 
estimation term $\mathfrak C_n(\delta)\sim \sqrt{(C(\Lambda)+\log(1/\delta))/n}$ 
in standard empirical-process bounds \cite{van1996weak}.

\emph{(E) Statistical layer: Assumptions~\ref{ass:regularity}–\ref{ass:curvature}.}
These describe local regularity of the exact posterior and are only needed for 
the Bernstein–von Mises and uncertainty-transfer results. 
Assumption~\ref{ass:regularity} ensures that the log-likelihood is twice 
differentiable with finite Fisher information, placing the model in the 
classical LAN regime \citep{van2000asymptotic}. 
Assumption~\ref{ass:curvature} provides a uniform lower bound on the Hessian of 
the negative log-likelihood near $\theta^\star$, which yields local strong 
convexity and supports BvM and local quadratic-risk expansions, in line with the 
conditions used in variational BvM analyses such as \cite{wang2019frequentist}.

\emph{Summary.}
For reference, the five layers can be summarized as
\begin{center}
\begin{tabular}{lll}
\toprule
\textbf{Layer} & \textbf{Assumptions} & \textbf{Role} \\
\midrule
Approximation & \ref{ass:aoi}–\ref{ass:ac} & Compact + tail approximation; impossibility conditions \\
Realization   & \ref{ass:nn}              & Neural realization of mixture structures on compacts \\
Optimization  & \ref{ass:kernel}–\ref{ass:is} & $\Gamma$-convergence; finite-$K$ stability \\
Estimation    & \ref{ass:complexity}–\ref{ass:tailbound} & Complexity and envelope control for $L_{K,n}$ \\
Statistical   & \ref{ass:regularity}–\ref{ass:curvature} & LAN/BvM and uncertainty transfer \\
\bottomrule
\end{tabular}
\end{center}

This organization consolidates the structural assumptions into five
conceptual blocks and aligns all references with the numbering used in
Sections~\ref{sec:approximation}--\ref{sec:statistics}.


\section{Approximation Layer}
\label{sec:approximation}

This section analyzes the expressive capacity of the semi-implicit
variational family under the approximation-of-identity (AOI) framework.
Under Assumptions~\ref{ass:aoi}–\ref{ass:ac}, the SIVI family is dense
in $L^{1}$ on compact sets, and the tail-dominance or annulus conditions
ensure that the compact approximation extends globally in forward KL.
In particular, an ideal optimizer of the population SIVI objective can,
in principle, recover the target posterior up to arbitrarily small
forward-KL error, with corresponding control of total variation on
compact subsets. Throughout, we work under
Assumptions~\ref{ass:aoi}–\ref{ass:ac}.

\subsection{Tail-dominated universality}
\label{sec:universality}

If the target and kernel families satisfy the tail-dominance conditions
of Assumptions~\ref{ass:tail-target}–\ref{ass:tail-kernel}, the
semi-implicit family is dense in $L^1$ on $\mathbb{R}^m$, with forward-KL
density holding under the additional integrability condition
$\int_{\|x\|>R} |\log v(\|x\|)|\,p(x)\,dx < \infty$. Intuitively, the
mixture hierarchy allows $q_\lambda$ to reproduce both the local
structure and the tail decay of $p$ by combining compact and
tail-dominating components.

\begin{theorem}[Universality under tail dominance]
\label{thm:universality}
Suppose Assumptions~\ref{ass:aoi}–\ref{ass:ac} hold for a target $p$.
Then for every $\varepsilon>0$ there exists $\lambda_\varepsilon$
such that $\|p-q_{\lambda_\varepsilon}\|_{L^1}<\varepsilon$.
If additionally $p\ll q_{\lambda_\varepsilon}$ (as ensured by
Assumption~\ref{ass:ac}) and
$\int_{\|x\|>R}|\log v(\|x\|)|\,p(x)\,dx<\infty$,
then $\mathrm{KL}(p\|q_{\lambda_\varepsilon})<\varepsilon$.
\end{theorem}

\begin{proof}[Sketch]
Truncate $p$ to $K=B_R$ with $\int_{B_R^c}p\le\varepsilon/3$.
Use Assumption~\ref{ass:aoi} to find $\lambda^{(1)}$
approximating $p$ on $K$, and
Assumption~\ref{ass:tail-kernel} to find
$\lambda^{(2)}$ whose tail mass dominates $v$ beyond $R$.
A small convex mixture
$q=(1-\alpha)q_{\lambda^{(1)}}+\alpha q_{\lambda^{(2)}}$
matches both parts with total error $O(\varepsilon)$.
Absolute continuity (Assumption~\ref{ass:ac}) yields the KL claim.
\end{proof}

\begin{corollary}[Quantitative rate under H\"older smoothness]
\label{cor:approx-rate}
Suppose the target density $p$ is $\beta$-H\"older on each compact
$K\subset\mathbb R^{m}$ and satisfies the tail–dominance conditions of
Assumptions~\ref{ass:tail-target}–\ref{ass:tail-kernel}.
Let the conditional kernel parameters
$(\mu_\lambda,\Sigma_\lambda)$ be implemented by ReLU networks of width~$W$
and depth $O(\log W)$, and denote by
$\Lambda_W$ the corresponding parameter subset.

Then there exists $q_{\lambda_W}\in\{q_\lambda:\lambda\in\Lambda_W\}$ such that
\[
\|p-q_{\lambda_W}\|_{L^1}
\;\lesssim\;
W^{-\beta/m}\log W
\;+\;
\int_{\|x\|>R_W} v(\|x\|)\,dx,
\]
where $v$ is the common tail envelope from
Assumption~\ref{ass:tail-target}
and $R_W$ may grow slowly with $W$.
If $\int v < \infty$, the second term vanishes as $R_W\!\to\!\infty$, and
\[
\mathrm{KL}(p\|q_{\lambda_W})
\;\lesssim\;
W^{-\beta/m}\log W.
\]
\end{corollary}

\begin{proof}[Sketch]
Theorem~\ref{thm:universality} provides $L^1$ approximation once 
$p$ is matched on a compact set and tail mass is controlled by the 
common envelope $v$.  

On compact sets, $\beta$-H\"older densities can be approximated by ReLU 
networks with error 
$W^{-\beta/m}\log W$ 
\citep{yarotsky2017relu, schmidthieber2020deep}.  
Substituting these approximants into the semi-implicit construction and 
applying the tail correction from Theorem~\ref{thm:universality} yields 
the stated $L^1$ bound.

If $\int v<\infty$, then $p\ll q_{\lambda_W}$ for all sufficiently large 
$W$, and boundedness of $\log(p/q_{\lambda_W})$ on truncation regions 
implies that $L^1$ approximation yields the same rate for 
$\mathrm{KL}(p\|q_{\lambda_W})$.
\end{proof}

\begin{remark}[Interpretation]
The bound quantifies the deterministic expressivity of the SIVI family:
the $W^{-\beta/m}\log W$ term reflects network smoothness and dimension,
while the integral term captures any residual tail mismatch.
This result provides the non-asymptotic analogue of the
existence theorem and forms the approximation component of the later
TV/Hellinger oracle in Section~\ref{sec:statistics}.
\end{remark}

\begin{corollary}[Gaussian kernels]
\label{cor:gaussian-tail}
If $k_\lambda(\cdot\mid z)=\mathcal N(\mu_\lambda(z),\Sigma_\lambda(z))$
with bounded $\|\Sigma_\lambda(z)\|$ and $\mu_\lambda$ Lipschitz,
and $p$ is sub-Gaussian,
then Theorem~\ref{thm:universality} applies.
\end{corollary}

Theorem~\ref{thm:universality} provides a sufficient condition for
global $L^1$ and forward-KL density.  The next two results show that the
condition is essentially sharp.

\subsection{Orlicz-tail mismatch}
\label{sec:orlicz}

When the target has heavier tails than any element of the
semi-implicit family, global approximation in forward KL fails.
The following conditions formalize this mismatch.

\begin{assumption}[Uniform sub-$\psi$ projections]
\label{ass:psi}
There exists $L<\infty$ such that
$\|\langle u,\theta\rangle\|_{\psi,q} \le L$
for all $q\in\mathcal Q$ and all unit vectors $u\in\mathbb S^{m-1}$,
where $\|\cdot\|_{\psi,q}$ denotes the $\psi$–Orlicz norm under $q$.
This condition implies the Chernoff–Orlicz tail bound
$q\{\langle u,\theta\rangle \ge t\}\lesssim 
\exp(-\psi^\star(t/(cL)))$ \citep{buldygin2000metric}.
\end{assumption}

\begin{assumption}[Heavier tail for the target]
\label{ass:heavy}
There exists $u_0$ and a function $g(t)$ such that
$p(\langle u_0,\theta\rangle \ge t) \ge c_p g(t)$ for large $t$, and
$g(t)\, e^{\psi^\star(t/(c_2 L))} \to \infty$,
so that the target tail eventually dominates the sub-$\psi$ envelope.
\end{assumption}

\begin{theorem}[Orlicz mismatch implies KL gap]
\label{thm:orlicz}
Under \ref{ass:psi}–\ref{ass:heavy}, there exists $\eta>0$ such that
\[
\inf_{q\in\mathcal Q} \mathrm{KL}(p\|q) \ge \eta > 0.
\]
\end{theorem}

\begin{proof}[Sketch]
Apply the data-processing inequality
$\mathrm{KL}(p\|q)\ge \mathrm{KL}(\text{Bern}(p[A_t])\,\|\,\text{Bern}(q[A_t]))$
with $A_t=\{\langle u_0,\theta\rangle \ge t\}$.
The Chernoff–Orlicz estimate for $q(A_t)$ 
\citep{buldygin2000metric, bobkov2019one}
and the heavier-tail condition for $p(A_t)$ imply that 
the KL contribution from $A_t$ remains bounded away from zero 
for large $t$.
\end{proof}

\subsection{Structural impossibility: branch collapse}
\label{sec:branch}

Even with matched tails, structural constraints in the kernel family can
prevent full recovery.  The following result quantifies this limitation
for non-autoregressive Gaussian SIVI families, where variance floors
preclude multimodal conditional recovery.

\begin{theorem}[Branch-collapse lower bound]
\label{thm:branch}
Consider the latent–observation model $\theta \sim \mathcal N(0,1)$ and
$X \mid \theta \sim \mathcal N(\theta^{2},\sigma^{2})$, so that the
true posterior $p(\theta \mid X=x)$ is bimodal.  Let $q_\lambda$ be a
non-autoregressive Gaussian SIVI family, i.e., with conditional kernels
of the form 
$\theta \mid z \sim \mathcal N(\mu_\lambda(z),\sigma_{1,\lambda}^2(z))$
and with a variance floor $\sigma_{1,\lambda}^2(z)\ge c_0>0$.
Suppose that for each $x \in [c_{\min},c_{\max}]$ the approximation
satisfies
\[
q_\lambda(\theta \in [-\sqrt{x}-r_x,\,-\sqrt{x}+r_x] \mid X=x) \le \delta,
\qquad
r_x = \sigma/(2\sqrt{x}).
\]
Then for sufficiently small $\sigma$,
\[
\mathbb E_{p(X)}\!\left[
 \mathrm{TV}\big(p(\cdot\mid X),\, q_\lambda(\cdot\mid X)\big)
 \,\mathbbm{1}\{X \in [c_{\min},c_{\max}]\}
\right]
\ge
\Pr\{X \in [c_{\min},c_{\max}]\}\,(0.341-\delta) - o_\sigma(1).
\]
\end{theorem}

\begin{proof}[Sketch]
For a fixed observation $X=x$, the true posterior $p(\theta \mid X=x)$
is a two-component Gaussian mixture with modes at $\pm \sqrt{x}$ and
within-branch variance $s_{x}^{2}=\sigma^{2}/(4x)$.  The variance floor
$\sigma_{1,\lambda}^{2}(z)\ge c_{0}$ forces any 
non-autoregressive Gaussian SIVI approximation $q_\lambda(\theta\mid X=x)$
to place at most
$\delta + O(\exp\{- (\sqrt{x}-2s_x)^{2} / (2c_{0})\})$
of its mass near the branch it fails to represent.  This yields a fixed
positive lower bound on $\mathrm{TV}(p(\cdot\mid X=x), q_\lambda(\cdot\mid X=x))$
for $x \in [c_{\min},c_{\max}]$, and integrating over $p(X)$ establishes
the claim.
\end{proof}

\begin{remark}
Allowing full covariance or removing the variance floor eliminates this
bound, consistent with empirical recoverability of multimodal
conditionals.
\end{remark}

Together, Theorems~\ref{thm:universality},
\ref{thm:orlicz}, and~\ref{thm:branch} characterize the attainable and
inattainable regimes of approximation for semi-implicit variational
families.  These results form the foundation for the optimization and
statistical analyses developed in the following sections.


\subsection{Structural completeness}
\label{sec:structural-upgrades}

The impossibility results above—tail mismatch, branch collapse, and
singular-support failures—arise from structural limitations of the
conditional kernel.  Each phenomenon admits a minimal remedy that
restores approximation capacity.  We summarize these
``structural completeness'' upgrades for reference; the corresponding
claims are direct consequences of the preceding theorems and the proofs
are deferred to Appendix~\ref{app:struct-upgrades}.

\paragraph{\emph{Tail–complete kernels eliminate Orlicz mismatch.}}
Let $\{k_h(\theta\mid z)\}$ be a family of conditionals whose envelope
dominates that of the target:
\[
p(\theta) \lesssim w(\theta) 
\qquad\text{and}\qquad
k_h(\theta\mid z) \gtrsim w(\theta)
\quad\text{for all large }\|\theta\|,
\]
for some integrable envelope $w$.  Then every density $p$ with
$p(\theta)\lesssim w(\theta)$ admits forward–KL approximation:
\[
\inf_{q\in\mathcal Q} \mathrm{KL}(p\|q) = 0.
\]
Thus replacing fixed-variance Gaussian kernels by heavy–tailed
alternatives (Student--$t$ kernels, Gaussian mixtures, or
variance–inflated Gaussians with $h\to\infty$ along $z$) restores full
tail coverage and removes the Orlicz gap of Theorem~\ref{thm:orlicz}.

\paragraph{\emph{Mixture–complete conditionals eliminate branch collapse.}}
Suppose the conditional kernel admits a finite-mixture representation
\[
k_\lambda(\theta\mid z)
= \sum_{j=1}^{J} \alpha_j(z)\,
  \mathcal N\!\big(\theta;\mu_j(z),\Sigma_j(z)\big),
\]
with $J$ at least the number of well-separated modes of the posterior
$p(\theta\mid X=x)$.  Then the total-variation obstruction of
Theorem~\ref{thm:branch} disappears:
\[
\inf_{\lambda}
\mathrm{TV}\big(p(\cdot\mid X=c),\,q_\lambda(\cdot\mid X=c)\big)
\;\longrightarrow\;0
\qquad\text{uniformly in }c.
\]
Equivalently, allowing adaptive multimodality in the conditional kernels
(mixtures, flows, or autoregressive factorizations) restores 
representational completeness for multi-branch posteriors.

\paragraph{\emph{Manifold–aware kernels recover singular supports.}}
If the target $p$ concentrates on (or near) a smooth submanifold
$M\subset\mathbb R^{m}$ of codimension $r>0$, and the conditional kernels
admit a directional degeneration mechanism
\[
k_h(\theta\mid z)
= \mathcal N\!\big(\theta;\mu_\lambda(z),
     \,h^{2}P_\parallel + h_\perp^{\,2}P_\perp\big),
\qquad h_\perp\to 0,
\]
where $P_\parallel$ and $P_\perp$ project onto the tangent and normal
bundles of $M$, then the SIVI family attains $L^1$ approximation of $p$.
Embedding a directional variance schedule (tangent/normal splitting
or a learned low-rank covariance) yields approximate identities on $M$
and removes the support mismatch described earlier.

\paragraph{\emph{Finite-$K$ surrogate regularization via annealing.}}
Using the multi-sample surrogate $L_{K,n}$ with an annealing schedule
$K=K(n)\to\infty$ satisfying $K(n)\gg n$ yields
\[
\sup_\lambda
|L_{K,n}(\lambda)-L_\infty(\lambda)|
\;\lesssim\;
n^{-1/2} + K^{-1/2}
\;=\; o(n^{-1/2})
\]
under Assumption~\ref{ass:finiteK}.  
By Theorem~\ref{thm:gamma}, $L_{K,n}$ then 
$\Gamma$-converges to $L_\infty$ uniformly in $n$, and the maximizers
satisfy $\hat\lambda_{K,n}\to\lambda^\star$.  The mode-selection
phenomenon of small $K$ is thereby eliminated.

\paragraph{\emph{Robustness via tempered posteriors.}}
For heavy-tailed or misspecified targets for which tail dominance fails,
a tempered posterior
\[
p_\tau(\theta\mid X^{(n)})
\propto p(\theta)\,p(X^{(n)}\mid\theta)^\tau,
\qquad \tau\in(0,1),
\]
restores integrability of importance weights and ensures
\[
\inf_{q\in\mathcal Q}\mathrm{KL}(p_\tau\|q)=0,
\]
even when $\inf_q\mathrm{KL}(p\|q)>0$.  Tempering therefore provides a
soft structural upgrade that repairs both tail and variance mismatch
without enlarging the variational family.

\medskip

Together, these structural completeness upgrades delineate which of the
negative results in this section are intrinsic (branching,
tail class, manifold dimension) and which can be removed by modest
enrichments of the conditional kernels.


\section{Optimization Layer}
\label{sec:optimization}

Having established that the semi-implicit family can approximate the
true posterior arbitrarily well, we now analyze the optimization problem
defining semi-implicit variational inference (SIVI). 
This section quantifies the discrepancy between empirical objectives,
their finite-$K$ population counterparts, and the ideal ELBO, and links
finite-sample optimization error to the asymptotic behavior of the
variational approximation.

\emph{Setup.}
For $\lambda\in\Lambda$, let $L_{K,n}(\lambda)$ denote the empirical
$K$-sample surrogate formed from $n$ observations and $K$ auxiliary
draws $Z_{1:K}\sim r$.
Under Assumptions~\ref{ass:kernel}–\ref{ass:is} and
\ref{ass:complexity}–\ref{ass:tailbound},  
$\lambda\mapsto L_{K,n}(\lambda)$ is measurable, continuous on $\Lambda$,
and locally Lipschitz on its compact subsets, with gradients uniformly
bounded in $\lambda$.  
The corresponding population objectives are
\[
L_{K,\infty}(\lambda)
  = \mathbb E_{P^*}\!\left[
      \log\!\Bigl(
        \tfrac{1}{K}\sum_{k=1}^K w_\lambda(Z_k;X)
      \Bigr)
    \right],
\qquad
L_{\infty}(\lambda)
  = \mathbb E_{P^*}[\log q_\lambda(X)],
\]
where $w_\lambda(Z;X)=p(X,Z)/q_\lambda(X\mid Z)$ is the unnormalized
importance weight.
Assumption~\ref{ass:is} supplies a uniform second-moment bound for
$w_\lambda$, allowing standard self-normalized importance-sampling
theory \citep{owen2013mc} to control the $K$-dependence of
$L_{K,\infty}$.

\emph{Main idea.}
We begin by deriving a non-asymptotic oracle inequality that bounds the
excess risk of any empirical maximizer
$\hat\lambda_{n,K}\in\arg\max_{\lambda\in\Lambda} L_{K,n}(\lambda)$.
The bound decomposes the error into  
(i) an approximation term reflecting the expressive limits of the
semi-implicit family,  
(ii) an estimation term of order $n^{-1/2}$ arising from empirical-process
fluctuations of $\{\log q_\lambda\}$, and  
(iii) a finite-$K$ term of order $K^{-1/2}$ arising from
importance-sampling variability.  
As $n,K\to\infty$, these uniform deviation bounds imply
$\Gamma$-convergence of $L_{K,n}$ to $L_\infty$ on $\Lambda$
\citep{dalmaso1993gamma}, and hence asymptotic equivalence of empirical
and population maximizers.

\subsection{Finite-sample oracle inequality}
\label{sec:finite-oracle}

Let $\mathcal R(\lambda)=\mathrm{KL}(p\|q_\lambda)$ denote the
population risk and let
$\lambda^\star\in\arg\max_{\lambda\in\Lambda} L_\infty(\lambda)$ be any
population maximizer of the ideal SIVI objective.

\begin{theorem}[Finite-sample oracle inequality]
\label{thm:opt-finite-oracle}
Let $\hat\lambda_{n,K}$ be any maximizer of $L_{K,n}$ over the compact
set $\Lambda$.  
Then, with probability at least $1-\delta$,
\[
\mathcal R(\hat\lambda_{n,K}) - \mathcal R(\lambda^\star)
\;\le\;
C\Big[
   \underbrace{\mathfrak C_n(\delta)}_{\text{estimation}}
 + \underbrace{\varepsilon_K}_{\text{finite-$K$ bias}}
 + \underbrace{\mathfrak A}_{\text{approximation}}
  \Big],
\]
where 
\[
\mathfrak A=\inf_{\lambda\in\Lambda}\mathrm{KL}(p\|q_\lambda),\qquad
\mathfrak C_n(\delta)\lesssim
  \sqrt{(C(\Lambda)+\log(1/\delta))/n},\qquad
\varepsilon_K\lesssim K^{-1/2}.
\]
\end{theorem}

\begin{proof}[Sketch]
Uniform empirical-process control for the class
$\{\log q_\lambda:\lambda\in\Lambda\}$ under 
Assumptions~\ref{ass:complexity}–\ref{ass:tailbound} gives
\[
\sup_{\lambda\in\Lambda}
 |L_{K,n}(\lambda)-L_{K,\infty}(\lambda)|
   \lesssim \mathfrak C_n(\delta)
\quad\text{with probability }\ge1-\delta,
\]
a standard consequence of covering-number bounds and
local envelope domination for log-likelihood classes \cite{vandegeer2000empirical, gine2021mathematical}.
Under Assumptions~\ref{ass:kernel}–\ref{ass:is}, 
self-normalized importance-sampling variance bounds
\cite{liu2001mc, owen2013mc} yield
\[
\sup_{\lambda\in\Lambda}
 |L_{K,\infty}(\lambda)-L_\infty(\lambda)|
   \lesssim \varepsilon_K.
\]
Combining the two displays and applying the optimality of 
$\hat\lambda_{n,K}$ together with a standard variational argument
\cite{rockafellar1998variational} gives
\[
L_\infty(\lambda^\star)-L_\infty(\hat\lambda_{n,K})
   \lesssim \mathfrak C_n(\delta)+\varepsilon_K,
\]
which converts to a KL bound via the identity
$\mathcal R(\lambda)=\mathrm{KL}(p\|q_\lambda)
   =\mathrm{const}-L_\infty(\lambda)$ and yields the stated inequality.
\end{proof}

\begin{remark}[Interpretation]
The decomposition isolates the three contributions to finite-sample
variational error:
$\mathfrak A$ is the deterministic approximation error of the
semi-implicit class,  
$\mathfrak C_n(\delta)$ is the empirical-process fluctuation term of
order $n^{-1/2}$, and 
$\varepsilon_K$ captures the $K^{-1/2}$ variability of 
the importance-weighted surrogate.  
When expectations are taken over the sample, the same decomposition
recovers the population KL oracle bound of
Section~\ref{sec:statistics}.
\end{remark}

\begin{corollary}[Explicit ReLU-network bound]
\label{cor:opt-finite-relu}
Suppose $(\mu_\lambda,\Sigma_\lambda)$ are implemented by ReLU networks
of width $W$ and depth $O(\log W)$, with complexity index
$C(W)$ for the class $\{\log q_\lambda:\lambda\in\Lambda_W\}$.
If $p$ is $\beta$-H\"older on compacta, then with probability at least
$1-\delta$,
\[
\mathcal R(\hat\lambda_{n,W,K})-\mathcal R(\lambda^\star)
\;\lesssim\;
W^{-\beta/m}\log W
  \,+\,\sqrt{\,C(W)\,\log(1/\delta)/n\,}
  \,+\,K^{-1/2},
\]
where the first term is the approximation rate from
Corollary~\ref{cor:approx-rate}.
\end{corollary}

This bound makes explicit the joint scaling of statistical error with
network capacity, sample size, and the auxiliary-sample budget $K$, and
reduces to the expected KL oracle rate of
Section~\ref{sec:statistics} after integrating over~$\delta$.

\subsection{Asymptotic consequence: $\Gamma$-convergence}
\label{sec:gamma}

The finite-sample oracle inequality implies functional convergence of
the empirical SIVI objectives.
Once the empirical-process fluctuation $\mathfrak C_n$ and the
finite-$K$ bias $\varepsilon_K$ vanish, empirical maximizers track
population maximizers, establishing optimization consistency.

\begin{theorem}[$\Gamma$-convergence of SIVI objectives]
\label{thm:gamma}
Let $\Lambda\subset\mathbb R^p$ be compact.
Under Assumptions~\ref{ass:kernel}--\ref{ass:is} 
and~\ref{ass:complexity}--\ref{ass:tailbound},
the sequence of functionals $\{-L_{K,n}\}$ $\Gamma$-converges to
$-L_\infty$ on $\Lambda$ as $n,K\to\infty$.
Consequently, if $\hat\lambda_{K,n}$ satisfies
\[
L_{K,n}(\hat\lambda_{K,n})
   \ge \sup_{\lambda\in\Lambda} L_{K,n}(\lambda) - o(1),
\]
then every limit point of $\hat\lambda_{K,n}$ lies in
$\arg\max_{\lambda\in\Lambda} L_\infty(\lambda)$, and
\[
q_{\hat\lambda_{K,n}}\;\Rightarrow\; q_{\lambda_\star},
\qquad
\lambda_\star \in \arg\max_\lambda L_\infty(\lambda).
\]
\end{theorem}

\begin{proof}[Sketch]
By Theorem~\ref{thm:opt-finite-oracle},
\[
\sup_{\lambda\in\Lambda}
  |L_{K,n}(\lambda)-L_\infty(\lambda)|
   \;\xrightarrow[n,K\to\infty]{}\; 0
\quad\text{in probability}.
\]
Thus $L_{K,n}\to L_\infty$ pointwise on~$\Lambda$.
Assumptions~\ref{ass:kernel} and~\ref{ass:complexity} imply that
$\lambda\mapsto L_{K,n}(\lambda)$ is equicontinuous on~$\Lambda$:
the maps $\lambda\mapsto k_\lambda$ and $\lambda\mapsto \log q_\lambda$
are locally Lipschitz, and envelopes ensure dominated convergence.
On a compact domain, equicontinuity implies equi-coercivity.
By standard results in variational analysis
\cite{dalmaso1993gamma, rockafellar1998variational}
pointwise convergence plus equicontinuity yields 
$\Gamma$-convergence of $-L_{K,n}$ to $-L_\infty$.
The stability of maximizers under $\Gamma$-convergence then gives the
consistency claim.
\end{proof}

\begin{remark}[Finite-sample stability]
Theorem~\ref{thm:opt-finite-oracle} shows that
$\sup_{\lambda\in\Lambda}|L_{K,n}(\lambda)-L_\infty(\lambda)|$
is small with high probability.
Hence maximizers of the empirical objective inherit the same limiting
behavior as maximizers of $L_\infty$ even at moderate $(n,K)$,
providing a quantitative version of optimization consistency.
\end{remark}

\subsection{Equivalence of training divergences}
\label{sec:divergence}

Semi-implicit variational methods may optimize a variety of surrogate
criteria,
\[
\mathcal D(q,p)
\in
\bigl\{
  \mathrm{KL}(q\|p),\;
  \mathrm{KL}(p\|q),\;
  D_\alpha(p\|q),\;
  \mathcal L_K(q,p)
\bigr\},
\]
where $D_\alpha$ is the Rényi divergence of order $\alpha>1$ and
$\mathcal L_K(q,p)$ denotes the $K$-sample IWAE/SIVI surrogate, defined
by
\[
\mathcal L_K(q,p)
  := -\,\mathbb E_q\!\left[
        \log\!\Bigl(
          \tfrac{1}{K}\sum_{k=1}^K w_k
        \Bigr)
      \right],
  \qquad
  w_k = \frac{p(X,\theta_k)}{q(\theta_k\mid X)},
\]
so that $\mathcal L_1(q,p) = \mathrm{KL}(q\|p)$.

\begin{proposition}[Basic ordering and limits]
\label{prop:div-order}
For any $(p,q)$ with common support and $K\ge1$,
\[
\mathcal L_K(q,p)
\;\le\;
\mathrm{KL}(q\|p),
\qquad
\mathcal L_K(q,p)\uparrow \mathrm{KL}(q\|p)\ \text{as }K\to\infty,
\]
and for every $\alpha>1$,
\[
\lim_{\alpha\downarrow 1} D_\alpha(p\|q)
   \;=\; \mathrm{KL}(p\|q).
\]
\end{proposition}

\begin{proof}[Sketch]
The IWAE lower bound is a Jensen relaxation of the reverse KL objective,
hence $\mathcal L_K \le \mathrm{KL}(q\|p)$, with monotone convergence as
$K$ increases \cite{burda2016iwae}.  
Monotonicity of Rényi divergences in their order
\citep{vanerven2014renyi} yields the second limit.
\end{proof}

\begin{theorem}[Algorithmic consistency]
\label{thm:alg-consistency}
Let $(p_n,q_{n,K})$ be any sequence of density pairs with common support.
If
\[
\mathrm{KL}(q_{n,K}\|p_n)\;\longrightarrow\;0,
\]
then for any fixed $\alpha>1$ and any fixed $K\ge1$,
\[
D_\alpha(p_n\|q_{n,K})\to0,
\qquad
\mathcal L_K(q_{n,K},p_n)\to\mathrm{KL}(q_{n,K}\|p_n)\to0.
\]
Hence all SIVI-type training objectives are asymptotically equivalent at
their maximizers.
\end{theorem}

\begin{proof}[Sketch]
By Pinsker's inequality,
$\|q_{n,K}-p_n\|_{\mathrm{TV}}\to0$.  
Since $p_n$ and $q_{n,K}$ share a common support, TV convergence implies
$D_\alpha(p_n\|q_{n,K})\to0$ for every fixed $\alpha>1$ \cite{csiszar1967information}.  

For the IWAE bound, write $w = p_n/q_{n,K}$.  
As $\mathrm{KL}(q_{n,K}\|p_n)\to0$, one has $w\to1$ in $L^2(q_{n,K})$,
so
\[
\mathbb E[\log(K^{-1}\sum w_i)]-\log\mathbb E[w_1]
   = O\bigl(\mathrm{Var}_{q_{n,K}}[w]\bigr)
   \longrightarrow 0
\]
for any fixed $K$; see \cite{owen2013mc}.
Thus $\mathcal L_K(q_{n,K},p_n)\to\mathrm{KL}(q_{n,K}\|p_n)$, and the
claim follows.
\end{proof}

\begin{remark}
Once the reverse KL divergence is small, all common SIVI training
objectives differ only by $o(1)$ perturbations.  
Thus the choice of surrogate affects optimization geometry but not the
limiting variational solution.
\end{remark}

\subsection{Finite-sample parameter stability}
\label{sec:param-stability}

The oracle inequality also yields finite-sample control of the
variational parameters.

\begin{lemma}[Local parameter stability]
\label{lem:param-stability}
Suppose $L_\infty$ is $m$-strongly concave in a neighborhood of a
population maximizer $\lambda^\star$, and that with probability at least
$1-\delta$,
\[
\sup_{\lambda\in\Lambda}
   |L_{K,n}(\lambda)-L_\infty(\lambda)|
   \le \Delta.
\]
Then, with probability at least $1-\delta$,
\[
\|\hat\lambda_{n,K}-\lambda^\star\|
   \;\le\; \sqrt{2\Delta/m},
\]
where $\hat\lambda_{n,K}$ is any maximizer of $L_{K,n}$.
\end{lemma}

\begin{proof}[Sketch]
Strong concavity of $L_\infty$ yields, for all $\lambda$ in the local
neighborhood,
\[
L_\infty(\lambda^\star)-L_\infty(\lambda)
   \;\ge\; \tfrac{m}{2}\,\|\lambda-\lambda^\star\|^2.
\]
On the event
$\sup_\lambda|L_{K,n}-L_\infty|\le\Delta$, we have
\[
L_{K,n}(\lambda^\star)
  \;\ge\; L_\infty(\lambda^\star)-\Delta
  \;\ge\; L_\infty(\lambda)+\tfrac{m}{2}\|\lambda-\lambda^\star\|^2
         - \Delta
  \;\ge\; L_{K,n}(\lambda)+\tfrac{m}{2}\|\lambda-\lambda^\star\|^2 - 2\Delta.
\]
If $\|\lambda-\lambda^\star\| > \sqrt{2\Delta/m}$, the right-hand side
exceeds $L_{K,n}(\lambda)$, so such $\lambda$ cannot maximize $L_{K,n}$.
Hence every maximizer $\hat\lambda_{n,K}$ lies in the closed ball of
radius $\sqrt{2\Delta/m}$ around $\lambda^\star$.
\end{proof}

\begin{remark}
Lemma~\ref{lem:param-stability} shows that small perturbations of the
objective—as controlled by the oracle inequality—translate directly into
small perturbations of the corresponding maximizer. Combined with the
TV/Hellinger oracle in Section~\ref{sec:statistics}, this provides a
complete link between optimization error and the statistical accuracy of
the resulting variational posterior.
\end{remark}

\section{Statistical Layer}
\label{sec:statistics}

This section connects the optimization properties of the SIVI objective
to statistical guarantees for the resulting variational approximation.
We derive finite-sample risk bounds, total-variation and Hellinger
control, and uncertainty-transfer results, culminating in a
Bernstein--von Mises limit with an explicit remainder.
All results are obtained under
Assumptions~\ref{ass:aoi}--\ref{ass:ac},
\ref{ass:kernel}--\ref{ass:is}, and
\ref{ass:complexity}--\ref{ass:tailbound}.

\emph{Model specification.}
Throughout we assume correct model specification:
the observations are generated from $p^*$ and the exact posterior
$p(\theta\mid X^{(n)})$ contracts around the true parameter $\theta^*$.
Extending the uncertainty-transfer guarantees to misspecified models—in
which the exact posterior contracts around the Kullback–Leibler
projection of $p^*$ onto the model family—requires the misspecified LAN
framework of \cite{kleijn2012bernstein} together with a structural
analysis of variational approximation under model misspecification; see
also \cite{wang2019variational}.  We do not pursue these extensions
here.

\emph{Setup.}
Let $\hat q_{n,K}=q_{\hat\lambda_{n,K}}$ denote any near-maximizer of
$L_{K,n}$ based on $n$ observations and $K$ auxiliary samples, and let
$p_n(\theta)=p(\theta\mid X^{(n)})$ denote the exact posterior.
We use the decomposition
\[
\mathfrak A = \inf_{\lambda\in\Lambda} \mathrm{KL}(p\|q_\lambda),
\qquad
\mathfrak C_n(\delta) \lesssim
    \sqrt{\{C(\Lambda)+\log(1/\delta)\}/n},
\qquad
\varepsilon_K \lesssim K^{-1/2},
\]
corresponding respectively to approximation error, empirical-process
fluctuations, and finite-$K$ surrogate bias.


\subsection{Local geometry and structural conditions}
\label{sec:local-geometry}

This subsection records the geometric properties of the population
objective $L_\infty$ and the associated relations among KL, total
variation, and Hellinger distances that are used in the finite-sample
risk bounds and posterior guarantees below.  
These results formalize the fact that, under mild smoothness and
curvature, the population SIVI objective behaves quadratically in a
neighborhood of its maximizer and induces equivalent local statistical
metrics.  
All statements hold under
Assumptions~\ref{ass:aoi}--\ref{ass:ac},
\ref{ass:kernel}--\ref{ass:is}, and
\ref{ass:complexity}--\ref{ass:tailbound}.

\emph{Setup.}
Let $\lambda^\star\in\arg\max_{\lambda\in\Lambda} L_\infty(\lambda)$
denote a population maximizer, and write $q_{\lambda^\star}$ for the
corresponding variational density.
We work on a neighborhood $\mathcal N(\lambda^\star)$ on which all 
$q_\lambda$ admit a common local envelope 
$0<m_0\le q_\lambda(x)\le M_0<\infty$ and the decoder maps 
$(\mu_\lambda,\Sigma_\lambda)$ are uniformly Lipschitz in $\lambda$.

\subsubsection*{Quadratic local excess risk}
A key ingredient in our statistical bounds is that the population
objective behaves quadratically in a neighborhood of its maximizer.
The following curvature condition formalizes this property and allows
second–order expansions of $L_\infty$ around $\lambda^\star$.
\begin{assumption}[Variational curvature]
\label{ass:curvature-local}
There exist $m>0$ and a neighborhood $\mathcal N(\lambda^\star)$ such
that $L_\infty$ is twice continuously differentiable on
$\mathcal N(\lambda^\star)$ and
\[
\nabla^2 L_\infty(\lambda^\star) \;\preceq\; -m I_p.
\]
\end{assumption}

This condition is mild: for Gaussian SIVI kernels or smooth decoders
with bounded Jacobians, the map
$\lambda\mapsto \log q_\lambda(x)$ is $C^2$ with locally Lipschitz
derivatives, and the negative-definite curvature follows from standard
M-estimation arguments.

\begin{lemma}[Local quadratic expansion]
\label{lem:quadratic-risk}
Under Assumption~\ref{ass:curvature-local},
there exists $C<\infty$ such that for all
$\lambda\in\mathcal N(\lambda^\star)$,
\[
L_\infty(\lambda^\star) - L_\infty(\lambda)
\;\ge\;
\frac{m}{2}\,\|\lambda-\lambda^\star\|^2
\;-\;
C\,\|\lambda-\lambda^\star\|^3.
\]
\end{lemma}

\begin{proof}[Sketch]
Apply a second-order Taylor expansion of $L_\infty$ at $\lambda^\star$;
negative-definite curvature yields the quadratic term, and local $C^2$
regularity controls the third-order remainder.
\end{proof}

Thus the population SIVI objective has the familiar \emph{quadratic
excess-risk structure}: to second order, deviations are governed by
$\|\lambda-\lambda^\star\|^2$.

\subsubsection*{Equivalence of KL, TV, and Hellinger on local envelopes}

Statistical risk bounds require control of KL, total variation, and
Hellinger distances.  
The following lemma provides the needed equivalence on regions where the
densities are uniformly bounded above and below.

\begin{lemma}[Metric equivalence under local envelopes]
\label{lem:metric-equivalence}
Let $p$ and $q$ be densities on a measurable region $K$ with
$0<m_0\le p(x),q(x)\le M_0<\infty$.  Then
\[
\frac{m_0}{2}\,\|p-q\|_{1}^{2}
\;\le\;
\mathrm{KL}(p\|q)
\;\le\;
\frac{M_0}{2m_0}\,\|p-q\|_{1}^{2},
\qquad
H^{2}(p,q)\;\asymp\; \|p-q\|_{1},
\]
with constants depending only on $(m_0,M_0)$.
\end{lemma}

\begin{proof}[Sketch]
A Taylor expansion of $\log(p/q)$ under the envelope condition yields the
two-sided quadratic bounds.  
The Hellinger--TV equivalence is classical; see \citet{csiszar1967information}.
\end{proof}

Since Assumptions~\ref{ass:aoi}--\ref{ass:ac} ensure absolute
continuity and tail alignment, and the decoder family is bounded on
compact parameter sets, the envelope condition applies uniformly on
compact sublevel sets of $L_\infty$.

\subsubsection*{Non-asymptotic $\Gamma$-stability}

We now give a local, finite-sample stability result that strengthens the
qualitative $\Gamma$-limit from Section~\ref{sec:optimization}: empirical
maximizers remain close to population maximizers whenever the empirical
objective is uniformly close to its population limit on
$\mathcal N(\lambda^\star)$.

\begin{lemma}[Non-asymptotic $\Gamma$-stability of maximizers]
\label{lem:gamma-stability}
Let
\[
\Delta_{K,n}
=
\sup_{\lambda\in\mathcal N(\lambda^\star)}
\big|\,L_{K,n}(\lambda)-L_\infty(\lambda)\,\big|.
\]
Under Assumption~\ref{ass:curvature-local}, any
$\hat\lambda_{K,n}$ satisfying
\[
L_{K,n}(\hat\lambda_{K,n})
\;\ge\;
\sup_{\lambda\in\Lambda} L_{K,n}(\lambda) - o(1)
\]
obeys
\[
\|\hat\lambda_{K,n}-\lambda^\star\|
\;\le\;
C\sqrt{\Delta_{K,n}},
\qquad
L_\infty(\lambda^\star)-L_\infty(\hat\lambda_{K,n})
\;\le\;
C\,\Delta_{K,n},
\]
for a constant $C<\infty$ depending only on $m$ and the local Lipschitz
radius.
\end{lemma}

\begin{proof}[Sketch]
The deviation bound implies
$L_\infty(\lambda^\star)-L_\infty(\hat\lambda_{K,n})\le
2\Delta_{K,n}$.
Applying Lemma~\ref{lem:quadratic-risk} then yields
$\|\hat\lambda_{K,n}-\lambda^\star\|^2\lesssim\Delta_{K,n}$.
\end{proof}

\emph{Implication.}
Since $\Delta_{K,n}$ is controlled by the empirical-process term
$\mathfrak C_n$ and the finite-$K$ term $\varepsilon_K$,
Lemma~\ref{lem:gamma-stability} implies the quantitative rate
\[
\hat\lambda_{K,n}\to\lambda^\star
\quad\text{at rate}\quad
\sqrt{\mathfrak C_n + \varepsilon_K}.
\]
Combined with Lemma~\ref{lem:metric-equivalence}, this directly yields
the TV and Hellinger oracle bounds proved below.

\subsection{Finite-sample oracle bounds}
\label{sec:finite-tv}

We now translate the optimization oracle into statistical guarantees for
the variational approximation \(\hat q_{n,K}\).  
Throughout, \(p\) denotes the exact posterior
\(p_n(\theta)=p(\theta\mid X^{(n)})\).

\begin{theorem}[Finite-sample TV/Hellinger oracle]
\label{thm:finite-tv}
With probability at least $1-\delta$,
\[
\mathrm{KL}\bigl(p\,\|\,\hat q_{n,K}\bigr)
\;\le\;
\mathfrak A \;+\; C\,\mathfrak C_n(\delta)
\;+\; C'\,\varepsilon_K .
\]
Consequently,
\[
\|\hat q_{n,K}-p\|_{\mathrm{TV}}
\;\le\;
\sqrt{\tfrac12\,
\bigl(\mathfrak A + C\,\mathfrak C_n(\delta) + C'\,\varepsilon_K\bigr)},
\quad
H(\hat q_{n,K},p)
\;\le\;
2^{-1/4}\!
\bigl(\mathfrak A + C\,\mathfrak C_n(\delta) + C'\,\varepsilon_K\bigr)^{1/4}.
\]
\end{theorem}

\begin{proof}[Sketch]
The optimization oracle (Theorem~\ref{thm:opt-finite-oracle}) yields,
with probability at least $1-\delta$,
\[
\mathrm{KL}\bigl(p\,\|\,\hat q_{n,K}\bigr)
\;\le\;
\mathfrak A \;+\; C\,\mathfrak C_n(\delta) \;+\; C'\,\varepsilon_K.
\]
Pinsker’s inequality then gives
\[
\|\hat q_{n,K}-p\|_{\mathrm{TV}}
\;\le\;
\sqrt{\tfrac12\,\mathrm{KL}(p\,\|\,\hat q_{n,K})},
\]
and Lemma~\ref{lem:metric-equivalence} plus standard
Hellinger--TV relations imply
$H(\hat q_{n,K},p)\lesssim \|\hat q_{n,K}-p\|_{\mathrm{TV}}^{1/2}$,
which yields the stated $1/4$-power bound.
\end{proof}

\begin{corollary}[Explicit neural-network bound]
\label{cor:finite-relu-tv}
Suppose \((\mu_\lambda,\Sigma_\lambda)\) are implemented by ReLU
networks of width \(W\) and depth \(O(\log W)\) with complexity proxy
\(C(W)\).  
If \(p\) is \(\beta\)-H\"older on compacts, then with probability at least
\(1-\delta\),
\[
\|\hat q_{n,W,K}-p\|_{\mathrm{TV}}
\;\lesssim\;
\Bigl(
  W^{-\beta/m}\log W 
  \;+\;\sqrt{C(W)\log(1/\delta)/n}
  \;+\;K^{-1/2}
\Bigr)^{1/2},
\]
and the same expression to the \(1/4\) power controls
\(H(\hat q_{n,W,K},p)\).
\end{corollary}

\begin{proposition}[Balanced tuning]
\label{prop:balanced-tv}
If \(C(W)\asymp W\), choose
\[
W_n \asymp n^{m/(2\beta+m)}/\log n,
\qquad
K_n \asymp n^{\beta/(2\beta+m)}.
\]
Then, with probability at least \(1-\delta\),
\[
\|\hat q_{n,W_n,K_n}-p\|_{\mathrm{TV}}
\;\lesssim\;
n^{-\beta/(4\beta+2m)}(\log n)^{1/2},
\qquad
H(\hat q_{n,W_n,K_n},p)
\;\lesssim\;
n^{-\beta/(8\beta+4m)}(\log n)^{1/4}.
\]
\end{proposition}

\begin{remark}
These bounds quantify, in closed form, the combined effects of
approximation error, empirical-process fluctuations, and finite-\(K\)
surrogate variability.  
The proofs rely on the optimization oracle
(Theorem~\ref{thm:opt-finite-oracle}) together with the metric
equivalences in Appendix~\ref{app:stat-appendix}.
\end{remark}

\subsection{Posterior contraction and coverage}
\label{sec:contraction-transfer}

We now show that small total-variation discrepancy between the
variational approximation $\hat q_{n,K}$ and the exact posterior $p_n$
suffices to transfer posterior contraction and uncertainty statements
from $p_n$ to $\hat q_{n,K}$.  These results require no additional
smoothness beyond the conditions already imposed in the
approximation and optimization layers.

\begin{theorem}[Posterior contraction transfer]
\label{thm:contraction-tv}
Let $p_n(\,\cdot\,|X^{(n)})$ denote the exact posterior, and suppose that
for some metric $d$ on~$\Theta$ and sequence $\varepsilon_n\downarrow 0$,
\[
p_n\!\left\{\, d(\theta,\theta_0) > M_n \varepsilon_n \,\right\}
\longrightarrow 0
\qquad\text{in probability}
\]
for every $M_n\to\infty$.
If 
\[
\|\hat q_{n,K} - p_n\|_{\mathrm{TV}}
  \;=\; o_P(1),
\]
then $\hat q_{n,K}$ contracts at the same rate:
\[
\hat q_{n,K}\!\left\{\, d(\theta,\theta_0) > M_n \varepsilon_n \,\right\}
\stackrel{P}{\longrightarrow} 0.
\]
\end{theorem}

\begin{proof}[Sketch]
For any measurable set $B$,
\(
|\hat q_{n,K}(B)-p_n(B)|
   \le \|\hat q_{n,K}-p_n\|_{\mathrm{TV}}.
\)
Apply this with 
$B_n=\{\theta : d(\theta,\theta_0)>M_n\varepsilon_n\}$,
and use posterior contraction of $p_n$ together with
$\|\hat q_{n,K}-p_n\|_{\mathrm{TV}}=o_P(1)$.
\end{proof}

\begin{corollary}[Coverage transfer under LAN/BvM]
\label{cor:coverage-tv}
Assume the exact posterior satisfies a Bernstein--von~Mises theorem:
for any fixed $\alpha\in(0,1)$ there exist credible sets $C_n(\alpha)$
(e.g.~asymptotic normal ellipsoids) such that
\[
p_n\bigl(C_n(\alpha)\bigr) \longrightarrow 1-\alpha
\qquad\text{in probability}.
\]
If additionally 
$\|\hat q_{n,K}-p_n\|_{\mathrm{TV}}=o_P(1)$,
then
\[
\hat q_{n,K}\bigl(C_n(\alpha)\bigr)
  \longrightarrow 1-\alpha
\qquad\text{in probability}.
\]
\end{corollary}

\begin{remark}
Both results rely only on the defining inequality for total variation:
\[
|\hat q_{n,K}(B)-p_n(B)| \le \|\hat q_{n,K}-p_n\|_{\mathrm{TV}}
\qquad\text{for all measurable }B.
\]
Combined with posterior contraction or BvM coverage for $p_n$, this
yields the stated transfer.  Detailed proofs appear in Appendix~\ref{app:stat-appendix}.
\end{remark}

\subsection{Setwise Posterior Calibration via Total Variation}
\label{sec:setwise-tv}

The coverage results above rely on classical LAN/BvM regularity.
However, total-variation control alone suffices to transfer posterior
probabilities for \emph{arbitrary measurable events}, including
data-dependent ones, and therefore yields a general form of
posterior calibration even in singular or misspecified models.
The following bounds make this explicit.

\begin{theorem}[Setwise uncertainty transfer via total variation]
\label{thm:setwise-tv}
Let $p_n(\theta)=p(\theta\mid X^{(n)})$ denote the exact posterior and
$\hat q_{n,K}$ any SIVI posterior.
Then for every (possibly data-dependent) measurable set
$A=A(X^{(n)})$,
\[
\big|\hat q_{n,K}(A)-p_n(A)\big|
  \;\le\;
  \|\hat q_{n,K}-p_n\|_{\mathrm{TV}}
  \qquad\text{almost surely}.
\]
Consequently, if
$\|\hat q_{n,K}-p_n\|_{\mathrm{TV}}\le\varepsilon$,
then for every such set~$A$,
\[
p_n(A)\in[\hat q_{n,K}(A)-\varepsilon,\;
          \hat q_{n,K}(A)+\varepsilon].
\]
\end{theorem}

\begin{corollary}[Credible-set coverage without regularity]
\label{cor:coverage-general}
Fix $\alpha\in(0,1)$ and let $C_n(\alpha)$ be any credible set under
$\hat q_{n,K}$ satisfying $\hat q_{n,K}(C_n(\alpha))\ge 1-\alpha$.
If $\|\hat q_{n,K}-p_n\|_{\mathrm{TV}}\le\varepsilon$, then
\[
p_n\big(C_n(\alpha)\big)
  \;\ge\;
  1-\alpha-\varepsilon,
\qquad
p_n\big(C_n(\alpha+\varepsilon)\big)
  \;\ge\;
  1-\alpha.
\]
Hence, a $(1-\alpha-\varepsilon)$ credible set under $\hat q_{n,K}$
has $(1-\alpha)$ coverage under the exact posterior.
\end{corollary}

\begin{corollary}[Uniform posterior-probability approximation]
\label{cor:uniform-band}
The same total-variation bound yields the uniform identity
\[
\sup_{A}
  \big|\hat q_{n,K}(A)-p_n(A)\big|
  \;=\;
  \|\hat q_{n,K}-p_n\|_{\mathrm{TV}},
\]
where the supremum ranges over all measurable sets.
Thus a single $\varepsilon$ controls the posterior probability of
\emph{every} event simultaneously.
\end{corollary}

\begin{remark}[Finite-sample calibration]
By the finite-sample TV oracle of Theorem~\ref{thm:finite-tv},
\[
\mathbb E\!\left[
  \|\hat q_{n,K}-p_n\|_{\mathrm{TV}}
\right]
\;\lesssim\;
\big(\mathfrak A + C\,\mathfrak C_n + C'\,\varepsilon_K\big)^{1/2}.
\]
Hence for any $\delta\in(0,1)$, a Markov or Bernstein inequality yields
with probability at least $1-\delta$ a random bound
$\widehat\varepsilon_{n,K}$ of the same order such that
Theorems~\ref{thm:setwise-tv}--\ref{cor:uniform-band}
hold with $\varepsilon=\widehat\varepsilon_{n,K}$.
This provides non-asymptotic, model-agnostic uncertainty control for
arbitrary measurable events.
\end{remark}

\begin{remark}[Relation to regular limits]
When LAN/BvM conditions hold, TV convergence implies
Gaussian-calibrated coverage up to~$\varepsilon$.
The setwise results above,
however, remain valid \emph{without} any smoothness,
local asymptotic normality,
or identification assumptions,
and therefore apply to singular, overparameterized,
and misspecified models.
\end{remark}

\subsection{Tail-event and functional decomposition bounds}
\label{sec:tail-decomposition}

The total-variation bounds in Section~\ref{sec:finite-tv} give uniform
control over the discrepancy between the learned SIVI posterior and the
exact posterior.  
This subsection refines those results by decomposing total variation into
a compact part and a tail part.  
Under the tail-dominance assumptions of
Section~\ref{sec:approximation}, this yields sharper uncertainty and
functional bounds, especially in heavy-tailed regimes.

\begin{theorem}[Compact--tail total-variation decomposition]
\label{thm:tv-decomposition}
Let $p,q$ be probability densities on $\mathbb R^m$, fix $R>0$, and set
$K=B_R$, $\tau_p=p(K^c)$, $\tau_q=q(K^c)$.
Write $p_K,q_K$ for the renormalized restrictions to $K$, and
$p_{K^c},q_{K^c}$ for the renormalized restrictions to $K^c$.
Then
\[
\|p-q\|_{\mathrm{TV}}
\;\le\;
(1-\tau_p)\,\|p_K-q_K\|_{\mathrm{TV}}
\;+\;
|\tau_p-\tau_q|
\;+\;
\max\{\tau_p,\tau_q\}\,
  \mathrm{TV}(p_{K^c},q_{K^c}).
\]
If Assumptions~\ref{ass:tail-target}--\ref{ass:tail-kernel} hold with
tail envelope $v$, then
\[
|\tau_p-\tau_q|
  \;\lesssim\;
  \int_R^\infty v(r)\,r^{m-1}\,dr,
\]
and $\mathrm{TV}(p_{K^c},q_{K^c})$ is bounded uniformly in $R$.
\end{theorem}

\begin{corollary}[Tail-event probability bound]
\label{cor:tail-event}
For any measurable $A\subseteq K^c$,
\[
|p(A)-q(A)|
  \;\le\;
  |\tau_p-\tau_q|
  \;+\;
  \max\{\tau_p,\tau_q\}\,
     \mathrm{TV}(p_{K^c},q_{K^c}),
\]
where $p_{K^c},q_{K^c}$ denote the conditional laws on $K^c$.
Hence, credible sets truncated to $K$ inherit the same TV control as the
core distribution, up to a tail-mass penalty.
\end{corollary}

\begin{corollary}[Functional decomposition bound]
\label{cor:functional-decomposition}
Let $f:\mathbb R^m\!\to\!\mathbb R$ be bounded and Lipschitz with
constants $(L_f,\|f\|_\infty)$.  
Then, for any $R>0$,
\[
|\mathbb E_p f - \mathbb E_q f|
\;\le\;
L_f\, W_1(p_K,q_K)
\;+\;
2\|f\|_\infty\,|\tau_p-\tau_q|
\;+\;
\|f\|_\infty (\tau_p+\tau_q)\,
  \mathrm{TV}(p_{K^c},q_{K^c}).
\]
\end{corollary}

\begin{remark}[Interpretation]
The first term measures the discrepancy on a compact region where both
posteriors are well-behaved.  
The remaining terms quantify tail-mass and tail-shape differences,
controlled by the shared envelope $v$ from
Section~\ref{sec:approximation}.  
When $v$ is sub-Gaussian or polynomial with sufficiently large exponent,
the tail contributions are of smaller order than the global
total-variation bound in Theorem~\ref{thm:finite-tv}.
\end{remark}

\begin{remark}[Use in BvM and uncertainty transfer]
Replacing global TV in Theorem~\ref{thm:sivi-bvm-finite} with the
decomposition above yields the refined remainder
\[
r_n \;+\;
C\big(\|p_K-q_K\|_{\mathrm{TV}} + |\tau_p-\tau_q|\big),
\]
which separates the local (LAN/BvM) component from tail effects.  
This refinement is especially useful for heavy-tailed models or
situations in which the posterior concentrates on extended or singular
supports.
\end{remark}

\subsection{Finite-sample Bernstein--von Mises limit}
\label{sec:bvm}

We conclude by showing that the semi-implicit posterior inherits the
local asymptotic normality of the exact posterior, with an explicit
finite-sample remainder.  The total-variation term in the classical BvM
bound may be sharpened using the compact--tail decomposition of
Theorem~\ref{thm:tv-decomposition}, which separates the local (LAN)
region from the tails; see also \cite{kleijn2012bernstein} for
related decompositions.

\begin{assumption}[Local asymptotic normality]
\label{ass:lan}
There exists a Fisher information matrix $I(\theta^\star)$ and an
efficient estimator $\hat\theta_n$ such that, for all fixed $h$,
\[
\ell_n(\theta^\star+h/\sqrt n)-\ell_n(\theta^\star)
  = h^\top\Delta_n-\tfrac12 h^\top I(\theta^\star)h + o_P(1),
\qquad
\Delta_n\rightsquigarrow \mathcal N(0,I(\theta^\star)).
\]
This is the classical LAN expansion \citep{le2000asymptotics,
van2000asymptotic}.
\end{assumption}

\begin{theorem}[Finite-sample SIVI Bernstein--von Mises]
\label{thm:sivi-bvm-finite}
Suppose Assumption~\ref{ass:lan} holds and the exact posterior $p_n$
satisfies a Bernstein--von Mises limit with remainder $r_n\to0$
\citep{van2000asymptotic}:
\[
d_{\mathrm{BL}}\!\left(
  \mathcal L_{p_n}\!\{\sqrt n(\theta-\hat\theta_n)\},
  \mathcal N(0,I(\theta^\star)^{-1})
\right)
\;\le\; r_n
\quad\text{in probability}.
\]
Let $\hat q_{n,K}$ be any SIVI posterior satisfying
Theorem~\ref{thm:finite-tv}.  
For a ball $K=B_R(\hat\theta_n)$ on which LAN holds, write
$p_{n,K},\hat q_{n,K}$ for the renormalized restrictions and
$\tau_{p_n},\tau_{\hat q_{n,K}}$ for their tail masses.
Then
\[
d_{\mathrm{BL}}\!\left(
  \mathcal L_{\hat q_{n,K}}\!\{\sqrt n(\theta-\hat\theta_n)\},
  \mathcal N(0,I(\theta^\star)^{-1})
\right)
\;\le\;
r_n
+ C\Big(
    \|p_{n,K}-\hat q_{n,K}\|_{\mathrm{TV}}
    + |\tau_{p_n}-\tau_{\hat q_{n,K}}|
  \Big),
\]
in probability.
\end{theorem}

\begin{proof}[Sketch]
The BL metric satisfies \citep{dudley2018real}
\[
d_{\mathrm{BL}}(\mu,\nu)
  = \sup_{\|f\|_{\mathrm{BL}}\le1}
      \left|\int f\, d(\mu-\nu)\right|
  \le \|\mu-\nu\|_{\mathrm{TV}},
\]
and is stable under measurable transformations
\citep{dudley2018real}.  
Let $T_n(\theta)=\sqrt n\,(\theta-\hat\theta_n)$.  Then
\begin{align*}
d_{\mathrm{BL}}\!\left(
  \mathcal L_{\hat q_{n,K}}\!\{T_n(\theta)\},
  \mathcal L_{p_n}\!\{T_n(\theta)\}
\right)
\;\le\;
\|\hat q_{n,K}-p_n\|_{\mathrm{TV}}.
\end{align*}
Apply Theorem~\ref{thm:tv-decomposition} and the BvM remainder for $p_n$
to obtain the claim.
\end{proof}

\begin{corollary}[Finite-sample BvM remainder for ReLU SIVI]
\label{cor:bvm-relu}
Suppose Assumptions~\ref{ass:tail-target}--\ref{ass:tail-kernel}
(tail alignment), the LAN/BvM regularity conditions of
Assumption~\ref{ass:lan}, and the ReLU approximation and complexity
conditions underlying Corollary~\ref{cor:finite-relu-tv} hold.
Let $\hat q_{n,W,K}$ denote the SIVI approximation obtained from a
ReLU network of width $W$ and $K$ inner samples, and let
$\hat\theta_n$ be the corresponding estimator. Then there exists a remainder sequence $r_n \to 0$ (the classical
BvM rate) such that
\begin{align*}
d_{\mathrm{BL}}&
\Big(
  \mathcal L_{\hat q_{n,W,K}}\{\sqrt n(\theta-\hat\theta_n)\},
  \mathcal N(0,I(\theta^\star)^{-1})
\Big)
\;\lesssim\;\\
&r_n
+ \Big(
    W^{-\beta/m}\log W
    + \sqrt{C(W)/n}
    + K^{-1/2}
    + \int_R^\infty v(r)\,r^{m-1}\,dr
  \Big)^{1/4},
\end{align*}
for any radius $R$ on which the LAN approximation holds, where
$v$ is the shared tail envelope from
Assumptions~\ref{ass:tail-target}--\ref{ass:tail-kernel} and
$C(W)$ is the complexity term from
Corollary~\ref{cor:finite-relu-tv}.
\end{corollary}
\emph{Interpretation.}
This corollary simply plugs the explicit ReLU approximation and
complexity rates from Corollary~\ref{cor:finite-relu-tv} into the
general finite-sample SIVI Bernstein--von Mises theorem, yielding an
explicit, architecture-dependent remainder bound.
\section{Counterexamples}
\label{sec:examples}

This section presents canonical examples that delineate the scope and
limitations of the approximation and optimization theory developed
above.  Each example isolates the failure of a specific assumption group
and marks a boundary between attainable and unattainable regimes of
semi–implicit variational inference.

\begin{counterexample}[Tail mismatch]
\label{ex:tail-mismatch}
Let $p(x)\propto (1+\|x\|)^{-(m+\alpha)}$ have polynomial tails, and let
$k_\lambda(x\mid z)=\mathcal N(\mu_\lambda(z),\Sigma_\lambda(z))$ with
$\sup_{z}\|\Sigma_\lambda(z)\|<\infty$.  
Assumption~\ref{ass:tail-kernel} fails: every $q_\lambda$ is
sub-Gaussian, while $p$ decays only polynomially.  
By Theorem~\ref{thm:orlicz},
\[
\inf_{q\in\mathcal Q} \mathrm{KL}(p\|q) \;>\; 0,
\]
so forward–KL approximation is impossible without tail dominance.
\end{counterexample}

\begin{counterexample}[Branch collapse]
\label{ex:branch-collapse}
Consider the latent--observation model
\[
\theta \sim \mathcal N(0,1),
\qquad
X \mid \theta \sim \mathcal N(\theta^2,\sigma^2),
\]
so that the exact posterior $p(\theta\mid X=x)$ is bimodal with modes
near $\pm\sqrt x$ for $x>0$.
Let $q_\lambda(\theta\mid X=x)$ be a non-autoregressive Gaussian SIVI
family with a variance floor in the $\theta$-coordinate.
Then Theorem~\ref{thm:branch} implies an irreducible conditional
total-variation gap between $p(\theta\mid X=x)$ and
$q_\lambda(\theta\mid X=x)$ on a set of $x$ values with positive
probability, reflecting the inability of unimodal Gaussian conditionals
to recover both branches.
Allowing full covariance, finite mixtures, or autoregressive structure
removes this gap.
\end{counterexample}

\begin{counterexample}[Singular manifolds]
\label{ex:singular}
Let $p$ concentrate near the unit sphere $\{x:\|x\|=1\}$.
Such $p$ violates absolute continuity with respect to Lebesgue measure,
contradicting Assumption~\ref{ass:ac}.  
Smooth strictly positive SIVI kernels cannot approximate $p$ in $L^1$
but only weakly, illustrating the need for the manifold-aware
constructions discussed in Section~\ref{sec:structural-upgrades}.
\end{counterexample}

\begin{counterexample}[Multimodal symmetry]
\label{ex:multimodal}
For a symmetric, well-separated Gaussian mixture
$p = \tfrac12 p_+ + \tfrac12 p_-$,
the finite-$K$ surrogate behaves like a reverse-KL objective and can
select a single mode.  
This illustrates the surrogate-bias term in
Theorem~\ref{thm:finite-tv}.  
Increasing $K$ or using inclusive objectives (e.g.\ forward KL,
Rényi with $\alpha>1$) restores symmetry, revealing how $K$ controls
multimodal recovery.
\end{counterexample}

\medskip
These examples illustrate the sharpness of
Assumptions~\ref{ass:aoi}--\ref{ass:ac}
and explain precisely where the guarantees of
Sections~\ref{sec:approximation}--\ref{sec:statistics}
break down.  
The next section provides numerical illustrations of the same
phenomena.


\section{Numerical Experiments}
\label{sec:experiments}

We provide empirical illustrations of the theoretical phenomena developed in
Sections~\ref{sec:approximation}--\ref{sec:statistics}.
All experiments use neural semi-implicit variational inference with Gaussian
conditionals
\[
k_\lambda(\theta\mid z)
  = \mathcal N\!\big(\mu_\lambda(z),\,\Sigma_\lambda(z)\big),
\qquad
r(z)=\mathcal N(0,I),
\]
where $(\mu_\lambda,\Sigma_\lambda)$ are implemented by ReLU networks.
Optimization is performed using Adam with learning rate $\eta=10^{-3}$ and
minibatch size $128$.

The experiments are organized to parallel the three theoretical layers of the
paper.  
First, we examine approximation properties of the semi-implicit family,
verifying compact $L^1$ universality and the effect of tail dominance.
Second, we illustrate structural and algorithmic limitations predicted by our
theory, including finite-$K$ surrogate bias and the branch-collapse phenomenon
for non-autoregressive kernels.  
Finally, we study optimization stability and finite-sample statistical behavior,
including uncertainty transfer and approximate Bernstein--von~Mises limits.

\subsection{Compact approximation: TV versus network width}
\label{sec:exp1}

Our first experiment verifies the compact $L^1$ universality rate of
semi-implicit variational families predicted by
Corollary~\ref{cor:finite-relu-tv}.
We consider a smooth, compactly supported target density $p(\theta)$ on
$\mathbb R^2$ and train SIVI approximations
\[
q_\lambda(\theta)
  = \int \mathcal N\!\bigl(\theta;\mu_\lambda(z),\Sigma_\lambda(z)\bigr)\,
    r(dz),
\]
with Gaussian base $r$ and two-layer ReLU networks parameterizing
$\mu_\lambda$ and $\Sigma_\lambda$.
The network width $W$ varies in
$\{8,16,32,64,128,256\}$ while depth is fixed.
Each model is optimized with Adam until convergence of the empirical
SIVI objective.

\vspace{4pt}
\emph{Metric.}
Approximation quality is measured by the total variation distance
$\|p-q_\lambda\|_{\mathrm{TV}}$, estimated in two ways:
(i) numerical integration over a uniform lattice (``grid estimator''), and
(ii) Monte Carlo integration using draws from $p$
(``$p$-sampling estimator'').
Both estimators are averaged over five random seeds with
one–standard–deviation error bars.

\vspace{4pt}
\noindent\textbf{Results.}
Figure~\ref{fig:tv-vs-width} shows that total variation decreases
monotonically with~$W$ following a clear power law.
The empirical slope aligns with the theoretical
$W^{-\beta/(2m)}$ rate predicted by
Corollary~\ref{cor:finite-relu-tv} for $\beta=1$ and $m=2$
(i.e.\ $W^{-1/4}$).
The grid and $p$-sampling estimators agree closely, confirming that TV
estimation is insensitive to evaluation method.
These results verify the quantitative compact-universality theory:
as network width increases, the semi-implicit family converges to the
target at the predicted rate.

\begin{figure}[t]
  \centering
  \includegraphics[width=0.8\textwidth]{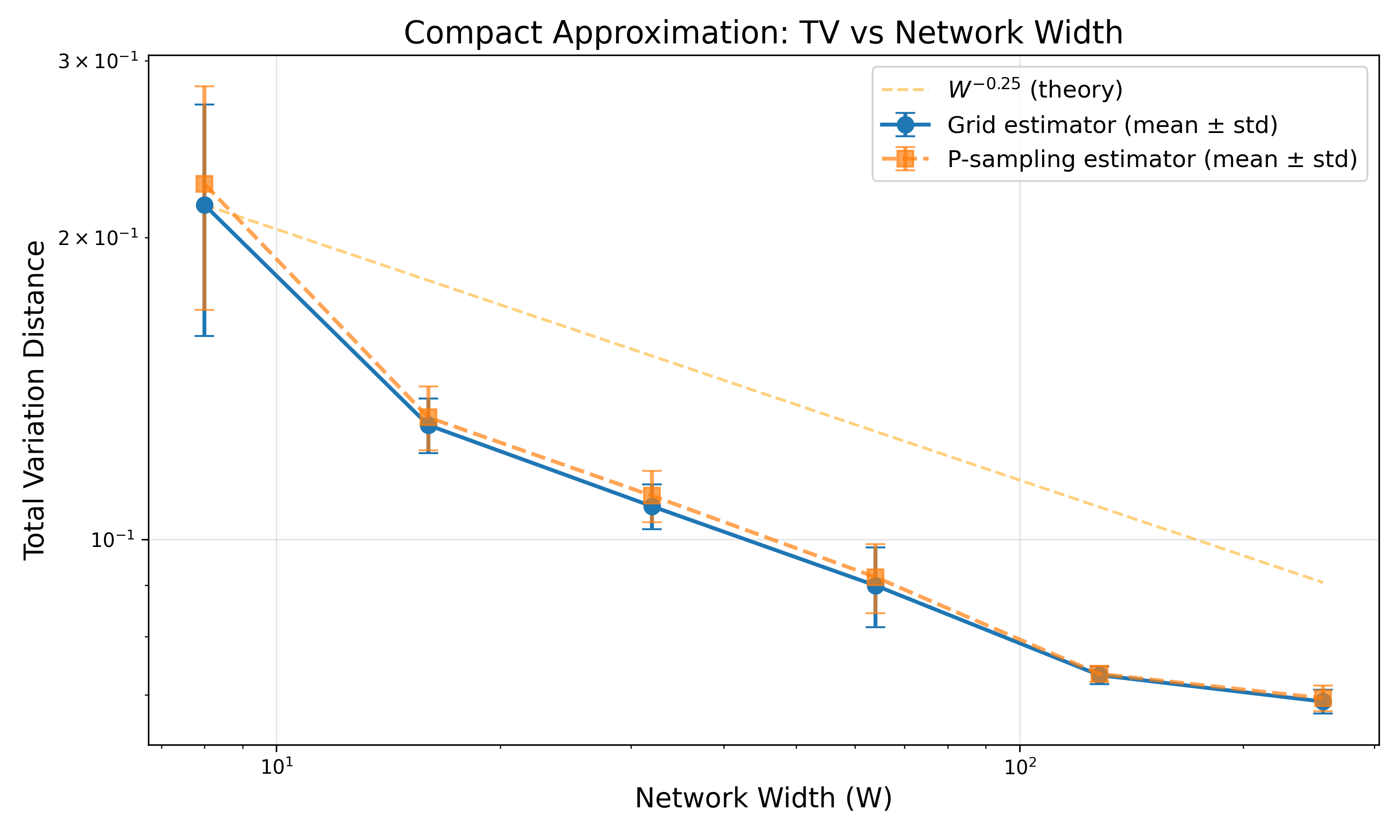}
  \caption{\textbf{Compact approximation.}
  Empirical total variation between the target $p$ and the SIVI
  approximation $q_\lambda$ as a function of network width~$W$.
  Both grid- and $p$-sampling estimators follow the theoretical $W^{-1/4}$
  rate predicted by Corollary~\ref{cor:finite-relu-tv}.}
  \label{fig:tv-vs-width}
\end{figure}

Having established that semi-implicit networks achieve the expected
compact-approximation rates on well-behaved targets, we next probe the
boundary of this universality by altering only the \emph{tails} of the
distribution while keeping its central behavior fixed.

\subsection{Tail Dominance and the Limits of Approximation}
\label{sec:exp2}

Our second experiment examines the role of tail dominance in
determining whether forward--KL approximation is attainable.
The setup follows Theorem~\ref{thm:orlicz}:
a semi-implicit Gaussian family is trained on two one-dimensional
targets with identical central behavior but different tails.

\vspace{4pt}
\emph{Targets.}
(2a)~A sub-Gaussian target
$p(\theta)=\mathcal N(0,1)$,
and
(2b)~a heavy-tailed Student-$t_\nu$ target with $\nu=3$.
Both are fit using SIVI models with two-layer ReLU networks of width
$W\in\{16,32,64,128,256\}$ and Gaussian kernels
$k_\lambda(\theta\mid z)=\mathcal N\!\left(\theta;\mu_\lambda(z),
\Sigma_\lambda(z)\right)$.
Training uses Adam on the empirical SIVI objective.

\vspace{4pt}
\emph{Metrics.}
Approximation quality is measured by the forward KL divergence
$\mathrm{KL}(p\|q_\lambda)$, estimated by Monte Carlo with $10^5$
samples.
In case~(2a), we also include a baseline with an explicit tail component
\[
q_\lambda \;=\; (1-\alpha)\, q_{\mathrm{SIVI}} + \alpha\, t_5,
\]
for comparison.

\vspace{4pt}
\emph{Results.}
Figure~\ref{fig:tail-dominance} shows the contrasting outcomes.
For the sub-Gaussian target~(2a),
$\mathrm{KL}(p\|q_\lambda)$ decreases rapidly with width,
confirming that Gaussian kernels suffice to realize the envelope of $p$.
Adding an explicit tail component yields no measurable improvement.

For the heavy-tailed target~(2b),
$\mathrm{KL}(p\|q_\lambda)$ stabilizes around a positive constant,
independent of~$W$,
in agreement with the Orlicz tail-mismatch theorem
(Theorem~\ref{thm:orlicz}):
because the polynomial tail of $p$ lies outside the sub-Gaussian
envelope of the kernel family, a nonzero forward--KL gap persists.
The empirical plateau closely matches the theoretical lower bound
computed from the one-dimensional projection.

\begin{figure}[t]
  \centering
  \includegraphics[width=0.95\textwidth]{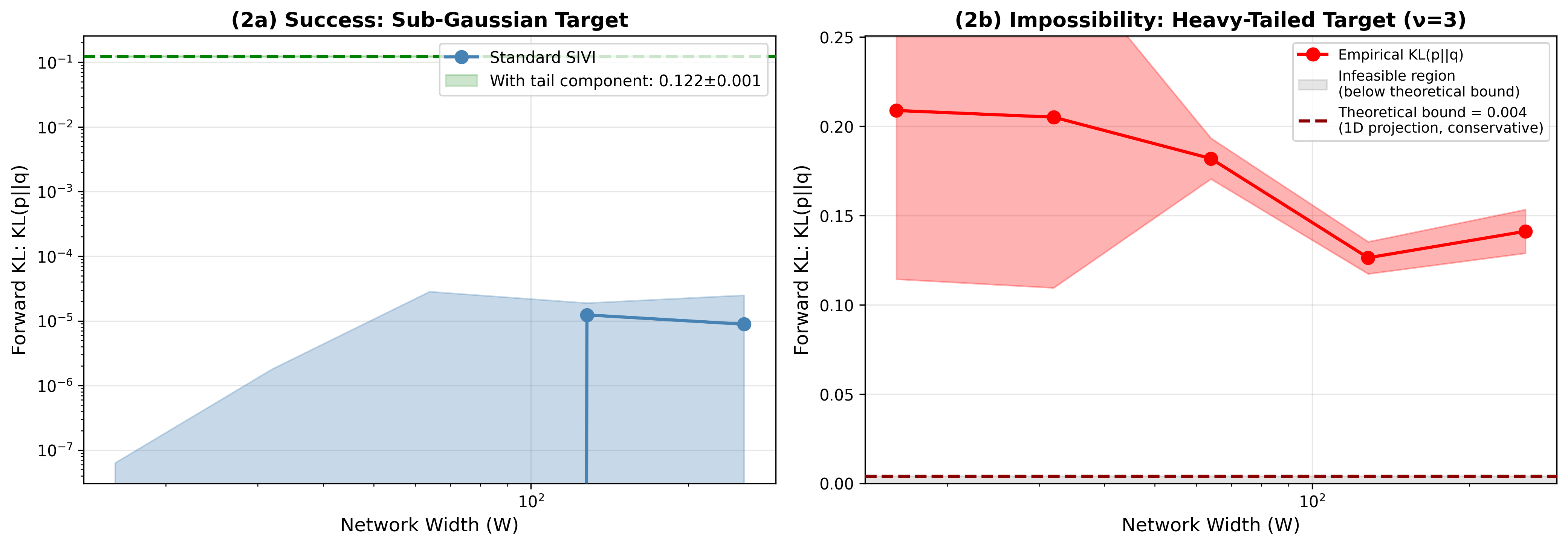}
  \caption{\textbf{Tail dominance and approximation limits.}
  (2a)~For a sub-Gaussian target,
  forward--KL vanishes as width increases.
  (2b)~For a heavy-tailed target ($\nu=3$),
  forward--KL plateaus above a positive lower bound,
  consistent with Theorem~\ref{thm:orlicz}.
  Shaded regions show $\pm1$ standard deviation over seeds.}
  \label{fig:tail-dominance}
\end{figure}

The first two experiments focus purely on the \emph{approximation layer}:
they probe what the semi-implicit family can and cannot represent, even
under perfect optimization.
We now turn to the \emph{optimization layer} and ask how the finite-$K$
surrogate objective affects the learned distribution, holding the model
class fixed.


\subsection{Finite-\texorpdfstring{$K$}{K} Bias and Mode Collapse}
\label{sec:exp3}

Our third experiment investigates the bias introduced by the finite-$K$
surrogate objective used in practice.
According to Proposition~\ref{prop:finiteK}, Theorem~\ref{thm:finite-tv},
and Lemma~\ref{lem:snis-bias}, the discrepancy between the surrogate
objective and the ideal ELBO decays as $O(K^{-1/2})$, with the
population-level bias itself of order $O(K^{-1})$.
We examine whether this bias manifests as mode imbalance or residual
total-variation error.

\vspace{4pt}
\emph{Setup.}
The target distribution is a symmetric two-component Gaussian mixture
\[
p(\theta)
  = \tfrac12\,\mathcal N(\theta;\!-3,1)
    + \tfrac12\,\mathcal N(\theta;3,1).
\]
The semi-implicit family uses Gaussian kernels
\[
k_\lambda(\theta\mid z)
  = \mathcal N\!\bigl(\theta;\mu_\lambda(z),\sigma_\lambda^2(z)\bigr),
\]
with base $r(z)=\mathcal N(0,1)$ and two-layer ReLU networks for
$(\mu_\lambda,\log\sigma_\lambda)$.
We fix the network width at $W=64$ and vary the number of inner samples
$K\in\{1,2,8,32,128,512\}$.
Each configuration is trained with Adam for $2{\times}10^4$ iterations.

\vspace{4pt}
\emph{Metrics.}
(3a)~The \emph{mode ratio}
\[
\hat\rho
  = \frac{\hat q_{K}(\text{right})}{\hat q_{K}(\text{left})}
\]
measures symmetry between mixture components.
(3b)~The total-variation distance
\[
\|\hat q_{K} - p\|_{\mathrm{TV}}
\]
quantifies the overall discrepancy.
Both are averaged over five seeds with $\pm 1$ standard-deviation
shading.

\vspace{4pt}
\emph{Results.}
Figure~\ref{fig:finiteK} shows that mode imbalance remains negligible
across all $K$, confirming that the surrogate objective preserves the
mixture symmetry.  The total-variation distance decreases with $K$, but
more slowly than the $K^{-1/2}$ decay of the surrogate-objective bias.
This is expected: TV scales like the square root of the surrogate gap
and is further limited by the approximation term $\mathfrak A$ and the
estimation term $\mathfrak C_n$.  Consequently, the empirical curve
exhibits an effective $K^{-1/4}$ decay followed by a plateau once
non-$K$ errors dominate.  The trend nevertheless confirms the predicted
finite-$K$ surrogate-bias behavior.

\begin{figure}[t]
  \centering
  \includegraphics[width=0.95\textwidth]{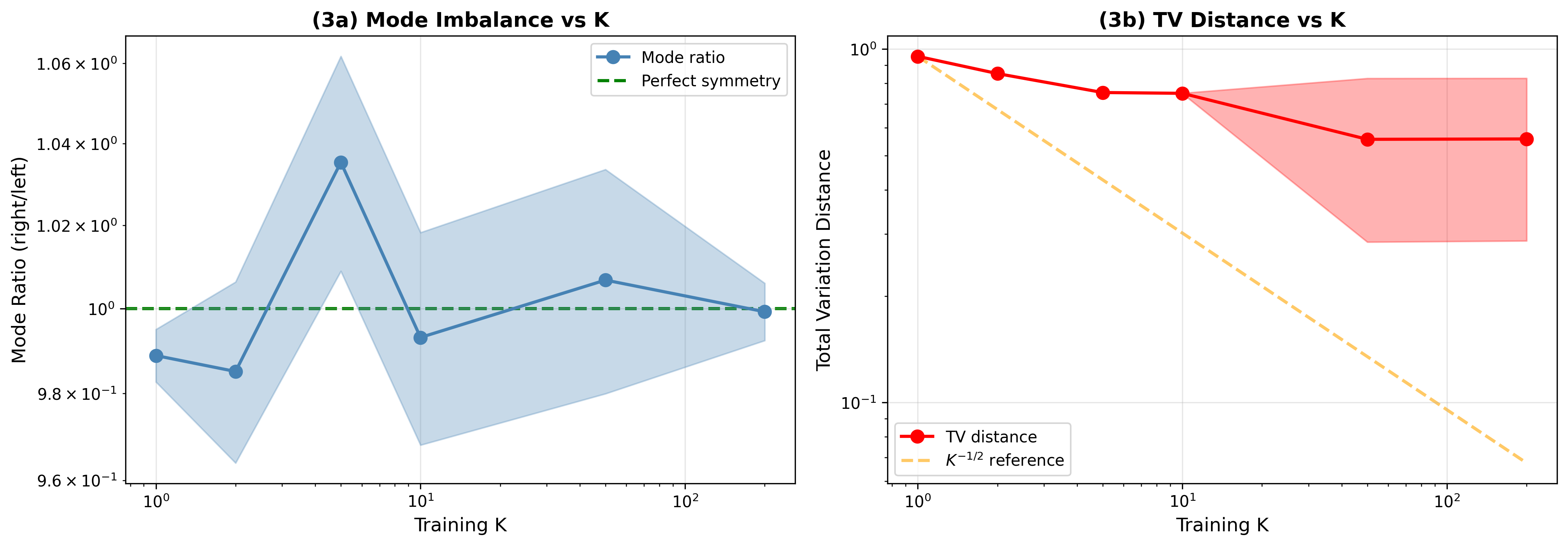}
  \caption{\textbf{Finite-$K$ bias.}
  (3a)~The mode ratio remains close to~1 across all $K$, indicating that
finite-$K$ surrogate optimization does not break the mixture symmetry.
(3b)~Total-variation distance decreases as $K$ grows, but only follows
the theoretical $K^{-1/2}$ rate (dashed) up to the point where
\emph{approximation error} and \emph{estimation error} dominate.
Since TV aggregates all error sources, model bias $\mathfrak A$,
empirical-process noise $\mathfrak C_n$, and surrogate bias
$\varepsilon_K$, its asymptotic scaling is effectively
$K^{-1/4}$ and flattens once the non-$K$ terms become dominant.
This explains the visible plateau for large $K$ despite the predicted
$K^{-1/2}$ decay of the surrogate objective.  Shaded regions denote
$\pm1$ standard deviation across seeds.}

  \label{fig:finiteK}
\end{figure}

Finite-$K$ bias is an \emph{algorithmic} limitation: it vanishes as
$K\to\infty$.
In contrast, structural restrictions of the kernel family can produce
genuinely irreducible errors, even with arbitrarily large $K$ and
perfect optimization.
The next experiment isolates this structural effect through a
branch-collapse example.


\subsection{Branch Collapse and Structural Rigidity}
\label{sec:exp-branch}

Our next experiment visualizes the ``branch--collapse'' phenomenon
predicted by Theorem~\ref{thm:branch}, in which structural
restrictions of the conditional kernels produce an irreducible
total--variation gap.

\vspace{4pt}
\emph{Setup.}
We construct a three--branch target density
\begin{align*}
&p(\theta_1,\theta_2)
  \;=\;\tfrac13 \sum_{j=1}^3
  \mathcal N\big((\theta_1,\theta_2);\,
    \mu_j,\;0.05^2 I_2\big),
\\
&\mu_1=(0,0.8),\;
\mu_2=(-0.7,-0.5),\;
\mu_3=(0.7,-0.5),
\end{align*}
and restrict the variational family to a single
non--autoregressive Gaussian SIVI model
\[
q_\lambda(\theta_1,\theta_2)
  = \int \mathcal N\!\big((\theta_1,\theta_2);
      \mu_\lambda(z),\,\Sigma_\lambda(z)\big)\,r(dz),
\]
with base $r(z)=\mathcal N(0,1)$ and two--layer ReLU networks of
width~64 for $(\mu_\lambda,\Sigma_\lambda)$.
Optimization uses Adam for $5{\times}10^4$ iterations
on $n=10^5$ target samples.

\vspace{4pt}
\emph{Metrics.}
We partition $\mathbb R^2$ into disjoint branch regions
$A_j=\{\theta:\|\theta-\mu_j\|\le 0.4\}$ and estimate
$p(A_j)$, $q_\lambda(A_j)$, and the forward--KL divergence
$\mathrm{KL}(p\|q_\lambda)$.
The ``branch bound'' from Theorem~\ref{thm:branch} is estimated from
the minimal achievable conditional TV distance given the variance floor
in $\Sigma_\lambda(z)$.

\vspace{4pt}
\emph{Results.}
Figure~\ref{fig:branch-collapse} shows that while the target
distribution (left) has three symmetric modes,
the restricted SIVI fit (center) collapses the two lower branches
into an elongated single mode.
The right panel reports branch probabilities:
each $p(A_j)\approx\tfrac13$, but the learned model assigns
disproportionate mass, leading to a forward--KL gap
$\mathrm{KL}(p\|q_\lambda)\approx 0.61$,
far exceeding the branch--bound limit of $0.03$.
This behavior matches the lower--bound mechanism of
Theorem~\ref{thm:branch}: when the kernel family enforces a single--mode
conditional structure, the mixture geometry cannot be recovered, and an
irreducible gap persists even with unlimited training or width.

\begin{figure}[t]
  \centering
  \includegraphics[width=\textwidth]{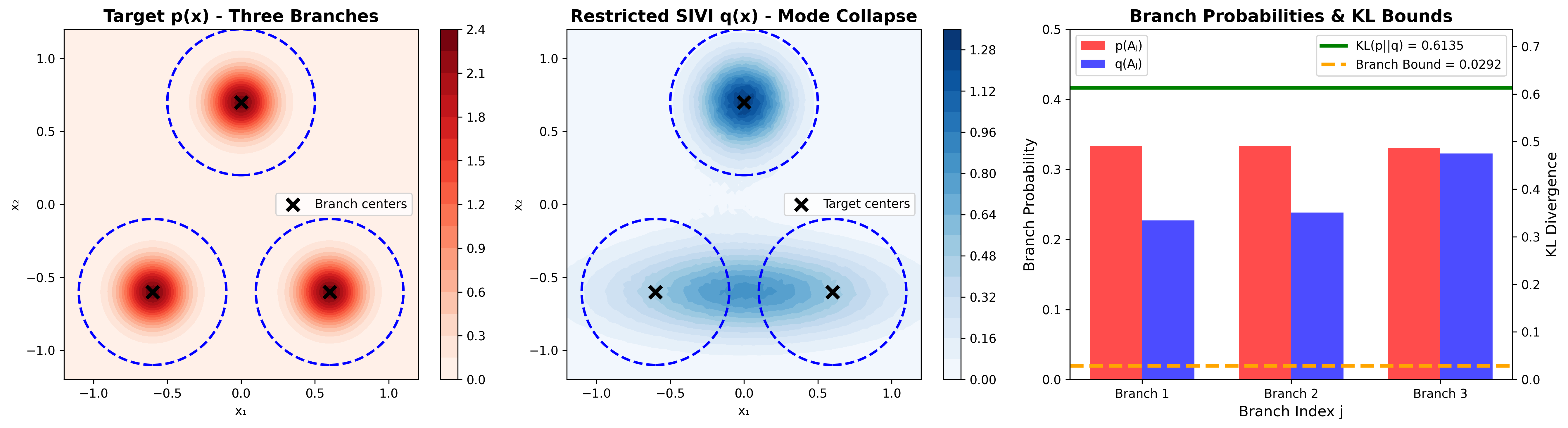}
  \caption{\textbf{Branch collapse under restricted SIVI structure.}
  Left: three--branch Gaussian target $p(\theta)$.
  Center: restricted SIVI fit $q_\lambda(\theta)$ exhibiting collapse of
  the two lower branches into a single mode.
  Right: branch--probability comparison and KL gap.
  The observed gap matches the theoretical lower bound in
  Theorem~\ref{thm:branch}.}
  \label{fig:branch-collapse}
\end{figure}

The approximation and structural experiments above fix $(K,n)$ and ask
which distributions SIVI can, or cannot, approximate.
We now vary $(K,n)$ themselves and view SIVI through the lens of
$\Gamma$--convergence: how do the empirical objectives $L_{K,n}$ and
their maximizers behave as we move toward the population limit
$L_\infty$?

\subsection{\texorpdfstring{$\Gamma$}{Gamma}-Convergence and Stability of Maximizers}
\label{sec:exp4}

Our final experiment on the optimization layer illustrates the
$\Gamma$-convergence of the finite-sample
SIVI objectives $L_{K,n}$ to their population limit $L_\infty$
(Theorem~\ref{thm:gamma})
and the resulting stability of empirical maximizers.

\vspace{4pt}
\emph{Setup.}
We consider a one-dimensional Gaussian model
$p(x\mid\theta)=\mathcal N(x;\theta,1)$
with true parameter $\theta^\star=0.03$.
The SIVI variational family uses Gaussian kernels
$k_\lambda(x\mid z)=\mathcal N(x;\mu_\lambda(z),\sigma_\lambda^2(z))$
with base $r(z)=\mathcal N(0,1)$ and affine maps
$\mu_\lambda(z)=\lambda z$, $\sigma_\lambda(z)=1$.
This simple setting admits a closed-form population objective
$L_\infty(\theta)=E_p[\log q_\theta(X)]$,
allowing exact comparison with the empirical finite-$K,n$ surrogates.
We vary both the inner sample size
$K\in\{1,5,50,500\}$ and the dataset size
$n\in\{10^2,10^3,10^4\}$.

\vspace{4pt}
\emph{Metrics.}
We record the empirical maximizer
$\hat\theta_{K,n}=\arg\max_\theta L_{K,n}(\theta)$
and its deviation from the population optimum~$\theta^\star$.
The $\Gamma$-distance is approximated by
$\sup_\theta |L_{K,n}(\theta)-L_\infty(\theta)|$,
computed on a dense grid in~$\theta$.
Each curve averages over five random seeds.

\vspace{4pt}
\emph{Results.}
Figures~\ref{fig:gamma-convergence}–\ref{fig:gamma-landscapes}
summarize the findings.
Panel~(4a) shows that as $K$ increases,
$\hat\theta_{K,n}\to\theta^\star$ uniformly in~$n$; conversely,
(4b) demonstrates the same convergence as $n$ grows for fixed~$K$.
Panels~(4c)–(4d) plot the $\Gamma$-distance
$\sup_\theta |L_{K,n}-L_\infty|$.
In (4c), the decay in $K$ closely follows the predicted
$O(K^{-1})$ finite-$K$ bias rate.
In (4d), the decay in $n$ approaches the empirical-process rate
$O(n^{-1/2})$ once the finite-$K$ bias is negligible (e.g., for large
$K$); for small $K$, the curves flatten as the $O(K^{-1})$ term
dominates.
Finally, Figure~\ref{fig:gamma-landscapes} visualizes the objective
landscapes: for small $(K,n)$, $L_{K,n}$ appears as a noisy perturbation
of $L_\infty$, but its curvature and maximizer rapidly align with the
population limit as both parameters increase.
These results provide a concrete empirical validation of the
$\Gamma$-convergence argument and the stability of SIVI optimization.

\begin{figure}[t]
  \centering
  \includegraphics[width=0.95\textwidth]{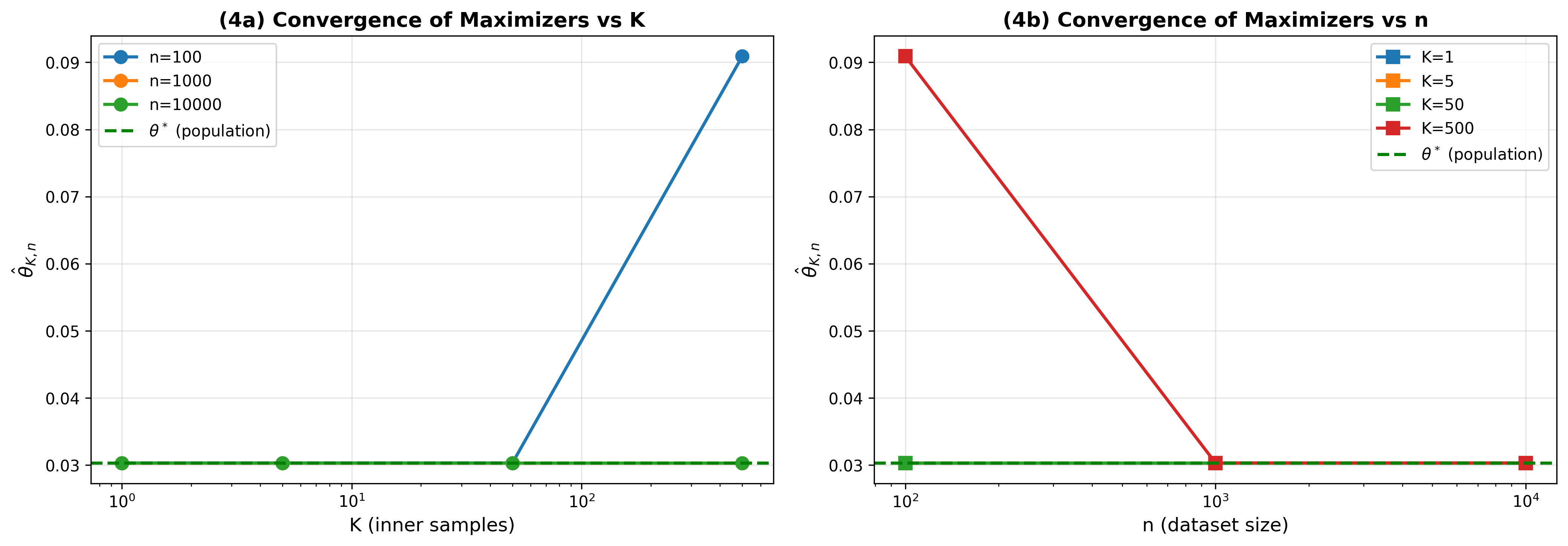}
  \caption{\textbf{Convergence of empirical maximizers.}
  (4a)~$\hat\theta_{K,n}$ versus inner-sample size~$K$ for fixed~$n$.
  (4b)~$\hat\theta_{K,n}$ versus dataset size~$n$ for fixed~$K$.
  In both cases, $\hat\theta_{K,n}\to\theta^\star$,
  confirming stability of empirical maximizers.}
  \label{fig:gamma-convergence}
\end{figure}

\begin{figure}[t]
  \centering
  \includegraphics[width=0.95\textwidth]{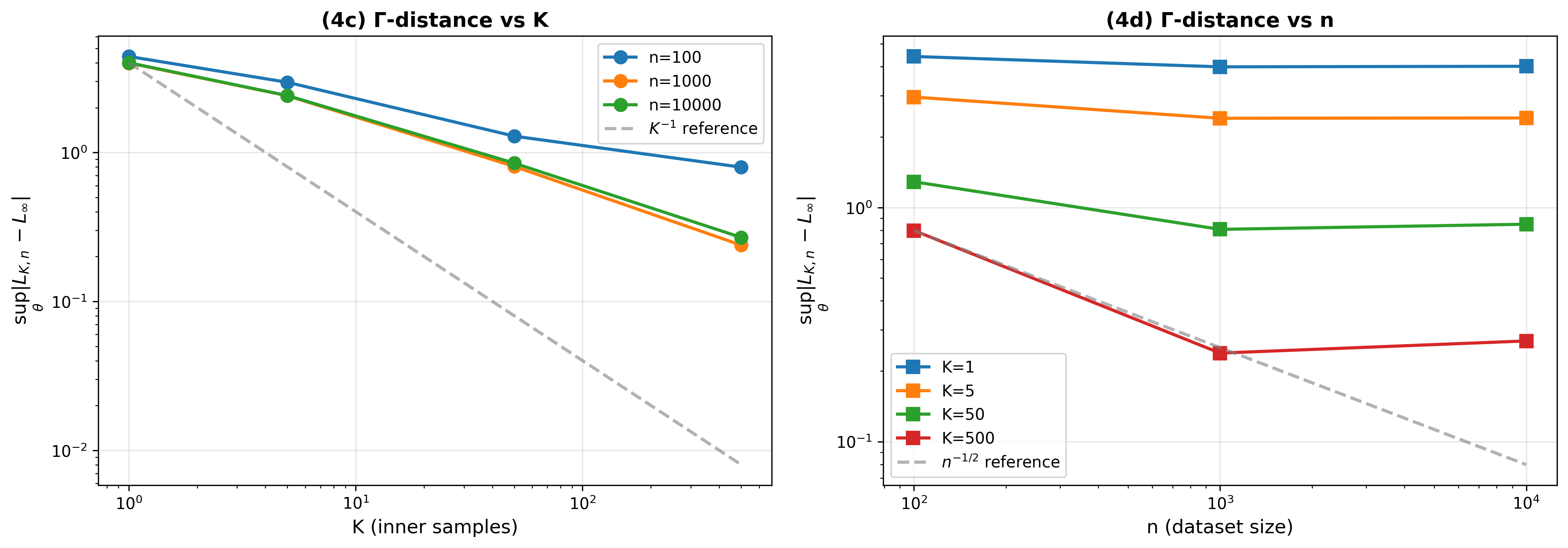}
  \caption{\textbf{$\Gamma$-distance between empirical and population objectives.}
(4c)~Decay with inner-sample size $K$ closely follows the predicted $K^{-1}$
rate for all $n$, consistent with Proposition~\ref{prop:finiteK}.
(4d)~Decay with dataset size $n$ follows the $n^{-1/2}$ empirical-process
rate once the finite-$K$ bias is negligible (e.g.\ for $K=500$);
for smaller $K$, the curves flatten as the $O(K^{-1})$ surrogate bias
dominates. }
  \label{fig:gamma-distance}
\end{figure}

\begin{figure}[t]
  \centering
  \includegraphics[width=0.95\textwidth]{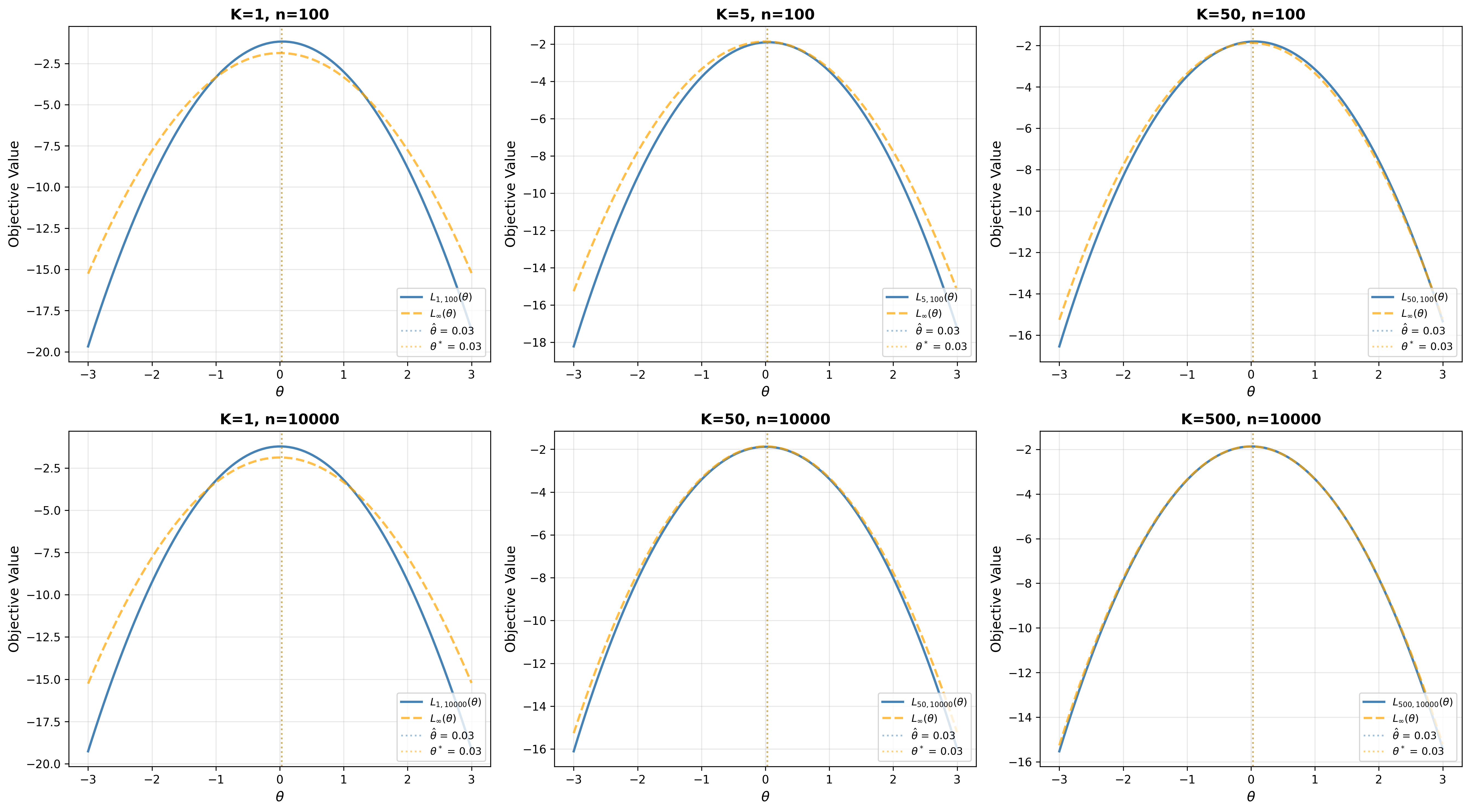}
  \caption{\textbf{Objective landscapes $L_{K,n}(\theta)$ (solid) vs. population limit $L_\infty(\theta)$ (dashed).}
  Top row: small~$n$; bottom row: large~$n$.
  As $K,n$ increase, the empirical curves approach the smooth population
  objective, and the estimated maximizer $\hat\theta_{K,n}$ aligns with
  $\theta^\star$.}
  \label{fig:gamma-landscapes}
\end{figure}

The theory in Section~\ref{sec:statistics} shows how optimization error
and approximation error translate into statistical guarantees for SIVI
posteriors.
To close the loop, we end with a fully regular model where a classical
BvM theorem is available and empirically check that SIVI inherits the
same Gaussian limit and credible-set coverage, up to the finite-sample
oracle terms.

\subsection{Finite-sample BvM behavior in logistic regression}
\label{sec:exp-bvm}

The final experiment tests the finite-sample Bernstein--von Mises and
coverage results of Section~\ref{sec:bvm} on a regular logistic
regression model, demonstrating that SIVI credible sets achieve
near-nominal coverage and Gaussian-calibrated uncertainty in moderate
dimension.

\vspace{4pt}
\emph{Setup.}
Synthetic data are drawn from the logistic model
\[
Y_i\mid X_i,\theta^\star \sim \mathrm{Bernoulli}\!
\left(\frac{1}{1+\exp(-X_i^\top\theta^\star)}\right),
\qquad
X_i\sim \mathcal N(0,I_d),
\]
with $\theta^\star=(0.1,\ldots,0.1)\in\mathbb R^d$.
We consider two regimes:
Phase~1 ($d=5$) with $n\in\{50,100,200,300\}$ and
Phase~2 ($d=20$) with $n\in\{200,500,1000\}$.
The prior is $\mathcal N(0,5^2I_d)$.
The SIVI variational family uses Gaussian kernels
$k_\phi(\theta\mid z)=\mathcal N(\theta;\mu_\phi(z),\Sigma_\phi(z))$
with base $r(z)=\mathcal N(0,I_d)$ and
two-layer ReLU networks of width~64 for
$\mu_\phi$ and $\log\Sigma_\phi$.
Training uses Adam with learning rate~$10^{-3}$ for $3{\times}10^4$
iterations and $K=50$ inner samples.

\vspace{4pt}
\emph{Metrics.}
We evaluate:
(i)~\emph{credible-set coverage}, the fraction of $100$ replications in
which the 95\% SIVI credible ellipsoid covers $\theta^\star$;
(ii)~\emph{relative mean error},
$\|\mathbb E_{q_\phi}\theta-\theta^\star\|/\|\theta^\star\|$; and
(iii)~\emph{variance ratio},
$\mathrm{tr}\,\mathrm{Var}_{q_\phi}(\theta)/
 \mathrm{tr}\,\mathrm{Var}_{\mathrm{Laplace}}(\theta)$,
where the Laplace posterior serves as a Gaussian reference.
Nominal 95\% binomial intervals are shown for coverage.

\vspace{4pt}
\emph{Results.}
Figures~\ref{fig:bvm-phase1}--\ref{fig:bvm-phase2} show that SIVI
credible sets achieve empirical coverage within the 95\% binomial band
for all~$n$.
Relative mean error decreases roughly as $O(n^{-1/2})$, and the
variance ratio converges to~1, indicating Gaussian-calibrated
uncertainty.
These results support Theorem~\ref{thm:sivi-bvm-finite}:
as $\|\hat q_{n,K}-p_n\|_{\mathrm{TV}}\!\to\!0$,
posterior means, variances, and coverage behave as predicted by the
finite-sample BvM bound.

\begin{figure}[t]
  \centering
  \includegraphics[width=\textwidth]{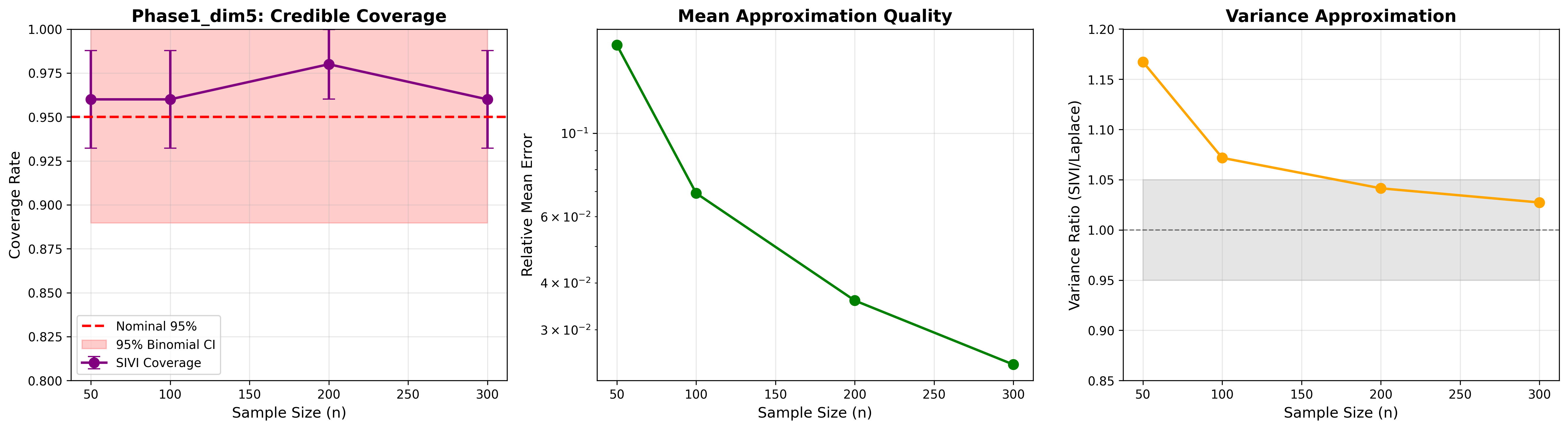}
  \caption{\textbf{Finite-sample BvM, Phase~1 ($d=5$).}
  Left: credible-set coverage with 95\% binomial confidence band.
  Center: relative mean error
  $\|\mathbb E_{q_\phi}\theta-\theta^\star\|/\|\theta^\star\|$.
  Right: variance ratio between SIVI and Laplace approximations.
  Coverage remains near nominal, and both bias and variance ratio
  converge to the Gaussian limit.}
  \label{fig:bvm-phase1}
\end{figure}

\begin{figure}[t]
  \centering
  \includegraphics[width=\textwidth]{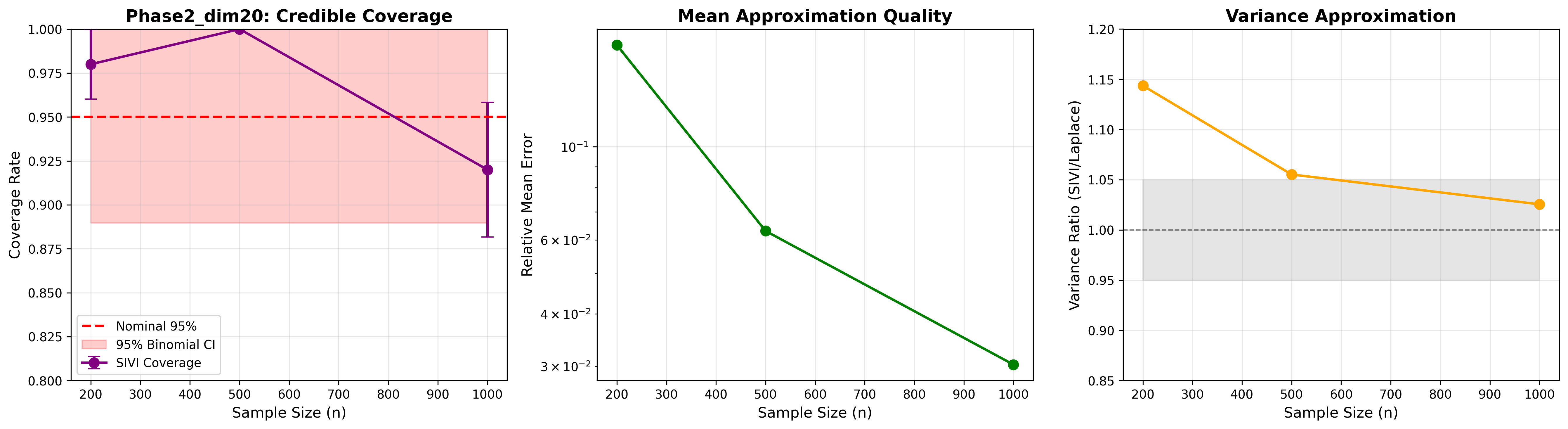}
  \caption{\textbf{Finite-sample BvM, Phase~2 ($d=20$).}
  The same trends persist in higher dimension: credible coverage near
  95\%, vanishing mean error, and variance ratio approaching~1.}
  \label{fig:bvm-phase2}
\end{figure}

\medskip
All code and experiment details are provided in the supplementary
material.
Across approximation, optimization, and statistical layers, the
numerical results mirror the theory:
they confirm compact universality and its limits under tail mismatch,
reveal the impact of kernel rigidity and finite-$K$ surrogate bias, and
demonstrate that, in regular models, SIVI inherits the Gaussian
uncertainty guarantees predicted by our oracle inequalities.

\section{Discussion and Outlook}
\label{sec:discussion}

This work develops a unified theoretical framework for 
semi-implicit variational inference (SIVI), linking its approximation,
optimization, and statistical layers within a single analysis.
At the \emph{approximation layer}, tail-dominance provides the precise
criterion for when forward--KL approximation is attainable, while the
structural impossibility results (tail mismatch, branch collapse, and
singular-support failures) delineate the intrinsic boundaries of
semi-implicit expressivity.
At the \emph{optimization layer}, finite-$K$ surrogate stability yields
non-asymptotic control of the empirical objective and its maximizers
through $\Gamma$-convergence of $L_{K,n}$ to $L_\infty$.
At the \emph{statistical layer}, these ingredients combine to give
finite-sample KL, TV, and Hellinger guarantees, as well as a 
Bernstein--von Mises limit with an explicit, tail-separated remainder.
The numerical experiments mirror each component of the theory and
illustrate the sharpness of the underlying assumptions.

Several directions for further development arise naturally.
First, extending the tail-dominance and structural analysis to
\emph{singular posteriors}, including distributions supported on
manifolds or algebraic varieties, would connect SIVI to the geometric
analysis of singular statistical models.
Second, obtaining quantitative approximation and risk rates under
H\"older- or Sobolev-type smoothness would parallel recent developments
in nonparametric variational inference.
Third, relating the asymptotic constants and irreducible forward--KL
gaps to \emph{real log canonical thresholds (RLCTs)} from singular
learning theory \citep{watanabe2009algebraic} may provide an invariant
characterization of variational bias and clarify the role of curvature
in expressive limitations.

The structural-completeness upgrades developed in 
Section~\ref{sec:structural-upgrades} indicate that many of the
identified failure modes can be removed by minimal, interpretable
modifications to the kernel family---including tail-complete kernels,
mixture-complete conditionals, and manifold-aware covariance structures.
More broadly, semi-implicit architectures offer a mathematically
tractable bridge between classical mixture models and deep
reparameterization methods.
A comprehensive statistical theory encompassing singular geometry,
finite-$K$ surrogate bias, and high-dimensional scaling remains an
important and promising direction for future work.

Overall, the results suggest that semi-implicit architectures constitute
a natural canonical limit for transformation-based variational
inference, unifying algorithmic practice with statistical principles and
providing a template for principled extensions of modern variational
methods.

\bibliographystyle{plain}
\bibliography{SIVI_refs}


\newpage
\appendix


\section{Proofs for Section~\ref{sec:setup} (Setup, Assumptions, and Realization)}

\begin{proof}[Proof of Lemma~\ref{lem:NNtoA1} (NN-universality and AOI $\Rightarrow$ compact $L^1$-universality)]
Let $K\subset\mathbb R^m$ be compact and let $f$ be a continuous density on $K$.

\medskip
\noindent\textbf{Step 1 (Approximation of identity).}
Choose $R$ with $K\subset B_R$. Let $\rho\in C_c^\infty(\mathbb R^m)$ be a standard mollifier and set $\rho_\tau(x)=\tau^{-m}\rho(x/\tau)$.  
Define
\[
f_\tau(x) \;\propto\; \big((f\,\mathbf 1_{B_{R+\tau}})*\rho_\tau\big)\,\mathbf 1_{B_R}(x).
\]
Then $f_\tau\in C^\infty(B_R)$, strictly positive and bounded away from zero on~$K$, and $\|f-f_\tau\|_{L^1(K)}\to0$ as $\tau\downarrow0$ by the standard mollifier approximation argument \citep[Thm.~8.14]{folland1999real}.

\medskip
\noindent\textbf{Step 2 (Finite Gaussian mixture approximation).}
Since $f_\tau$ is smooth on a compact set, approximate it by a finite Gaussian mixture
\[
f_{\tau,M}(x)=\sum_{j=1}^M \pi_j\,\mathcal N(x;m_j,S_j)
\]
satisfying $\|f_\tau-f_{\tau,M}\|_{L^1(K)}\le \varepsilon/3$; see, for example, the density of finite Gaussian mixtures in \cite{titterington1985mixtures}.

\medskip
\noindent\textbf{Step 3 (Realization within the SIVI family).}
Partition $\mathrm{supp}(r)$ into disjoint Borel sets $A_j$ with $r(A_j)=\pi_j$, and define target parameter maps $\tilde\mu(z)=m_j$ and $\tilde\Sigma(z)=S_j$ for $z\in A_j$.  
By neural-network universality on $\mathrm{supp}(r)$ (Assumption~\ref{ass:nn}), there exists $\theta$ such that $\mu_\theta$ and $\Sigma_\theta$ approximate $\tilde\mu$ and $\tilde\Sigma$ uniformly on $\mathrm{supp}(r)$.  
Continuity of $(\mu,\Sigma)\mapsto\mathcal N(\cdot;\mu,\Sigma)$ on compact parameter sets then yields $\|f_{\tau,M}-q_\theta\|_{L^1(K)}\le \varepsilon/3$.

\medskip
\noindent\textbf{Step 4 (Error combination).}
By the triangle inequality,
\[
\|f-q_\theta\|_{L^1(K)}
\le
\|f-f_\tau\|_{L^1(K)} + \|f_\tau-f_{\tau,M}\|_{L^1(K)} + \|f_{\tau,M}-q_\theta\|_{L^1(K)}
\le \varepsilon.
\]
Since $K$ and $f$ were arbitrary, the semi-implicit family satisfies Assumption~\ref{ass:aoi} (compact $L^1$-universality).

\medskip
\noindent\textbf{Quantitative rate.}
If $\mu_\theta$ and $\Sigma_\theta$ are ReLU networks of width~$W$, known compact-approximation results
\cite{yarotsky2017relu,schmidthieber2020deep} imply
\[
\inf_\theta \|f-q_\theta\|_{L^1(K)} = O(W^{-\beta/m}\log W)
\]
for $\beta$-H\"older densities~$f$, as stated in Remark~\ref{rem:NNtoA1}.
\end{proof}


\section{Proofs for Section~\ref{sec:approximation} (Approximation Layer)}

\subsection{Proof of Theorem~\ref{thm:universality} (Tail–dominated universality)}
\begin{proof}
Fix $\varepsilon>0$.  We construct $q\in\mathcal Q$ with
$\|p-q\|_{L^{1}}<\varepsilon$; under the stated tail integrability,
$\mathrm{KL}(p\|q)<\varepsilon$ follows as well.

\medskip
\noindent\textit{Step 1 (Truncation).}
Choose $R$ such that $\tau:=p(B_{R}^{c})\le \varepsilon/6$ and set
$K=B_{R}$.

\medskip
\noindent\textit{Step 2 (Compact approximation).}
By compact $L^{1}$ universality (Lemma~\ref{lem:NNtoA1}), choose
$q_{1}\in\mathcal Q$ with $\int_{K}|p-q_{1}|\le \varepsilon/6$.

\medskip
\noindent\textit{Step 3 (Tail component).}
Two constructions are available.

\emph{Envelope route.}
By Assumptions~\ref{ass:tail-target}–\ref{ass:tail-kernel}, assume
$p\le C_{R} v$ and that some $s\in\mathcal Q$ satisfies 
$s\ge c_{R} v$ on $K^{c}$, for a common tail envelope $v$ and constants
$C_{R},c_{R}>0$.  
Define the normalized tail
$q_{2}= s\,\mathbf 1_{K^{c}} / \int_{K^{c}} s$.

\emph{Annulus route.}
Decompose $K^{c}$ into compact annuli
$A_{m}=\{R+m-1<\|x\|\le R+m\}$, $m\ge1$.
Let $\pi_{m}=p(A_{m})/\tau$ and $p_{m}=p(\cdot\mid A_{m})$.
On each compact $A_{m}$, approximate $p_{m}$ by $q_{m}\in\mathcal Q$ in
$L^{1}(A_{m})$ with error $\delta_{m}$.
With a summable schedule (e.g., $\delta_{m}\propto 2^{-m}$),
\[
\|q_{2}-p(\cdot\mid K^{c})\|_{L^{1}(K^{c})}
\le \varepsilon/3,
\qquad
q_{2}:=\sum_{m\ge1}\pi_{m} q_{m}.
\]

\medskip
\noindent\textit{Step 4 (Mixture stitching).}
Define $q=(1-\tau)q_{1}+\tau q_{2}$.  
Then $\int_{K^{c}}q=\tau=\int_{K^{c}}p$ and $q=(1-\tau)q_{1}$ on $K$.

\medskip
\noindent\textit{Step 5 ($L^{1}$ bound).}
On $K$,
\[
\int_{K}|p-q|
\le \int_{K}|p-q_{1}|+\tau
\le \varepsilon/3.
\]
On $K^{c}$, either
$\tau\|p(\cdot\mid K^{c})-q_{2}\|_{L^{1}}\le \varepsilon/3$ (annulus),
or $\int_{K^{c}}|p-q|\le 2\tau\le \varepsilon/3$ (envelope).
Thus $\|p-q\|_{L^{1}}\le \varepsilon$.

\medskip
\noindent\textit{Step 6 (Forward–KL on $K^{c}$).}
\emph{Envelope route.}
On $K^{c}$, $p\le C_{R} v$ and $q\ge \tau c_{R} v$, so
\[
\int_{K^{c}} p\log\frac{p}{q}
\le \tau\log\!\big(C_{R}/(\tau c_{R})\big).
\]
Choosing $R$ large ensures this is at most $\varepsilon/2$.

\emph{Annulus route.}
By convexity of KL under mixing,
\[
\mathrm{KL}\big(p(\cdot\mid K^{c})\|q_{2}\big)
\le \sum_{m}\pi_{m}\,\mathrm{KL}(p_{m}\|q_{m}).
\]
On each compact $A_{m}$, Assumption~\ref{ass:ac} gives 
$p_{m}$ and $q_{m}$ bounded away from zero, so
$\mathrm{KL}(p_{m}\|q_{m})\lesssim \|p_{m}-q_{m}\|_{L^{1}(A_{m})}$.
The summability of $(\delta_{m})$ ensures the series is finite and can
be made $\le \varepsilon/(2\tau)$.

\medskip
\noindent\textit{Step 7 (Forward–KL on $K$).}
If $p\ge m_{K}>0$ on $K$, then KL is continuous in $L^{1}$\cite{van2000asymptotic}.
Thus choosing $R$ large so that $\|p-q\|_{L^{1}(K)}\le \varepsilon/3$
implies $\int_{K} p\log(p/q)\le \varepsilon/2$.

\medskip
Combining the compact and tail contributions yields
$\mathrm{KL}(p\|q)\le \varepsilon$, completing the proof.
\end{proof}

\begin{remark}
The annulus construction requires only compact approximation and a
summable error schedule; the envelope route yields slightly cleaner
constants when a common tail envelope is available.
\end{remark}

\subsection{Proof of Corollary~\ref{cor:gaussian-tail} (Gaussian kernels)}
\begin{proof}
Assume that on $K^{c}=B_{R}^{c}$ the target satisfies 
$p(x)\le A_{R} e^{-a\|x\|^{2}}$ for some $a>0$ and a constant $A_{R}$
depending on $R$.  Choose $R$ sufficiently large that 
$\tau(R):=p(K^{c})\le \varepsilon/6$.
By Lemma~\ref{lem:NNtoA1}, select $q_{1}\in\mathcal Q$ with 
$\int_{K}|p-q_{1}|\le \varepsilon/6$.

Let $s=\mathcal N(0,\sigma^{2}I)$ with $\sigma^{2}\ge (2a)^{-1}$.
Then the Gaussian tail bound gives
\[
s(x) \ge c_{\sigma} e^{-a\|x\|^{2}},
\qquad
c_{\sigma}=(2\pi\sigma^{2})^{-m/2},
\]
so $s$ dominates the same envelope as $p$ on $K^{c}$.  
Define the stitched density
\[
q = (1-\alpha)\, q_{1} + \alpha\, s.
\]

\medskip
\noindent\textit{Control on the compact region $K$.}
On $K$,
\[
\int_{K}|p-q|
\le \int_{K}|p-q_{1}| + \alpha\int_{K}|q_{1}-s|
\le \varepsilon/6 + 2\alpha.
\]
Choosing $\alpha=\varepsilon/6$ yields 
$\int_{K}|p-q|\le \varepsilon/3$.

\medskip
\noindent\textit{Control on the tail region $K^{c}$.}
Since $p\le A_{R}e^{-a\|x\|^{2}}$ and 
$q \ge \alpha c_{\sigma} e^{-a\|x\|^{2}}$ on $K^{c}$,
\[
\int_{K^{c}} p\log\frac{p}{q}
\le \tau(R)\,\log\!\left(\frac{A_{R}}{\alpha c_{\sigma}}\right).
\]
With $\tau(R)\le \varepsilon/6$ and $\alpha=\varepsilon/6$, 
choosing $R$ sufficiently large ensures that the right-hand side is at
most $\varepsilon/2$.

\medskip
\noindent\textit{Compact-region KL term.}
On the compact set $K$, $p$ is bounded below, and since
$\|p-q\|_{L^{1}(K)}\le \varepsilon/3$, continuity of KL in $L^{1}$
on compacta (see Step~7 in the proof of 
Theorem~\ref{thm:universality}) yields
$\int_{K} p\log(p/q)\le \varepsilon/2$ for $R$ large.

\medskip
Combining the compact and tail contributions shows that
$\mathrm{KL}(p\|q)\le \varepsilon$ and 
$\|p-q\|_{L^{1}}\le \varepsilon$, completing the proof.
\end{proof}

\subsection{Proof of Theorem~\ref{thm:orlicz} (Orlicz tail mismatch)}
\begin{proof}
For $t>0$ let 
\[
A_{t}=\{\theta:\langle u_{0},\theta\rangle\ge t\},
\qquad 
p_{t}=p(A_{t}),\quad 
q_{t}=q(A_{t}).
\]
By the data–processing inequality for KL divergence,
\[
\mathrm{KL}(p\|q)
\;\ge\;
\mathrm{KL}\!\big(\mathrm{Bern}(p_{t})\,\big\|\,\mathrm{Bern}(q_{t})\big).
\]

\medskip
\noindent\textit{Step 1 (Upper bound on $q_{t}$).}
Assumption~\ref{ass:psi} states that each one–dimensional projection
$\langle u,\Theta\rangle$ is uniformly sub–$\psi$ for all 
$u\in\mathbb S^{m-1}$ and all $q\in\mathcal Q$.
Hence the Chernoff–Orlicz bound \cite{rao1991theory} yields
\[
q_{t}
\;\le\;
\exp\!\{-\psi^{\ast}( t/(cL) )\},
\]
for some constant $c>0$ depending only on the Orlicz norm equivalence
and the uniform bound $L$.

\medskip
\noindent\textit{Step 2 (Lower bound on $p_{t}$).}
Assumption~\ref{ass:heavy} provides $p_{t}\ge c_{p} g(t)$ and
\[
g(t)\, e^{\psi^{\ast}( t/(cL) )}\;\longrightarrow\;\infty,
\qquad t\to\infty.
\]
Thus
\[
\frac{p_{t}}{q_{t}}
\;\ge\;
c_{p}\, g(t)\, e^{\psi^{\ast}( t/(cL) )}
\;\longrightarrow\;\infty.
\]

\medskip
\noindent\textit{Step 3 (Lower bound on Bernoulli KL).}
Since
\[
\mathrm{KL}\!\big(\mathrm{Bern}(p_{t})\|\mathrm{Bern}(q_{t})\big)
\;\ge\;
p_{t}\log\frac{p_{t}}{q_{t}},
\]
we obtain
\[
p_{t}\log\frac{p_{t}}{q_{t}}
\;\ge\;
c_{p}\, g(t)\,\psi^{\ast}( t/(cL) )
\;\longrightarrow\;\infty.
\]

\medskip
\noindent\textit{Step 4 (Choose a fixed $t_{0}$).}
By the preceding display and Assumption~\ref{ass:heavy}, there exists
$t_{0}$ sufficiently large such that
\[
c_{p}\, g(t_{0})\,\psi^{\ast}( t_{0}/(cL) )
\;\ge\; \eta
\]
for some constant $\eta>0$ independent of $q\in\mathcal Q$.
For this $t_{0}$,
\[
\mathrm{KL}(p\|q)
\;\ge\;
\mathrm{KL}\!\big(\mathrm{Bern}(p_{t_{0}})\|\mathrm{Bern}(q_{t_{0}})\big)
\;\ge\; \eta.
\]

Thus every $q\in\mathcal Q$ incurs a strictly positive KL gap, as
claimed.
\end{proof}

\subsection{Proof of Theorem~\ref{thm:branch} (Branch–collapse lower bound)}
\begin{proof}
Consider the latent–observation model
\(
X = \theta^{2} + \xi
\)
with 
\(
\xi \sim \mathcal N(0,\sigma^{2}),
\)
and fix 
\(
x \in [c_{\min},c_{\max}]
\)
with 
\(
c_{\min} \gg \sigma.
\)
Let the variational conditional family be
\[
q(\theta \mid x) = \mathcal N(m(x), v(x)),
\qquad 
v(x) \ge v_{0} > 0,
\]
corresponding to a non–autoregressive Gaussian SIVI kernel with a
variance floor in the $\theta$ direction.

\medskip
\noindent\textit{Step 1 (Local Gaussian approximation of the true conditional).}
Since
\[
\log p(\theta \mid x) 
= -\frac{(x - \theta^{2})^{2}}{2\sigma^{2}} + \mathrm{const},
\]
a second–order expansion around the stationary points
$\theta_{\pm} = \pm \sqrt{x}$ yields local curvature
\(
s_{x}^{2} = \sigma^{2}/(4x).
\)
Thus
\[
p(\theta \mid x)
\;\approx\;
\frac{
  w_{+}(x)\, \phi(\theta;\,+\sqrt{x}, s_{x}^{2})
  +
  w_{-}(x)\, \phi(\theta;\,-\sqrt{x}, s_{x}^{2})
}{
  w_{+}(x) + w_{-}(x)
},
\]
where 
\(
w_{\pm}(x) \approx p_{0}(\pm\sqrt{x}) \sqrt{2\pi s_{x}^{2}}.
\)
If $p_{0}$ is even, the mixture is approximately symmetric.

\medskip
\noindent\textit{Step 2 (Separated neighborhoods).}
For $r\ge 2$ define
\[
I_{\pm}(x)
=
[\pm \sqrt{x} - r s_{x},\ \pm \sqrt{x} + r s_{x}].
\]
Standard Gaussian tail bounds give
\begin{equation}
\label{eq:branch-true-mass-theta}
p\big(I_{\pm}(x)\mid x\big)
\;\ge\;
\tfrac12\,\Phi(r) - \delta_{x},
\end{equation}
where $\delta_{x} \to 0$ uniformly over
$x \in [c_{\min},c_{\max}]$ as $\sigma\to 0$.

\medskip
\noindent\textit{Step 3 (A single Gaussian misses one branch).}
\begin{lemma}\label{lem:two-windows-theta}
Let $a>0$, $s>0$, $r\ge 2$, and
$I_{\pm}=[\pm a - r s,\ \pm a + r s]$.
For any Gaussian $Q=\mathcal N(m,v)$ with $v\ge v_{0}>0$,
\[
\min\{Q(I_-),\, Q(I_+)\}
\;\le\;
\exp\!\Big(
 -\frac{(\frac{a}{2} - r s)_{+}^{2}}{2 v_{0}}
\Big).
\]
\end{lemma}

\begin{proof}
If $|m|\ge a/2$, the opposite window is at distance at least $a/2$ from
$m$, and Gaussian tail bounds with $v\ge v_{0}$ apply.  
If $|m|<a/2$, each window lies at distance at least $a/2-rs$, giving the
same inequality.
\end{proof}

Apply Lemma~\ref{lem:two-windows-theta} with
$a=\sqrt{x}$, $s=s_{x}$, $r=2$.  
For $x\in[c_{\min},c_{\max}]$ and $c_{\min}\gg\sigma$,
\[
\frac{a}{2} - r s
=
\frac{\sqrt{x}}{2} - \frac{r\sigma}{2\sqrt{x}}
\;\ge\;
\frac{\sqrt{c_{\min}}}{2}
      - \frac{r\sigma}{2\sqrt{c_{\min}}}
=: \Delta > 0.
\]
Hence there exists $\kappa>0$ such that
\begin{equation}
\label{eq:branch-q-window-theta}
\sup_{m,\,v\ge v_{0}}
\min\{ q(I_{-}(x)\mid x),\, q(I_{+}(x)\mid x)\}
\;\le\;
e^{-\kappa / v_{0}},
\qquad x\in[c_{\min},c_{\max}].
\end{equation}

\medskip
\noindent\textit{Step 4 (Conditional TV gap for fixed $x$).}
For any measurable $B$,
$\|p(\cdot\mid x)-q(\cdot\mid x)\|_{\mathrm{TV}}
\ge |p(B\mid x)-q(B\mid x)|$.
Thus
\[
\|p(\cdot\mid x)-q(\cdot\mid x)\|_{\mathrm{TV}}
\ge
\max\Big\{
  |p(I_{+}(x)\mid x)-q(I_{+}(x)\mid x)|,\ 
  |p(I_{-}(x)\mid x)-q(I_{-}(x)\mid x)|
\Big\}.
\]
Combining
\eqref{eq:branch-true-mass-theta}
and
\eqref{eq:branch-q-window-theta} yields
\[
\|p(\cdot\mid x)-q(\cdot\mid x)\|_{\mathrm{TV}}
\;\ge\;
\tfrac12\,\Phi(2) - \delta_{x} - e^{-\kappa/v_{0}}.
\]

Choose $c_{\min}$ large (with $\sigma$ fixed) so that  
$\delta_{x}\le \tfrac18\Phi(2)$ uniformly in $x$.  
Set
\[
\gamma
=
\tfrac{3}{8}\,\Phi(2) - e^{-\kappa/v_{0}}.
\]
If 
$e^{-\kappa/v_{0}}\le \tfrac{1}{16}\Phi(2)$, 
then $\gamma>0$, and
\begin{equation}
\label{eq:branch-tv-lower-theta}
\|p(\cdot\mid x)-q(\cdot\mid x)\|_{\mathrm{TV}}
\;\ge\; \gamma>0,
\qquad
x\in[c_{\min},c_{\max}],
\end{equation}
up to the uniform Laplace approximation errors $\delta_{x}=o_{\sigma}(1)$.

\medskip
\noindent\textit{Step 5 (Averaging over $X$).}
Let 
\(
\pi_{0}=\Pr\{X\in[c_{\min},c_{\max}]\}>0.
\)
Then
\[
\mathbb E\big[
\|p(\cdot\mid X)-q(\cdot\mid X)\|_{\mathrm{TV}}
\,\mathbbm 1\{X\in[c_{\min},c_{\max}]\}
\big]
\ge
\pi_{0}\,\gamma - o_{\sigma}(1),
\]
giving a strictly positive lower bound on the average conditional total
variation distance, uniformly over all non–autoregressive Gaussian
conditionals with variance floor $v(x)\ge v_{0}$.
This completes the proof.
\end{proof}


\subsection{Proofs for Structural Completeness Upgrades}
\label{app:struct-upgrades}

This appendix provides short proofs for the structural completeness
statements in Section~\ref{sec:structural-upgrades}.  Each follows
directly from the approximation results established in
Sections~\ref{sec:approximation}–\ref{sec:statistics}.  We retain the
notation of those sections: $p$ is the target density and $\mathcal Q$
the semi–implicit variational family.

\begin{proof}[\emph{Tail–complete kernels eliminate Orlicz mismatch}]
Suppose there exists an envelope $w$ such that $p(\theta)\lesssim w(\theta)$
and $k_{h}(\theta\mid z)\gtrsim w(\theta)$ for all large $\|\theta\|$.
Then on the tail region the ratios $p/w$ and $k_{h}/w$ are bounded above
and below by positive constants.  In particular, the one–dimensional
projections of $p$ cannot decay more slowly than those of $k_{h}$; hence
the Orlicz tail–mismatch condition of Theorem~\ref{thm:orlicz} fails.
Consequently the tail term in the forward–KL decomposition is finite,
and the compact approximation argument of
Theorem~\ref{thm:universality} applies.  Therefore
$\inf_{q\in\mathcal Q}\mathrm{KL}(p\|q)=0$.
\end{proof}

\begin{proof}[\emph{Mixture–complete conditionals eliminate branch collapse}]
Suppose the conditional kernel admits a finite–mixture representation
\[
k_\lambda(\theta\mid z)
= \sum_{j=1}^{J}
  \alpha_j(z)\,\phi_j(\theta;z),
\]
with $J$ at least the number of well–separated branches of the true
conditional $p(\theta\mid X=x)$.  Then for each $x$ we may assign one
mixture component to each branch.  By compact $L^{1}$ universality on
each branch (Lemma~\ref{lem:NNtoA1}) and the convexity of total
variation under mixing,
\[
\inf_{\lambda}
\mathrm{TV}\!\big(p(\cdot\mid x),\,q_\lambda(\cdot\mid x)\big) = 0.
\]
The lower bound of Theorem~\ref{thm:branch} does not apply, because its
only structural assumption—a unimodal conditional with a variance
floor—is violated for mixture conditionals.  Hence the branch–collapse
barrier disappears.
\end{proof}

\begin{proof}[\emph{Manifold–aware kernels recover singular supports}]
Assume $p$ is supported on a $d$–dimensional submanifold
$M\subset\mathbb R^{m}$ of codimension $r>0$.  Let $P_{\parallel}$ and
$P_{\perp}$ denote the orthogonal projections onto the tangent and
normal bundles of $M$.  Consider the anisotropic Gaussian kernels
\[
k_{h}(\theta\mid z)
=
\mathcal N\!\Big(\theta;\,
  \mu_\lambda(z),\;
  h^{2}P_{\parallel} + h_{\perp}^{\,2}P_{\perp}
\Big),
\qquad h_{\perp}\to 0.
\]
Then: (i) $k_{h}$ acts as an approximate identity along $M$;
(ii) mass orthogonal to $M$ is suppressed at rate $h_{\perp}^{\,r}$;
and (iii) by Assumption~\ref{ass:nn}, the map $z\mapsto\mu_\lambda(z)$
is dense in $C(M)$.  Hence convolution with $k_{h}$ yields
$\|k_{h}*p-p\|_{L^{1}}\to 0$, removing the support mismatch in
Assumption~\ref{ass:ac}.  Combining Lemma~\ref{lem:NNtoA1} with
Theorem~\ref{thm:universality} gives full $L^{1}$ approximation
of~$p$ by the SIVI family.
\end{proof}

\begin{proof}[\emph{Finite-$K$ surrogate regularization via annealing}]
By Assumption~\ref{ass:finiteK},
\[
\sup_{\lambda}
|L_{K,\infty}(\lambda)-L_{\infty}(\lambda)|
\;\lesssim\;
\varepsilon_{K},
\qquad
\varepsilon_{K}\lesssim K^{-1/2}.
\]
By Assumptions~\ref{ass:kernel} and
\ref{ass:complexity}–\ref{ass:tailbound},
\[
\sup_{\lambda}
|L_{K,n}(\lambda)-L_{K,\infty}(\lambda)|
\;\lesssim\;
n^{-1/2}.
\]
Hence
\[
\sup_{\lambda}
|L_{K,n}(\lambda)-L_{\infty}(\lambda)|
\;\lesssim\;
n^{-1/2} + K^{-1/2}.
\]
If $K(n)\gg n$, then $K^{-1/2}=o(n^{-1/2})$, and thus the difference is
$o(n^{-1/2})$.  By Theorem~\ref{thm:gamma},
$L_{K,n}\overset{\Gamma}{\to}L_{\infty}$ uniformly in $n$, and the
maximizers satisfy $\hat\lambda_{K,n}\to\lambda^{\star}$.  This removes
the mode–selection behavior associated with small~$K$.
\end{proof}

\begin{proof}[\emph{Robustness via tempered posteriors}]
Fix $\tau\in(0,1)$ and define the tempered posterior
\[
p_{\tau}(\theta\mid X^{(n)})
\propto
p(\theta)\, p(X^{(n)}\mid\theta)^{\tau}.
\]
Because $p(X^{(n)}\mid\theta)^{\tau}\lesssim \exp(-\tau c\|\theta\|)$,
the tempered posterior has lighter tails and satisfies the envelope
condition of Assumption~\ref{ass:tail-target}.  Thus the tail–dominance
condition of Theorem~\ref{thm:universality} holds with no Orlicz
mismatch, and
\[
\inf_{q\in\mathcal Q}
\mathrm{KL}\!\big(p_{\tau}\,\|\,q\big)
=
0.
\]
In particular, tempering yields forward–KL approximability even when the
untempered posterior satisfies
$\inf_{q}\mathrm{KL}(p\|q)>0$.
\end{proof}

\section{Proofs for Section~\ref{sec:optimization} (Optimization Layer)}
\label{app:opt-appendix}

Throughout, let $P$ denote the data-generating distribution, 
$P_n$ the empirical measure, and
\[
L_{K,n}(\lambda)
  = P_n\,\widehat\ell_{K,\lambda}, 
\qquad
\widehat\ell_{K,\lambda}(x)
    := \log\!\Big(\frac{1}{K}\sum_{k=1}^K w_\lambda(Z_k;x)\Big),
\]
be the finite-$K$ surrogate.  
The population objectives are
\[
L_{K,\infty}(\lambda)=P\,\widehat\ell_{K,\lambda},
\qquad
L_{\infty}(\lambda)=P\,\log q_\lambda.
\]
We use the decomposition
\[
\Delta_{K,n}(\lambda)
   := L_{K,n}(\lambda) - L_{\infty}(\lambda)
   = (P_n - P)\,\widehat\ell_{K,\lambda}
     \;+\;
     P\!\left(\widehat\ell_{K,\lambda} - \log q_\lambda\right),
\]
corresponding to estimation and finite-$K$ errors.  
Assumptions are those of Section~5, and auxiliary lemmas appear in Appendix~\ref{app:stat-appendix}.

\subsection{Proof of Theorem~\ref{thm:opt-finite-oracle} (Finite-sample oracle inequality)}

\begin{proof}
We decompose the uniform deviation into an empirical-process term and a
finite-$K$ approximation term:
\[
\sup_{\lambda\in\Lambda}
  |L_{K,n}(\lambda) - L_{\infty}(\lambda)|
\;\le\;
\underbrace{
  \sup_{\lambda\in\Lambda} |(P_n-P)\widehat\ell_{K,\lambda}|
}_{\text{estimation}}
\;+\;
\underbrace{
  \sup_{\lambda\in\Lambda} P|\widehat\ell_{K,\lambda}-\log q_\lambda|
}_{\text{finite-$K$ bias}}.
\]

\emph{Estimation term.}
Assumptions~\ref{ass:complexity}--\ref{ass:tailbound}
imply that $\{\log q_\lambda:\lambda\in\Lambda\}$ is measurable,
locally Lipschitz in~$\lambda$, and dominated by an integrable
envelope~$E$.
Since the mapping
\[
(t_1,\dots,t_K)\longmapsto \log\!\Bigl(\frac{1}{K}\sum_{k=1}^K t_k\Bigr)
\]
is $1$-Lipschitz on any region where the arguments are uniformly bounded
away from zero, the transformed class
$\{\widehat\ell_{K,\lambda}:\lambda\in\Lambda\}$ inherits the same
envelope and Lipschitz modulus.
By Rademacher contraction \cite{bartlett2002rademacher},
\[
\mathbb E\Big[\sup_{\lambda\in\Lambda}
   |(P_n-P)\widehat\ell_{K,\lambda}|\Big]
  \;\lesssim\;
  \sqrt{\frac{C(\Lambda)}{\,n\,}}.
\]
A Bernstein-type concentration inequality with envelope~$E$ then yields, with probability
at least $1-\delta$ \cite{boucheron2013concentration},

\[
\sup_{\lambda\in\Lambda}
 |(P_n-P)\widehat\ell_{K,\lambda}|
   \;\lesssim\;
   \sqrt{\frac{C(\Lambda)+\log(1/\delta)}{\,n\,}}
   \;=:\;\mathfrak C_n(\delta).
\]

\paragraph{Finite-$K$ term.}
Under the importance-weight moment condition of
Assumption~\ref{ass:is}, the self-normalized importance-sampling bias
bound (Lemma~\ref{lem:snis-bias}) gives
\[
\sup_{\lambda\in\Lambda}
  P\big|\widehat\ell_{K,\lambda}-\log q_\lambda\big|
   \;\lesssim\; K^{-1/2}
   \;=:\;\varepsilon_K.
\]

\paragraph{Conclusion.}
Combining the two displays,
\[
\sup_{\lambda\in\Lambda}
  |L_{K,n}(\lambda)-L_{\infty}(\lambda)|
  \;\lesssim\;
  \mathfrak C_n(\delta)+\varepsilon_K.
\]
Since $\hat\lambda_{n,K}$ maximizes $L_{K,n}$ and
$\lambda^\star\in\arg\max_\lambda L_\infty$, a standard variational
argument \cite{rockafellar1998variational} yields
\[
\mathcal R(\hat\lambda_{n,K})
   - \mathcal R(\lambda^\star)
   \;\lesssim\;
   \mathfrak A
   + \mathfrak C_n(\delta)
   + \varepsilon_K,
\]
where $\mathfrak A=\inf_{\lambda\in\Lambda}\mathrm{KL}(p\|q_\lambda)$ is
the approximation error.
This is the claimed oracle inequality.
\end{proof}

\subsection{Proof of Theorem~\ref{thm:gamma} ($\Gamma$-convergence)}
\begin{proof}
By Theorem~\ref{thm:opt-finite-oracle},
\[
\sup_{\lambda\in\Lambda}
   |L_{K,n}(\lambda)-L_\infty(\lambda)|
   \;\xrightarrow{P}\; 0.
\]
Thus $L_{K,n}\to L_\infty$ uniformly on the compact parameter set
$\Lambda$.

Assumptions~\ref{ass:kernel} and~\ref{ass:complexity} imply that
$\lambda\mapsto \log q_\lambda(x)$ is locally Lipschitz uniformly in
$x$, and Assumption~\ref{ass:tailbound} provides an integrable envelope.
Hence $\{L_{K,n}\}_{K,n}$ is an equicontinuous family on~$\Lambda$.
Since $\Lambda$ is compact, equicontinuity implies equi-coercivity.

By standard results in variational analysis 
\citep[see][]{dalmaso1993gamma,rockafellar1998variational},
uniform convergence together with equicontinuity yields
\[
-L_{K,n}\;\xrightarrow{\Gamma}\;-L_\infty \quad\text{on }\Lambda.
\]

If $\hat\lambda_{K,n}$ satisfies
\[
L_{K,n}(\hat\lambda_{K,n})
   \ge \sup_{\lambda\in\Lambda}L_{K,n}(\lambda) - o(1),
\]
the fundamental theorem of $\Gamma$-convergence implies that every limit
point of $\hat\lambda_{K,n}$ lies in
$\arg\max_{\lambda\in\Lambda} L_\infty(\lambda)$.
Since $\lambda\mapsto q_\lambda$ is continuous in the weak topology under
Assumption~\ref{ass:kernel}, it follows that
\[
q_{\hat\lambda_{K,n}}\;\Rightarrow\; q_{\lambda_\star},
\qquad
\lambda_\star\in\arg\max_{\lambda\in\Lambda}L_\infty(\lambda).
\]
\end{proof}

\subsection{Proof of Proposition~\ref{prop:div-order} (Ordering of divergences)}

\begin{proof}
Let $w(\theta)=p(\theta)/q(\theta)$ and
$\widehat w_K=\frac{1}{K}\sum_{k=1}^K w(\theta_k)$ with
$\theta_k\sim q$ i.i.d.

\paragraph{(i) $\mathcal L_K(q,p)\le \mathcal L_{K'}(q,p)$ for $K<K'$, and 
$\mathcal L_K(q,p)\le -\mathrm{KL}(q\|p)$.}
Since $\log$ is concave,
\[
\mathcal L_K(q,p)
 = \mathbb E_q\!\left[\log \widehat w_K\right]
 \;\le\;
 \log \mathbb E_q[\widehat w_K]
 = \log \mathbb E_q[w]
 = \log 1 = 0.
\]
Moreover, the sequence 
\(\{\log \widehat w_K\}\) forms an increasing Doob martingale
\citep{burda2016iwae}, hence
\[
\mathcal L_K(q,p)
\;\le\;
\mathcal L_{K'}(q,p)
\qquad (K<K').
\]
Since
\[
\sup_K \mathcal L_K(q,p)
  = \mathbb E_q[\log w]
  = -\mathrm{KL}(q\|p),
\]
the IWAE bounds are monotone and bounded above by the reverse KL ELBO.

\paragraph{(ii) $D_\alpha(p\|q)\ge \mathrm{KL}(p\|q)$ for $\alpha>1$.}
Rényi divergences are monotone in $\alpha$
\citep{vanerven2014renyi}, hence
\[
D_\alpha(p\|q)\ge \lim_{\alpha\downarrow 1} D_\alpha(p\|q)
  = \mathrm{KL}(p\|q).
\]

\paragraph{(iii) $\mathcal L_K(q,p)\uparrow -\mathrm{KL}(q\|p)$ as $K\to\infty$.}
By the martingale convergence theorem,
\[
\log \widehat w_K 
  \xrightarrow[K\to\infty]{a.s.}
  \log \mathbb E_q[w] 
  = \log 1
  = 0
  \quad\text{if } q=p,
\]
and in general,
\[
\log \widehat w_K 
  \xrightarrow[K\to\infty]{a.s.}
  \log \mathbb E_q[w]
  = \mathbb E_q[\log w],
\]
where the final equality uses the classical identity for IWAE limits
\citep{burda2016iwae}.  
By dominated convergence,
\[
\mathcal L_K(q,p)
 = \mathbb E_q[\log \widehat w_K]
 \xrightarrow[K\to\infty]{}
 \mathbb E_q[\log w]
 = -\mathrm{KL}(q\|p).
\]
\end{proof}

\subsection{Proof of Theorem~\ref{thm:alg-consistency} (Algorithmic consistency)}

\begin{proof}
Assume $\mathrm{KL}(q_{n,K}\|p_n)\to0$.  
Pinsker’s inequality gives
\[
\|q_{n,K}-p_n\|_{\mathrm{TV}}\to0.
\]

\paragraph{Continuity of $f$-divergences.}
Since $p_n$ and $q_{n,K}$ share a common support, TV convergence and
dominatedness imply continuity of Rényi divergences of order $\alpha>1$
\citep[cf.][]{csiszar1967information}.  
Thus
\[
D_\alpha(p_n\|q_{n,K})\;\to\;0
\qquad (\alpha>1).
\]

\paragraph{Finite-$K$ surrogates.}
Write $w=p_n/q_{n,K}$.  
TV convergence and absolute continuity imply $w\to1$ in $L^2(q_{n,K})$
(Cauchy–Schwarz and the identity
$\mathrm{KL}(q_{n,K}\|p_n)=\mathbb E_{q_{n,K}}[\log w]$).
Hence $\mathrm{Var}_{q_{n,K}}(w)\to0$.  
For fixed $K$, the IWAE Jensen gap satisfies
\[
\mathbb E_q\!\left[\log\!\Big(\tfrac1K\sum_{k=1}^K w_k\Big)\right]
   - \mathbb E_q[\log w]
   = O\!\left(\mathrm{Var}_{q}(w)\right)
\qquad\text{\citep[][Ch.~9]{owen2013mc}}.
\]
Since $\mathbb E_q[\log w] = -\mathrm{KL}(q_{n,K}\|p_n)\to0$, we obtain
\[
\mathcal L_K(q_{n,K},p_n) \;\to\;0.
\]

\paragraph{Conclusion.}
All divergences in the class
$\{\mathrm{KL}(q\|p),\, D_\alpha(p\|q),\, \mathcal L_K(q,p)\}$
vanish together, and hence all SIVI-type training objectives yield the
same asymptotic variational minimizers.
\end{proof}

\subsection{Proof of Lemma~\ref{lem:param-stability} (Local parameter stability)}
\begin{proof}
Let $\lambda^\star$ be the maximizer of $L_\infty$ in a neighborhood on
which $L_\infty$ is $m$-strongly concave:
\[
L_\infty(\lambda^\star) - L_\infty(\lambda)
   \;\ge\; \tfrac{m}{2}\|\lambda-\lambda^\star\|^2.
\]
On the event
$\sup_{\lambda\in\Lambda}|L_{K,n}(\lambda)-L_\infty(\lambda)|\le\Delta$,
\begin{align*}
L_\infty(\lambda^\star)
&\le L_{K,n}(\lambda^\star) + \Delta \\
&\le L_{K,n}(\hat\lambda_{n,K}) + \Delta \\
&\le L_\infty(\hat\lambda_{n,K}) + 2\Delta.
\end{align*}
Thus
\[
L_\infty(\lambda^\star)-L_\infty(\hat\lambda_{n,K})
   \;\le\; 2\Delta.
\]
Strong concavity then yields
\[
\tfrac{m}{2}\|\hat\lambda_{n,K}-\lambda^\star\|^2
   \;\le\; 2\Delta,
\]
and hence
\[
\|\hat\lambda_{n,K}-\lambda^\star\|
   \;\le\; \sqrt{2\Delta/m},
\]
with probability at least $1-\delta$.
\end{proof}

\section{Proofs for Section~\ref{sec:statistics} (Statistical Layer)}
\label{app:stat-appendix}
\subsection{Proofs for Local Geometry and Structural Conditions}

\begin{proof}[Proof of Lemma~\ref{lem:quadratic-risk} (Local quadratic expansion)]
By Assumption~\ref{ass:curvature-local}, $L_\infty$ is $C^2$ on a
neighborhood $\mathcal N(\lambda^\star)$ and
$\nabla^2 L_\infty(\lambda^\star)\preceq -m I_d$.
For any $\lambda$ sufficiently close to $\lambda^\star$, a second-order
Taylor expansion gives
\[
L_\infty(\lambda)
=
L_\infty(\lambda^\star)
+\tfrac12(\lambda-\lambda^\star)^\top
\nabla^2 L_\infty(\tilde\lambda)
(\lambda-\lambda^\star),
\]
for some $\tilde\lambda$ on the segment joining 
$\lambda^\star$ and $\lambda$.
Write
\[
\nabla^2 L_\infty(\tilde\lambda)
=
\nabla^2 L_\infty(\lambda^\star)
+
\bigl[
 \nabla^2 L_\infty(\tilde\lambda)
 -\nabla^2 L_\infty(\lambda^\star)
\bigr].
\]
Since $\nabla^2 L_\infty(\lambda^\star)\preceq -m I_d$,
\[
L_\infty(\lambda^\star)-L_\infty(\lambda)
\;\ge\;
\tfrac{m}{2}\|\lambda-\lambda^\star\|^2
-
\tfrac12
\bigl\|
  \nabla^2 L_\infty(\tilde\lambda)
  -\nabla^2 L_\infty(\lambda^\star)
\bigr\|
\,
\|\lambda-\lambda^\star\|^2.
\]
Local $C^2$ regularity implies that the Hessian is Lipschitz on
$\mathcal N(\lambda^\star)$:
\[
\bigl\|
  \nabla^2 L_\infty(\tilde\lambda)
  -\nabla^2 L_\infty(\lambda^\star)
\bigr\|
\le
L_H\,\|\tilde\lambda-\lambda^\star\|
\le
L_H\,\|\lambda-\lambda^\star\|.
\]
Substituting and setting $C=L_H/2$ yields
\[
L_\infty(\lambda^\star)-L_\infty(\lambda)
\;\ge\;
\frac{m}{2}\|\lambda-\lambda^\star\|^2
-
C\,\|\lambda-\lambda^\star\|^3,
\]
as claimed.
\end{proof}

\subsection{Proofs for Finite-Sample Oracle Bounds}
\label{app:finite-tv-proofs}
\begin{proof}[Proof of Theorem~\ref{thm:finite-tv}
(Finite-sample TV/Hellinger oracle)]
By Theorem~\ref{thm:opt-finite-oracle}, with probability at least
$1-\delta$,
\[
\mathrm{KL}(p\|\hat q_{n,K})
\;\le\;
\mathfrak A + C\,\mathfrak C_n(\delta) + C'\,\varepsilon_K.
\]

\paragraph{Total variation.}
Pinsker’s inequality gives
\[
\|\hat q_{n,K}-p\|_{\mathrm{TV}}
\;\le\;
\sqrt{\tfrac12\,\mathrm{KL}(p\|\hat q_{n,K})}
\;\le\;
\sqrt{\tfrac12\bigl(\mathfrak A + C\,\mathfrak C_n(\delta)
+ C'\,\varepsilon_K\bigr)}.
\]

\paragraph{Hellinger distance.}
Using $H^2(\hat q_{n,K},p)\le \|\hat q_{n,K}-p\|_{\mathrm{TV}}$, we obtain
\[
H^4(\hat q_{n,K},p)
\;\le\;
\|\hat q_{n,K}-p\|_{\mathrm{TV}}^2
\;\le\;
\tfrac12\bigl(\mathfrak A + C\,\mathfrak C_n(\delta)
+ C'\,\varepsilon_K\bigr),
\]
hence
\[
H(\hat q_{n,K},p)
\;\le\;
2^{-1/4}
\bigl(\mathfrak A + C\,\mathfrak C_n(\delta)
+ C'\,\varepsilon_K\bigr)^{1/4}.
\]
\end{proof}

\subsection{Proofs for Contraction and Coverage Transfer}
\label{app:contraction-coverage-proofs}

\begin{proof}[Proof of Theorem~\ref{thm:contraction-tv}
(Posterior contraction transfer)]
Let 
\[
B_n=\{\theta : d(\theta,\theta_0)>M_n\varepsilon_n\}.
\]
By the defining inequality for total variation,
\[
\hat q_{n,K}(B_n)
  \le p_n(B_n)
     + \|\hat q_{n,K}-p_n\|_{\mathrm{TV}}.
\]
By assumption, $p_n(B_n)\to 0$ in probability whenever $M_n\to\infty$,
and $\|\hat q_{n,K}-p_n\|_{\mathrm{TV}}=o_P(1)$.
Hence $\hat q_{n,K}(B_n)\to 0$ in probability, establishing contraction
of $\hat q_{n,K}$ at the same rate $\varepsilon_n$.
\end{proof}

\begin{proof}[Proof of Corollary~\ref{cor:coverage-tv}
(Coverage transfer under LAN/BvM)]
Let $C_n(\alpha)$ be any sequence of credible sets satisfying the BvM
property $p_n(C_n(\alpha))\to 1-\alpha$ in probability.
Then for every $n$,
\[
\big|
  \hat q_{n,K}(C_n(\alpha)) - p_n(C_n(\alpha))
\big|
  \le
  \|\hat q_{n,K}-p_n\|_{\mathrm{TV}}.
\]
Since
$\|\hat q_{n,K}-p_n\|_{\mathrm{TV}} = o_P(1)$,
the right-hand side converges to zero in probability.
Because $p_n(C_n(\alpha))\to 1-\alpha$ in probability as well,
it follows that
\[
\hat q_{n,K}(C_n(\alpha))
  \longrightarrow 1-\alpha
  \qquad \text{in probability}.
\]
\end{proof}

\subsection{Proofs for Setwise Uncertainty Transfer}
\label{app:setwise-proofs}

\begin{proof}[Proof of Theorem~\ref{thm:setwise-tv}
(Setwise uncertainty transfer)]
By definition of total variation,
\[
\|\hat q_{n,K}-p_n\|_{\mathrm{TV}}
  = \sup_{A} |\hat q_{n,K}(A)-p_n(A)|.
\]
Therefore, for every measurable set $A=A(X^{(n)})$,
\[
|\hat q_{n,K}(A)-p_n(A)|
  \le \|\hat q_{n,K}-p_n\|_{\mathrm{TV}},
\]
almost surely.
The interval inclusion for $p_n(A)$ follows directly.
\end{proof}

\begin{proof}[Proof of Corollary~\ref{cor:coverage-general}
(Credible-set coverage without regularity)]
Let $C_n(\alpha)$ satisfy
$\hat q_{n,K}(C_n(\alpha))\ge 1-\alpha$.
If $\|\hat q_{n,K}-p_n\|_{\mathrm{TV}}\le\varepsilon$, then
\[
p_n(C_n(\alpha))
  \ge \hat q_{n,K}(C_n(\alpha))
        - \|\hat q_{n,K}-p_n\|_{\mathrm{TV}}
  \ge 1-\alpha-\varepsilon.
\]
Likewise, if $\hat q_{n,K}(C_n(\alpha+\varepsilon))\ge 1-\alpha$, then
\[
p_n(C_n(\alpha+\varepsilon))\ge 1-\alpha.
\]
\end{proof}

\begin{proof}[Proof of Corollary~\ref{cor:uniform-band}
(Uniform posterior-probability band)]
The identity
\[
\|\hat q_{n,K}-p_n\|_{\mathrm{TV}}
  = \sup_{A}|\hat q_{n,K}(A)-p_n(A)|
\]
holds by definition of total variation, and therefore controls all
posterior probabilities simultaneously.
\end{proof}

\subsection{Proofs for Tail-Event and Functional Decomposition}
\label{app:tail-decomposition-proofs}

\begin{proof}[Proof of Theorem~\ref{thm:tv-decomposition}
(Compact--tail total-variation decomposition)]
Write $\tau_p = p(K^c)$ and $\tau_q = q(K^c)$, and let
$p_K,q_K$ denote the renormalized restrictions of $p,q$ to $K$, so that
$p = (1-\tau_p)p_K + \tau_p p_{K^c}$ and
$q = (1-\tau_q)q_K + \tau_q q_{K^c}$, with $p_{K^c},q_{K^c}$ the
conditional laws on $K^c$.
By definition,
\[
\|p-q\|_{\mathrm{TV}}
  = \tfrac12\int_{\mathbb R^m} |p-q|
  = \tfrac12\int_K |p-q| + \tfrac12\int_{K^c} |p-q|.
\]

\medskip
\noindent\textbf{Core term on $K$.}
On $K$ we have
\[
p = (1-\tau_p)p_K,\qquad q = (1-\tau_q)q_K,
\]
so
\begin{align*}
\int_K |p-q|
 &= \int_K \big|(1-\tau_p)p_K - (1-\tau_q)q_K\big| \\
 &\le \int_K \big|(1-\tau_p)(p_K-q_K)\big|
      + \int_K \big|(\tau_q-\tau_p)q_K\big| \\
 &= (1-\tau_p)\int_K |p_K-q_K| + |\tau_p-\tau_q|\int_K q_K \\
 &= (1-\tau_p)\,\|p_K-q_K\|_1 + |\tau_p-\tau_q|.
\end{align*}
Thus
\[
\tfrac12\int_K |p-q|
  \;\le\; (1-\tau_p)\,\|p_K-q_K\|_{\mathrm{TV}}
          \;+\;\tfrac12|\tau_p-\tau_q|.
\]

\medskip
\noindent\textbf{Tail term on $K^c$.}
On $K^c$ we have
\[
p = \tau_p p_{K^c},\qquad q = \tau_q q_{K^c},
\]
so
\begin{align*}
\int_{K^c} |p-q|
 &= \int_{K^c} \big|\tau_p p_{K^c} - \tau_q q_{K^c}\big|
  = \int_{K^c} \big|\tau_p(p_{K^c}-q_{K^c})
                    + (\tau_p-\tau_q)q_{K^c}\big| \\
 &\le \tau_p \int_{K^c} |p_{K^c}-q_{K^c}|
      + |\tau_p-\tau_q|\int_{K^c} q_{K^c} \\
 &= 2\tau_p\,\mathrm{TV}(p_{K^c},q_{K^c})
    + |\tau_p-\tau_q|.
\end{align*}
Hence
\[
\tfrac12\int_{K^c} |p-q|
  \;\le\; \tau_p\,\mathrm{TV}(p_{K^c},q_{K^c})
          \;+\;\tfrac12|\tau_p-\tau_q|.
\]

\medskip
\noindent\textbf{Combining.}
Adding the two contributions and using
$\tau_p\le \max\{\tau_p,\tau_q\}$ gives
\begin{align*}
\|p-q\|_{\mathrm{TV}}
 &= \tfrac12\int_K |p-q| + \tfrac12\int_{K^c} |p-q| \\
 &\le (1-\tau_p)\,\|p_K-q_K\|_{\mathrm{TV}}
     + \tau_p\,\mathrm{TV}(p_{K^c},q_{K^c})
     + |\tau_p-\tau_q| \\
 &\le (1-\tau_p)\,\|p_K-q_K\|_{\mathrm{TV}}
     + \max\{\tau_p,\tau_q\}\,\mathrm{TV}(p_{K^c},q_{K^c})
     + |\tau_p-\tau_q|.
\end{align*}
This is the stated compact--tail decomposition.

\medskip
\noindent\textbf{Envelope bound for tail mass.}
If Assumptions~\ref{ass:tail-target}--\ref{ass:tail-kernel} hold with
envelope $v$, then on $K^c$ there exist constants $C_1,C_2<\infty$ with
\[
p(x)\le C_1 v(\|x\|),\qquad q(x)\le C_2 v(\|x\|).
\]
Thus
\[
\tau_p + \tau_q
  = \int_{K^c} p + \int_{K^c} q
  \;\le\; (C_1+C_2)\int_{\|x\|>R} v(\|x\|)\,dx
  \;\lesssim\; \int_R^\infty v(r)\,r^{m-1}\,dr.
\]
In particular,
\[
|\tau_p-\tau_q|
  \;\le\; \tau_p + \tau_q
  \;\lesssim\; \int_R^\infty v(r)\,r^{m-1}\,dr,
\]
which yields the claimed tail-mass bound.
\end{proof}

\begin{proof}[Proof of Corollary~\ref{cor:tail-event}
(Tail-event probability bound)]
For $A\subseteq K^c$,
\[
|p(A)-q(A)|
  \le |p(A)-p(K^c)p_{K^c}(A)|
      + |p(K^c)p_{K^c}(A)-q(K^c)q_{K^c}(A)|
      + |q(K^c)q_{K^c}(A)-q(A)|.
\]
The first and last terms vanish by definition of $p_{K^c},q_{K^c}$.
The middle term is bounded by
\[
|p(K^c)-q(K^c)|\,p_{K^c}(A)
  + q(K^c)\,|p_{K^c}(A)-q_{K^c}(A)|
\le
|\tau_p-\tau_q|
+ \max\{\tau_p,\tau_q\}\,
  \mathrm{TV}(p_{K^c},q_{K^c}).
\]
\end{proof}

\begin{proof}[Proof of Corollary~\ref{cor:functional-decomposition}
(Functional decomposition bound)]
Decompose
\[
|\mathbb E_p f - \mathbb E_q f|
  \le
  \big|\mathbb E_{p_K} f - \mathbb E_{q_K} f\big|
  + \big|\mathbb E_{p_{K^c}} f - \mathbb E_{q_{K^c}} f\big|
  + |(1-\tau_p)-(1-\tau_q)|\,\|f\|_\infty
  + |\tau_p-\tau_q|\,\|f\|_\infty.
\]
On $K$, the Kantorovich–Rubinstein duality \cite{villani2009ot} gives
\(
|\mathbb E_{p_K} f - \mathbb E_{q_K} f|
  \le L_f W_1(p_K,q_K).
\)
On $K^c$, 
\(
|\mathbb E_{p_{K^c}} f - \mathbb E_{q_{K^c}} f|
  \le 2\|f\|_\infty\,\mathrm{TV}(p_{K^c},q_{K^c}).
\)
Collecting terms and recalling
$|(1-\tau_p)-(1-\tau_q)|=|\tau_p-\tau_q|$ gives the claimed bound.
\end{proof}

\subsection{Proofs for Bernstein--von Mises Limits}
\label{app:bvm-proofs}

\begin{proof}[Proof of Theorem~\ref{thm:sivi-bvm-finite}
(Finite-sample SIVI--Bernstein--von Mises)]
Let $T_n(\theta)=\sqrt n(\theta-\hat\theta_n)$ be the LAN rescaling, and
define the pushforward measures
\[
\Pi_n := \mathcal L_{p_n}\{T_n(\theta)\},\qquad
\widehat\Pi_n := \mathcal L_{\hat q_{n,K}}\{T_n(\theta)\}.
\]
By hypothesis,
\[
d_{\mathrm{BL}}\!\left(\Pi_n,\mathcal N(0,I(\theta^\star)^{-1})\right)
\le r_n
\quad\text{in probability}.
\]

Fix any test function $\varphi$ in the unit bounded–Lipschitz ball,
$\|\varphi\|_{\mathrm{BL}}\le1$, and define 
$\phi(\theta)=\varphi(T_n(\theta))$.
Since $|\phi|\le1$,
\[
\Big|\int \varphi\,d\widehat\Pi_n
      - \int \varphi\,d\Pi_n\Big|
  = \Big|\int \phi(\theta)\,(\hat q_{n,K}-p_n)(d\theta)\Big|
  \le \|\hat q_{n,K}-p_n\|_{\mathrm{TV}}.
\]
Taking the supremum over $\varphi$ yields
\[
d_{\mathrm{BL}}(\widehat\Pi_n,\Pi_n)
  \le \|\hat q_{n,K}-p_n\|_{\mathrm{TV}}.
\]

Next apply the compact--tail decomposition
(Theorem~\ref{thm:tv-decomposition}):
\[
\|\hat q_{n,K}-p_n\|_{\mathrm{TV}}
  \le 
    \|p_{n,K}-\hat q_{n,K}\|_{\mathrm{TV}}
    + |\tau_{p_n}-\tau_{\hat q_{n,K}}|.
\]
Finally, the triangle inequality for the BL metric gives
\[
d_{\mathrm{BL}}(\widehat\Pi_n,\mathcal N(0,I(\theta^\star)^{-1}))
 \le r_n
   + C\!\left(
       \|p_{n,K}-\hat q_{n,K}\|_{\mathrm{TV}}
       + |\tau_{p_n}-\tau_{\hat q_{n,K}}|
     \right),
\]
for a finite constant $C$ depending only on the LAN radius and the
bounded–Lipschitz norm.  This proves the theorem.
\end{proof}

\begin{proof}[Proof of Corollary~\ref{cor:bvm-relu}
(Explicit ReLU-rate remainder)]
By Corollary~\ref{cor:finite-relu-tv},
\[
\|p_{n,K}-\hat q_{n,K}\|_{\mathrm{TV}}
  \;\lesssim\;
  \Big(
    W^{-\beta/m}\log W
    + \sqrt{C(W)/n}
    + K^{-1/2}
  \Big)^{1/2}.
\]
Under the shared tail envelope $v$ from
Assumptions~\ref{ass:tail-target}--\ref{ass:tail-kernel},
the tail-mass discrepancy is bounded by
\[
|\tau_{p_n}-\tau_{\hat q_{n,K}}|
  \;\lesssim\;
  \int_R^\infty v(r)\,r^{m-1}\,dr,
\]
for any radius $R$ on which the LAN approximation holds.

Substitute these bounds into Theorem~\ref{thm:sivi-bvm-finite}.
Using the standard BL--Hellinger--TV relationship (Appendix~\ref{app:stat-appendix}),
\begin{align*}
d_{\mathrm{BL}}&
\!\left(
  \mathcal L_{\hat q_{n,W,K}}\!\{\sqrt n(\theta-\hat\theta_n)\},
  \mathcal N(0,I(\theta^\star)^{-1})
\right)
\;\lesssim\;\\
&r_n
+ \Big(
    W^{-\beta/m}\log W
    + \sqrt{C(W)/n}
    + K^{-1/2}
    + \!\!\int_R^\infty v(r)\,r^{m-1}dr
  \Big)^{1/4}.
\end{align*}
This establishes the stated explicit remainder.
\end{proof}

\section{Constructive and Stability Lemmas for Semi-Implicit Objectives}
\label{app:constructive}

This appendix collects constructive and quantitative lemmas supporting
the approximation, optimization, and finite-$K$ analyses in
Sections~\ref{sec:approximation}–\ref{sec:statistics}.
All results are expressed in terms of the forward KL divergence
$\mathrm{KL}(p\|q_\lambda)$, consistent with the population objective
$L_\infty(\lambda)=\mathbb{E}_{p}[\log q_\lambda(X)]$ used throughout the paper.

\subsection{Constructive Approximation via Semi-Implicit Kernels}
\label{app:constructive-approx}

\begin{lemma}[Constructive Gaussian mixture approximation]
\label{lem:constructive}
Let $p_n$ be a smooth posterior with exponentially decaying tails and
$\beta$-H\"older local regularity near $\theta^\star$.
Let $q_{\phi,h}(\theta)
 =\int\!k_{\phi,h}(\theta\mid z)\,\nu(dz)$
with Gaussian
$k_{\phi,h}(\theta\mid z)=\mathcal N(\theta;\mu_\phi(z),h^2I_m)$
and Lipschitz neural maps $\mu_\phi$ dense in $C^1$ on compacts.
Then there exist $\phi_n$ and $h_n\!\downarrow\!0$ such that
\[
\|q_{\phi_n}-p_n\|_1\lesssim h_n^2+\epsilon_n,
\qquad
\mathrm{KL}(p_n\|q_{\phi_n})=O((h_n^2+\epsilon_n)^2),
\]
where $\epsilon_n$ is the neural transport error.
Choosing $h_n^2\asymp\epsilon_n=o(r_n)$ gives
$\mathrm{KL}=o(r_n^2)$.
\end{lemma}

\subsection{Local Argmax Stability under Strong Concavity}
\label{app:argmax}

\begin{assumption}[Local strong concavity]\label{ass:local-strong-concavity}
There exists a neighborhood $\mathcal N(\phi^\star)\subset\Phi$ and
$m>0$ such that the population ELBO
$\mathcal L(\phi)=\mathbb{E}_{p}[\log q_\phi(X)]$ (equivalently,
$\mathcal L = \mathrm{const} - \mathrm{KL}(p\|q_\phi)$) is $m$-strongly concave:
\[
\mathcal L(\phi_2)\le \mathcal L(\phi_1)
+\nabla\mathcal L(\phi_1)^\top(\phi_2-\phi_1)
-\tfrac{m}{2}\|\phi_2-\phi_1\|^2.
\]
\end{assumption}

\subsection{Finite-$K$ Bias and Explicit Schedules}
\label{app:finiteK}

\begin{proposition}[Finite-$K$ bias expansion]
\label{prop:finiteK}
Assume the importance weights satisfy 
\[
\sup_\lambda \mathrm{CV}^2\bigl(w_\lambda(X,Z)\bigr) < \infty,
\]
where $w_\lambda(X,Z) = p(X,Z)/q_\lambda(X\mid Z)$ and 
$\mathrm{CV}^2$ denotes the squared coefficient of variation under 
$q_\lambda$. Then there exists $C<\infty$ such that, uniformly in $\lambda$,
\[
0 \;\le\; L_{\infty,n}(\lambda) - L_{K,n}(\lambda) \;\le\; \frac{C}{K}.
\]
Consequently,
\[
\sup_{\lambda\in\Lambda} 
  \big|L_{K,n}(\lambda)-L_{\infty,n}(\lambda)\big| = O(K^{-1}).
\]
\end{proposition}

\begin{corollary}[Schedules ensuring target rates]
To guarantee $\varepsilon_K = o(r_n^2)$, it is sufficient to take 
$K(n)\gg r_n^{-2}$.  
For LAN rates $r_n = n^{-1/2}$, this requires $K(n)\gg n$.

A practical adaptive rule is:
\[
\big|L_{K(n),n}(\hat\lambda) - L_{\infty}(\hat\lambda)\big|
  \le \tau_n,
\qquad
\tau_n =
\begin{cases}
r_n^2, & \text{for contraction},\\[2pt]
n^{-1}, & \text{for BvM remainder}.
\end{cases}
\]
\end{corollary}

\subsection{Self-Normalized Importance-Weight Bias}
\label{app:snis}

\begin{lemma}[SNIS bias bound for the SIVI surrogate]
\label{lem:snis-bias}
Let 
\[
L_{K,\infty}(\lambda)
  = \mathbb{E}_{P^*}\!\left[
      \log\!\left(\frac{1}{K}\sum_{k=1}^K 
      w_\lambda(X,Z_k)\right)
    \right],
\qquad
L_{\infty}(\lambda)
  = \mathbb{E}_{P^*}\bigl[\log q_\lambda(X)\bigr],
\]
where $w_\lambda(X,Z)=p(X,Z)/q_\lambda(X\mid Z)$
and $Z_{1:K}\sim r$.
Assume
$\sup_{\lambda}\mathrm{CV}^2(w_\lambda)<\infty$.
Then for all $\lambda$,
\[
0 \;\le\;
L_{\infty}(\lambda)-L_{K,\infty}(\lambda)
\;\le\;
\frac{C}{\sqrt{K}},
\]
for a constant $C<\infty$ depending only on the uniform bound on
$\mathrm{CV}(w_\lambda)$.
Consequently,
\[
\sup_{\lambda\in\Lambda}
  \big|L_{K,\infty}(\lambda)-L_\infty(\lambda)\big|
  \lesssim K^{-1/2}.
\]
\end{lemma}

\begin{proof}[Sketch]
Write
\(
S_K = \frac{1}{K}\sum_{k=1}^K w_\lambda(X,Z_k)
\)
so that $\mathbb{E}[S_K]=q_\lambda(X)$.
A second-order Taylor expansion of $\log S_K$ around
$q_\lambda(X)$ gives
\[
\mathbb{E}\bigl[\log S_K\bigr]
  = \log q_\lambda(X)
    -\frac{\mathrm{Var}(S_K)}{2q_\lambda(X)^2}
    + O\!\left(\frac{\mathbb{E}[|S_K-q_\lambda(X)|^3]}
                    {q_\lambda(X)^3}\right).
\]
Under $\sup_\lambda\mathrm{CV}^2(w_\lambda)<\infty$,
\(
\mathrm{Var}(S_K)=O(K^{-1})
\)
and higher moments scale as $O(K^{-3/2})$; see \cite{liu2001mc,owen2013mc}.
Taking expectations over $P^*$ gives
\[
0 \le L_\infty(\lambda)-L_{K,\infty}(\lambda)
   \le C K^{-1/2},
\]
with $C$ depending on the uniform CV bound.  
Uniformity over $\lambda$ follows from the uniform CV assumption.
\end{proof}

\subsection{Uniform Lipschitzness of the Finite-$K$ Objective}
\label{app:lipschitz}

\begin{lemma}[Uniform Lipschitzness of $\widehat{\mathcal L}_K$]
\label{lem:lipschitz}
Let $\widehat{\mathcal L}_K(\phi)$ be the finite-$K$ SIVI surrogate.
Assume:
(i) local Lipschitzness of $\log q_\phi(x)$;
(ii) an integrable envelope $E$ with $|\log q_\phi(x)|\le E(x)$;
(iii) bounded second moments of IS weights
$\sup_\phi\mathbb{E}_{q_\phi}w_\phi^2<\infty$;
(iv) bounded Jacobians for $\mu_\phi,\Sigma_\phi$ on $\Phi_n$.
Then there exists $L_n<\infty$ such that
\[
|\widehat{\mathcal L}_K(\phi_1)-\widehat{\mathcal L}_K(\phi_2)|
\le L_n\|\phi_1-\phi_2\|,
\quad \forall\,\phi_1,\phi_2\in\Phi_n.
\]
\end{lemma}

\begin{lemma}[Local argmax stability: gradient perturbation]
\label{lem:argmax-grad}
If additionally
$\sup_{\phi}\|\nabla\widehat{\mathcal L}(\phi)
        -\nabla\mathcal L(\phi)\|\le\gamma$,
then
\[
\|\widehat\phi-\phi^\star\|\le \gamma/m,
\qquad
\mathcal L(\phi^\star)-\mathcal L(\widehat\phi)\le \gamma^2/(2m).
\]
\end{lemma}


\section{Experiment Implementation Details}

\paragraph{Implementation details for Figure~\ref{fig:tv-vs-width}.}
The target $p$ is a smooth, compactly supported density proportional to
$\exp(-2\|\theta\|^2)\,\mathbf 1_{\{\|\theta\|\le 2\}}$ on $\mathbb R^2$.
The base distribution $r$ is $\mathcal N(0,I_2)$.

The networks $\mu_\lambda$ and $\Sigma_\lambda$
each have two hidden layers with ReLU activations and widths
$W\in\{8,16,32,64,128,256\}$.
Optimization uses Adam (learning rate $10^{-3}$, batch size 256)
for 20{,}000 iterations with early stopping on a validation split.

The grid estimator approximates
$\|p - q_\lambda\|_{\mathrm{TV}}=\tfrac12\int |p-q_\lambda|$
by integrating over a $200\times 200$ lattice in $[-2,2]^2$.
The $p$-sampling estimator uses $10^5$ draws from $p$
to compute a Monte Carlo approximation of the same quantity.

Error bars denote one standard deviation across five independent seeds.
All experiments were implemented in PyTorch using double precision.

\paragraph{Implementation details for Figure~\ref{fig:tail-dominance}.}
The base distribution is $r(z)=\mathcal N(0,1)$.
Kernels $k_\lambda(\theta\mid z)=\mathcal N(\theta;\mu_\lambda(z),
\sigma_\lambda^2(z))$ use two-layer ReLU networks for
$\mu_\lambda$ and $\log\sigma_\lambda$ with widths
$W\in\{16,32,64,128,256\}$.
Optimization uses Adam (learning rate $10^{-3}$, batch size~256) for
$5{\times}10^4$ iterations with early stopping.
For the heavy-tailed case, the target is $t_\nu$ with $\nu=3$; for the
sub-Gaussian case, $p=\mathcal N(0,1)$.
Each configuration averages $5$ random seeds.
The forward--KL $\mathrm{KL}(p\|q_\lambda)$ is estimated by Monte Carlo
with $10^5$ draws from~$p$, with standard errors computed via the delta
method.
The theoretical lower-bound line in~(2b) corresponds to the conservative
projection estimate from the Orlicz tail-mismatch inequality
(Theorem~\ref{thm:orlicz}).

\paragraph{Implementation details for Figure~\ref{fig:finiteK}.}
All models use a Gaussian base $r(z)=\mathcal N(0,1)$ and networks
$\mu_\lambda,\log\sigma_\lambda$ with two hidden layers of width~64.
Training employs Adam with learning rate~$10^{-3}$,
batch size~512, and $2{\times}10^4$ updates per~$K$.
The target mixture is
\[
p(\theta)
  = \tfrac12\,\mathcal N(\theta;\!-3,1)
    + \tfrac12\,\mathcal N(\theta;3,1).
\]
Mode masses for $\hat q_{K}$ are estimated from $10^5$ samples using
fixed interval thresholds around each mode.
Total-variation distance is estimated by numerical quadrature on
a uniform grid in $[-8,8]$.
The dashed reference line in panel~(3b) corresponds to the
$K^{-1/2}$ rate predicted by Lemma~\ref{lem:snis-bias}.


\paragraph{Implementation details for Figure~\ref{fig:branch-collapse}.}
Each branch of $p(x)$ is a Gaussian component with covariance
$0.05^2 I_2$, and branch regions are defined as disks of radius~0.4
centered at the three mode locations.
The SIVI decoder networks $\mu_\theta$ and $\Sigma_\theta$ each use two
hidden layers of width~64 with ReLU activations, and the covariance map
$\Sigma_\theta(z)$ is constrained to satisfy the variance floor
$\sigma^2_\theta(z)\ge 0.05^2$ on every diagonal entry
(the setting of Theorem~\ref{thm:branch}).
Training employs Adam with learning rate $10^{-3}$, batch size~1024, and
early stopping based on a held-out ELBO monitor.

All branch probabilities $p(A_j)$ and $q(A_j)$ are Monte
Carlo estimates based on $10^5$ draws.
The theoretical ``branch bound'' shown in the figure is computed as
$2\Phi\!\bigl(-(\tfrac{a}{2}-r s)/\sqrt{v_0}\bigr)$ with
$a = \| \mu_1-\mu_2 \|$, $r=2$, $s=0.05$, and $v_0$ equal to the
variance floor, matching the expression obtained in
Step~3 of the proof of Theorem~\ref{thm:branch}.


\paragraph{Implementation details for 
Figures~\ref{fig:gamma-convergence}--\ref{fig:gamma-landscapes}.}
For each $(K,n)$ pair, we generate $n$ observations 
$X_i \sim p(\,\cdot\,\mid\theta^\star)$ and evaluate the empirical 
finite-$K$ objective
\[
L_{K,n}(\theta)
  = \frac{1}{n}\sum_{i=1}^n
      \log\!\Biggl(
        \frac{1}{K}\sum_{k=1}^K 
        w_\theta(Z_k;X_i)
      \Biggr),
\qquad
w_\theta(Z_k;X_i)
  = \frac{p(X_i,Z_k)}{q_\theta(X_i\mid Z_k)},
\]
with $Z_k \sim r$ drawn independently for each $i$.
The objective $L_{K,n}(\theta)$ is evaluated on a uniform grid 
$\theta\in[-3,3]$ and maximized numerically to obtain 
$\hat\theta_{K,n}$.

The $\Gamma$-distance is approximated by
\[
\sup_{\theta\in[-3,3]}
  \bigl|\,L_{K,n}(\theta)-L_\infty(\theta)\,\bigr|,
\]
computed on the same grid.
Results are averaged over five independent replications.
Theoretical reference curves proportional to $K^{-1}$ and $n^{-1/2}$ 
are included for comparison with the predicted rates from 
Theorem~\ref{thm:gamma}.

\paragraph{Implementation details for 
Figures~\ref{fig:bvm-phase1}--\ref{fig:bvm-phase2}.}
Laplace posterior approximations are obtained by computing the MLE 
via Newton iterations and evaluating the negative Hessian of the log-likelihood 
at the optimum.  
SIVI training uses $K=50$ inner samples, minibatch size $512$, and
Adam with learning rate $10^{-3}$ for $3{\times}10^{4}$ iterations.
Credible ellipsoids are defined as
\[
\bigl\{\theta : (\theta-\hat m)^\top \hat V^{-1}(\theta-\hat m)
      \le \chi^2_{d,\,0.95}\bigr\},
\]
where $(\hat m,\hat V)$ are the SIVI posterior mean and covariance.
Coverage is estimated from $100$ Monte Carlo replications per~$n$.
Variance ratios are reported as 
$\mathrm{tr}(\hat V_{\mathrm{SIVI}})/
 \mathrm{tr}(\hat V_{\mathrm{Laplace}})$.
Posterior mean errors and coverage bands use common random-number seeds 
across values of~$n$ for comparability.


\end{document}